\documentclass[12pt]{amsart}

\usepackage[export]{adjustbox}  %For\adjustbox{\valign=c}{}
\usepackage{amsmath}
\usepackage{amssymb}
\usepackage{amsthm}
\usepackage[title]{appendix}
\usepackage{array}
\usepackage[backend=bibtex,giveninits=true,style=alphabetic,sorting=nyt,doi=false,isbn=false,url=false,eprint=true]{biblatex}
\usepackage{blindtext}
\usepackage{booktabs} %For \midrule
\usepackage{caption}
\usepackage{color}
\usepackage{changepage}
\usepackage{csquotes}
\usepackage{enumitem}
\usepackage{float}
\usepackage[margin=1in]{geometry}
\usepackage{graphicx}
\usepackage[hidelinks,linktoc=all]{hyperref}
\usepackage{nicefrac}
\usepackage{parskip}
\usepackage{rotating}
\usepackage{stackengine}
\usepackage{stmaryrd}   %For \llbracket and \rrbracket
\usepackage{thmtools}
\usepackage{tikz}
\usetikzlibrary{cd}
\usepackage{titlesec}
\usepackage{mathtools}
\usepackage{multicol} %For multicols environment
\usepackage{verbatimbox}

\usepackage{tocloft} % for changing layout of table of contents
\usepackage{lipsum}
\usepackage{todonotes}
\counterwithout{figure}{section}
\counterwithout{table}{section}
\titleformat{\section}{\normalfont\normalfont\bfseries}{\thesection}{1em}{}
\titleformat{\subsection}{\normalfont\normalfont\bfseries}{\thesubsection}{1em}{}
\titleformat{\paragraph}{\normalfont\normalfont\bfseries}{\thesection}{1em}{}
\theoremstyle{plain}
\newtheorem{theorem}{Theorem}
\newtheorem{lemma}[theorem]{Lemma}
\newtheorem{corollary}[theorem]{Corollary}
\newtheorem{proposition}[theorem]{Proposition}
\newtheorem{conjecture}[theorem]{Conjecture}
\theoremstyle{definition}
\newtheorem{definition}[theorem]{Definition}
\newtheorem{remark}[theorem]{Remark}

\newcommand{\defeq}{\coloneqq}

\newcommand{\NN}{{\mathbb{N}}}
\newcommand{\ZZ}{{\mathbb{Z}}}

\newcommand{\RR}{{\mathbb{R}}}

\newcommand{\col}{{\text{Col}}}
\newcommand{\mcol}{{\text{MCol}}}
\newcommand{\includeFig}[1]{\adjustbox{valign = c}{\includegraphics[scale=1]{Figures/#1.pdf}}}
\newcommand{\includeSymb}[1]{\,\adjustbox{valign = c}{\includegraphics[scale=0.2]{FigureAssembly/Symbols/Symbol-#1/symbol.pdf}}\,}
\usepackage{nicematrix}
\NiceMatrixOptions
  {
    custom-line = 
     {
       letter = I , % for the vertical rules
       command = boldhline , % the horizontal rules 
       total-width = 1 pt , 
       tikz = { line width = 1 pt } 
     }
  }
\title[A module structure on odd Khovanov homology]{A module structure on odd Khovanov homology and the odd invariant for ribbon 2-knots}
\author{Jacob Migdail}
\address{Washington and Lee University; Mathematics Department; Chavis Hall; Lexington, VA 24450}
\email{jmigdail-smith@wlu.edu}

\author{Stephan Wehrli}
\address{Syracuse University; Department of Mathematics; 215 Carnegie; Syracuse, NY 13244}
\email{smwehrli@syr.edu}
\date{\today}
\bibliography{bib}
\begin{document}

\begin{abstract} 
We prove that the reduced odd Khovanov homology of a link $L$ is naturally a module
over the exterior algebra of the first homology of the link's branched double-cover.
We then describe this module structure more geometrically and related it to the odd Khovanov maps induced by link cobordisms.
As an application, we will give a combinatorial proof of a recent result of Spyropoulos-Vidyarthi-Zhang
\cite[Theorem~1.2]{SVZ2026}
about the odd invariant for $2$-knots 
in the special case where the $2$-knot is a ribbon $2$-knot.
Additionally, we will show that Levine-Zemke's main result from \cite{LZ2019}
remains true for odd Khovanov homology with rational coefficients
and with coefficients in $\mathbb{Z}_{2^k}$.
\end{abstract}

\maketitle

\section{Introduction}\label{s:intro}

Odd Khovanov homology was introduced by Ozsv{\'a}th-Rasmussen-Szab{\'o}
\cite{ORS2007} as an alternative categorification of the Jones polynomial
\cite{Jo1987}.
When coefficients are taken in the two-element field $\mathbb{F}_2$, odd Khovanov homology
agrees with the original \enquote{even} Khovanov homology from \cite{Kh1999}.
On the other hand, the two theories are
substantially different
over integer (or rational) coefficients,
with neither of them being a refinement of the other.

The construction of odd Khovanov homology was originally
motivated by an attempt to lift the spectral sequence from \cite{OS2005}
to integer coefficients. In particular, it was conjectured in \cite{ORS2007}
that there is a spectral sequence starting with the reduced
odd Khovanov homology $\overline{OKh}(L)$, and converging
to the Heegaard Floer homology $\widehat{HF}(-\Sigma(L))$,
where $\Sigma(L)$ denotes the branched double-cover of $S^3$,
branched along the link $L\subset S^3$.
While such a spectral sequence has not yet been established over integer coefficients,
Daemi in \cite{Dae2015} and Scaduto in \cite{Sca2015} constructed spectral sequences
from $\overline{OKh}(L)$ to Floer theories associated with $\Sigma(L)$.
Branched double-covers also appeared more implicitly in the context
of Bloom's proof of the invariance of odd Khovanov homology
under knot mutation \cite{Blo2010}.

In the present paper, we will describe a new connection
between odd Khovanov homology and branched double-covers,
in the form of the following theorem:

\begin{restatable}{theorem}{TheoremModule}\label{thm:module}
The reduced odd Khovanov homology of a link $L\subset\mathbb{R}^3$ is a module over the exterior algebra $\Lambda^*H_1(\Sigma(L);\mathbb{Z})$.
\end{restatable}

The module structure from this theorem will be defined combinatorially by using maps induced by
dots placed on the edges of a link diagram. When the link $L$ is odd-homologically thin over $\mathbb{Z}$, in the
sense of \cite{Sh2011}, these maps are zero for grading reasons, and thus
the module action from Theorem~\ref{thm:module} is trivial for such links.
On the other hand, we will see that if $L$ is odd-homologically thick over $\mathbb{Z}$,
then $\Lambda^*H_1(\Sigma(L);\mathbb{Z})$ can act nontrivially, even if the link $L$ is a knot.

By \cite{OS2005}, the nonreduced odd Khovanov homology of a link $L$
can be identified with the reduced odd Khovanov homology of the disjoint union $L\cup U$,
where $U$ is an unknot. Theorem~\ref{thm:module} therefore shows:

\begin{restatable}{theorem}{TheoremModuleNonRed}\label{thm:modulenonred}
The odd Khovanov homology of a link $L\subset\mathbb{R}^3$
is a module over the exterior algebra $\Lambda^*H_1(\Sigma(L\cup U);\mathbb{Z})$.
\end{restatable}

While we have formulated this theorem for integer coefficients, its proof remains
valid over any commutative unital ring $\Bbbk$.
In particular, we obtain that the $\mathbb{F}_2$-Khovanov homology of~a~link $L$
carries an action of the exterior algebra of $H_1(\Sigma(L\cup U);\mathbb{F}_2)\cong(\mathbb{F}_2)^{\ell}$,
where $\ell$ denotes the number of link components. This is precisely the well-known action of
\[
\mathbb{F}_2[x_1,\ldots,x_\ell]/(x_1^2=\ldots=x_\ell^2=0),
\]
which was studied, e.g., in \cite{HN2013}. Similarly,
we obtain that the rational odd Khovanov homology of a nonempty link $L$
carries an action of an exterior algebra in \[\dim H_1(\Sigma(L\cup U);\mathbb{Q})=\beta(L)+1\leq\ell\] variables,
where $\beta(L)$ denotes the first Betti number of $\Sigma(L)$ (which is also equal
to the nullity of the symmetrized Seifert pairing).

Unlike even Khovanov homology, odd Khovanov homology comes in two types, called type X and type Y.
It is known \cite{Bei2012,SSS2020,MW2024} that
odd Khovanov homologies of the two types are isomorphic,
where the isomorphism is canonical up to an overall sign
\cite[Theorem~2]{MW2024} 
(this isomorphism is not directly related to the proof of \cite[Lemma~2.4]{ORS2007}, which was incomplete).
We will show:

\begin{restatable}{theorem}{TheoremTypes}\label{thm:types}
The canonical isomorphism between the
odd Khovanov homologies of type X and type Y intertwines the actions
of $\Lambda^*H_1(\Sigma(L\cup U);\mathbb{Z})$.
\end{restatable}

To prove this theorem, we will describe the module actions from Theorems~\ref{thm:module} and \ref{thm:modulenonred} 
more explicitly in terms of embedded arcs in $\mathbb{R}^3$. Specifically, we will consider an oriented embedded arc $\alpha\subset\mathbb{R}^3$
which meets the given link $L$ along the boundary $\partial\alpha$ and at no other points (see Figure~\ref{fig:arcalpha}).
The full preimage of the arc $\alpha$ in $\Sigma(L)$ is an unoriented simple closed curve $\widehat{\alpha}\subset\Sigma(L)$,
and we will consider two oriented versions of this curve, denoted $\widehat{\alpha}_X$ and $\widehat{\alpha}_Y$ (see Subsection~\ref{subs:arcs} for details).
Let $a$, $b$, $a_1,\ldots,a_r$, and  $b_1,\ldots,b_s$ be the edges of the link diagram shown in Figure~\ref{fig:arcalpha}.
Then the module actions of $[\widehat{\alpha}_X]=\pm[\widehat{\alpha}_Y]\in H_1(\Sigma(L);\mathbb{Z})$ can be described as follows:

\begin{figure}[h]
\includeFig{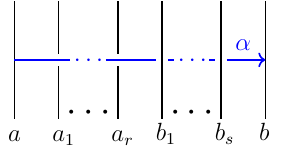}
\caption{The oriented arc $\alpha\subset\mathbb{R}^3$. The vertical lines represent strands of the link $L\subset\mathbb{R}^3$.
In particular, $a_1,\ldots,a_r$ are the edges of the link diagram that $\alpha$
overcrosses, and $b_1,\ldots,b_s$ are the edges that $\alpha$ undercrosses.
In the above picture, the $a_j$ precede the $b_j$, but this is not required in general.
In general, the arc $\alpha$ is allowed
to have self-crossings, but they will not enter our calculations.}\label{fig:arcalpha}
\end{figure}

\begin{restatable}{theorem}{TheoremArcAction}\label{thm:arcaction}
On the reduced type X odd Khovanov homology, $[\widehat{\alpha}_X]$ acts by
\[
-x_a+2x_{a_1}-2x_{a_2}+\ldots+(-1)^{r-1}2x_{a_r}+(-1)^rx_b,
\]
and on the reduced type Y odd Khovanov homology, $[\widehat{\alpha}_Y]$ acts by
\[
-x_a+2x_{b_1}-2x_{b_2}+\ldots+(-1)^{s-1}2x_{b_s}+(-1)^sx_b,
\]
where $x_e$ denotes the dot map assigned to a dot placed on the edge $e$ of the link diagram.
\end{restatable}

It was shown in \cite{MW2024,Spy2025} that every smooth link cobordism $F\subset\mathbb{R}^3\times I$
induces a map $OKh(F)$ on odd Khovanov homology, which is well-defined up to an overall sign. This map is related
to Theorem~\ref{thm:arcaction} in the following way. Consider an embedded arc $\alpha\subset\mathbb{R}^3$ as in Figure~\ref{fig:arcalpha}, and let $h\subset\mathbb{R}^3\times I$ be a $3$-dimensional thickening
of the arc $\alpha\times\{1/2\}$ in $\mathbb{R}^3\times I$. We can interpret $h$ as a $3$-dimensional $1$-handle, and by performing surgery along this $1$-handle,
we can attach a horizontal tube $t\subset\partial h$ to the identity link cobordism $\operatorname{id}_L:=L\times I$.
Call the resulting link cobordism $F_\alpha$. A broken surface diagram of $F_\alpha$ is shown in Figure~\ref{fig:tubecobordism}.

\begin{figure}[h]
\includeFig{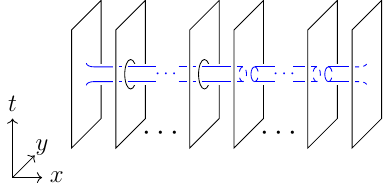}
\caption{Broken surface diagram of the link cobordism $F_\alpha$
for $\alpha$ as in Figure~\ref{fig:arcalpha}.
In this picture, we can assume that the shown portions of
$F_\alpha$ live in the hyperplane $z=0$,
and the unshown portions live in $z<0$.}\label{fig:tubecobordism}
\end{figure}

The odd Khovanov map induced by $F_\alpha$ can now be described as follows:

\begin{restatable}{theorem}{TheoremTubeAction}\label{thm:tubeaction}
The map $OKh(F_\alpha)$ coincides with the action of $\pm[\widehat{\alpha}]$.
\end{restatable}

To prove this theorem, we will introduce a category $\mathcal{C}\mathit{ob}^4_\Lambda$ of decorated link cobordisms.
The morphisms in this category are pairs $(F,c)$, where $F\subset\mathbb{R}^3\times I$ is a smooth link cobordism,
and $c$ is an element of the exterior algebra $\Lambda^*H_1(\Sigma(F\cup\operatorname{id}_U);\mathbb{Z})$. Such pairs $(F,c)$ will be
considered up to an equivalence relation, which we will describe in Subsection~\ref{subs:decorated}. We have:

\begin{restatable}{theorem}{TheoremFunctor}\label{thm:functor}
Up to sign, odd Khovanov homology extends to a functor on $\mathcal{C}\mathit{ob}^4_\Lambda$.
\end{restatable}

In Subsection~\ref{subs:dottedlinkcobordisms}, 
we will see that $\mathcal{C}\mathit{ob}^4_\Lambda$ is equivalent to a category
$\mathcal{C}\mathit{ob}^4_\bullet$ of dotted link cobordisms. In the latter category,
link cobordisms are allowed to be decorated by finitely many distinct dots, which are
placed on the facets of their broken surface diagrams. In general, dots
are not allowed to transition freely from one facet to another. However, we will see
that if $S$ is a broken surface diagram of a $2$-knot $F\subset\mathbb{R}^3\times I$, 
and $d$ is a single dot placed on one of the facets of $S$, then the induced
odd Khovanov map $OKh(S,d)$ is independent of the placement of the dot $d$.
In particular, this map is given by multiplication by $\pm n(F)\in\mathbb{Z}$ for a
positive integer $n(F)\geq 0$.
By results from \cite{Ra2005,Ta2005},
the modulo $2$ reduction of the $n(F)$ is always nonzero,
and hence the invariant $n(F)$ is always an odd number. We can further prove:

\begin{restatable}{theorem}{TheoremRibbonKnot}\label{thm:ribbonknot}
If $F$ is a ribbon 2-knot, then $n(F)$ agrees with $|H_1(\Sigma(F);\mathbb{Z})|$.
\end{restatable}

Our proof of this theorem was already announced in our paper \cite{MW2024} and uses the module structure on odd Khovanov homology.
As such, our proof is purely combinatorial. In \cite{MW2024}, the authors also conjectured that the
conclusion of Theorem~\ref{thm:ribbonknot} holds true for any smooth $2$-knot $F$.
Recently, Spyropoulos-Vidyarthi-Zhang \cite{SVZ2026} found an analytic proof of this conjecture, based on Daemi's spectral sequence
from reduced odd Khovanov homology to
oriented plane Flower homology \cite{Dae2015}.
In our upcoming paper \cite{MW2026b}, we will give a different and purely combinatorial proof of the same result for the special case
where $F$ is an even-twist spun knot.
In the present paper, we will further use a modification of our proof of Theorem~\ref{thm:ribbonknot}
to show a partial odd analog of Levine-Zemke's main result from \cite{LZ2019}:

\begin{restatable}{theorem}{TheoremRibbonConcordance}\label{thm:ribbonconcordance}
Any ribbon concordance induces an injective map on odd Khovaonv homology with rational coefficients.
The same is true for coefficients in $\mathbb{Z}_{2^k}$.
\end{restatable}

By a ribbon concordance, we here mean a smooth oriented genus zero link cobordism $C\subset\mathbb{R}^3\times I$ for
which the projection onto $I$ has no local maxima in $C\setminus\partial C$.
If $\partial_-\Sigma(C)$ denotes the lower boundary of the branched double-cover $\Sigma(C)$, and
$\overline{C}$ denotes the reflection of $C$ along $\mathbb{R}^3\times\{1/2\}$, then
we will actually show that the map $OKh(\overline{C}\circ C)$
is given by multiplication by $\pm a$ for
\[
a=|H_1(\Sigma(C),\partial_-\Sigma(C);\ZZ)|.
\]
Because of the main result from \cite{LZ2019}, the reduction of $a$ modulo $2$ is always nonzero,
and hence $a$ is itself nonzero, and always odd. In particular, $a$ is invertible over $\mathbb{Q}$
and over $\mathbb{Z}_{2^k}$, and this implies Theorem~\ref{thm:ribbonconcordance}.
In the special
case where $a=1$, the conclusion of the theorem also holds over integer coefficients.

Finally, we remark that a variant of our Theorem~\ref{thm:module} first appeared in the first
author's PhD Thesis \cite{Mig24} (see also \cite{WehrliTalk}, where this theorem was discussed in its present version). Recently, Ebert-Schelstraete~\cite{ES2025} independently
discovered an action of the Lie superalgebra $\mathfrak{gl}_{1|1}$ on odd Khovanov homology, which turns out to be related to
our module action from Theorem~\ref{thm:module}.

\noindent
{\bf Conventions.}
Throughout this paper, 
we will assume that link cobordisms are smooth and oriented.
Given a link $L\subset\mathbb{R}^3$, we will denote by $\Sigma(L)$ the branched double-cover of $S^3$,
branched along the link $L$, and we will assume that $\Sigma(L)$ is equipped with a distinguished preimage, $\widetilde{x}_0\in\Sigma(L)$,
of the basepoint $x_0:=\infty\in S^3$. This will ensure that $\Sigma(L)$ is unique up to unique isomorphism
of based branched covering spaces. Given a link cobordism $F\subset\mathbb{R}^3\times I$, we will
similarly denote by $\Sigma(F)$ the branched double-cover of $S^3\times I$, branched along $F$, and we will assume that
$\Sigma(F)$ is equipped with a distinguished preimage of $(x_0,1/2)$. Note that this preimage also determines a distinguished
preimage of the segment $\{x_0\}\times I$, and hence of the basepoints $(x_0,0)$ and $(x_0,1)$.
While almost all of our results will be formulated over integer coefficients, our results that hold over the integers
will remain valid over any commutative unital ring $\Bbbk$.

\noindent
{\bf Organization.}
The remainder of this paper is organized in the following manner.  
In \textbf{Section~\ref{s:preliminaries}}, we review Putyra's
odd chronological cobordism category,
and we recall the definitions of odd Khovanov homology
and of the coloring module of a link diagram.
In \textbf{Section~\ref{s:module}}, we
introduce dot chain maps on the odd Khovanov bracket and, in Subsection~\ref{subs:moduledef},
prove
Theorems~\ref{thm:module} and \ref{thm:modulenonred} about the
module structure on odd Khovanov homology.
In the next subsection, we prove Theorem~\ref{thm:arcaction}
about the description of the module structure in terms of
embedded arcs, and we then use this description
to prove Theorem~\ref{thm:types}
about the equivalence of the module structures in types X and Y.
In the last two subsections of Section~\ref{s:module},
we consider two examples of links where the module structure from Theorem~\ref{thm:module}
is nontrivial: the $2$-component unlink, and the $(3,3,-3)$-pretzel knot.
In \textbf{Section~\ref{s:cobordism}}, we define the coloring module of a link cobordism,
develop the decorated link cobordism category \mbox{$\mathcal{C}\mathit{ob}^4_\Lambda$} and
the dotted link cobordism category $\mathcal{C}\mathit{ob}^4_\bullet$,
and, in Subsections~\ref{subs:oddmaps} and \ref{subs:isotopyinvariance}, prove Theorem~\ref{thm:functor}
about the extension of the odd Khovanov functor to \mbox{$\mathcal{C}\mathit{ob}^4_\Lambda$}.
In Subsection~\ref{subs:tubes}, we will then use this extension to prove
Theorem~\ref{thm:arcaction} about the correspondence
between embedded tubes and arcs.
In \textbf{Section~\ref{s:okjnumbers}}, we introduce 
the $2$-knot invariant $n(F)$ and prove Theorems~\ref{thm:ribbonknot} and \ref{thm:ribbonconcordance}.
Finally, in \textbf{Section~\ref{s:oddend}}, we revisit the computations that prove that odd Khovanov homology is not functorial in $S^3\times I$.
We also describe two modifications of
our constructions that will allow us to eliminate signs that occur in the formula
for the composition in \mbox{$\mathcal{C}\mathit{ob}^4_\Lambda$}
and in naturality result for the module structure on odd Khovanov homology.

\noindent
{\bf Acknowledgments.}
The second author would like to thank Andy Manion, Jake Rasmussen,
Dean Spyropoulos, and Matt Stoffregen for valuable discussions.
He was also partially supported by
a grant from the Simons Foundation (\#632059 Stephan
Wehrli).

\section{Preliminaries}\label{s:preliminaries}

\subsection{Chronological Cobordisms}\label{subs:chronological}
Chronological cobordisms were introduced by Putyra in \cite{Pu2015} in the context of his
construction of a generalized Khovanov bracket.
In this paper, we will assume that chronological cobordisms are embedded, and
we will use the following definition from \cite{MW2024}:

\begin{definition}\label{def:chroncob}
 A chronological cobordism is a smooth properly embedded compact surface $S\subset\mathbb{R}^2\times I$
such that the following hold:
\begin{enumerate}
\item The projection $\mathbb{R}^2\times I\rightarrow I$ restricts to a separating Morse function on $S$.
\item The index 1 and 2 critical points of this Morse function are equipped with orientations of their
descending manifolds.
\end{enumerate}
\end{definition}
Some examples of chronological cobordisms are shown below, where the arrows
indicate the chosen orientations at the critical points:
\begin{equation*}
	\adjustbox{valign = c}{\includegraphics[scale=1.25]{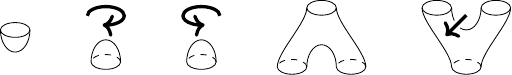}}
\end{equation*}
When no arrows are shown in a picture, we will assume by default that deaths are oriented clockwise, and that saddles are oriented to the front or to the right, whichever makes sense. We will often consider chronological cobordisms
up to the following types of isotopies:

\begin{definition}
A chronology-preserving isotopy of $S\subset\mathbb{R}^2\times I$ is
a smooth ambient isotopy which fixes $\partial S$, and which preserves the weak order on $\mathbb{R}^2\times I$
given by $(x,t)\leq (x',t')$ iff $t\leq t'$.
\end{definition}

Following \cite{Pu2015}, we will define the chronological degree of a chronological cobordism $S$ as the pair $(a,b)\in\ZZ\times\ZZ$ given by
\[
(a,b):=\bigl(\#\mbox{births}-\#\mbox{merges},\#\mbox{deaths}-\#\mbox{splits}\bigr),
\]
where we are using the every critical point in $S$ corresponds to a birth, a death, a merge, or
a split. Note that $(a,b)$ is preserved under arbitrary isotopies of $S$ rel boundary
because it can be described topologically as
\[
(a,b)=\bigl((\chi+D)/2,(\chi-D)/2\bigr),
\]
where $\chi$ is the Euler characteristic of the orientable surface $S$,
and $D$ is the difference between the number of top boundary components of $S$ and
the number of bottom boundary components of $S$.
In this paper, we will use two specializations of the chronological degree:
\begin{itemize}
\item the quantum degree $a+b=\chi$,
\item the superdegree $|S|\in\mathbb{Z}_2$, which is defined as the modulo $2$ reduction of $b$.
\end{itemize}
Note that birth and merge cobordismss satisfy $b=0$, and thus have even superdegree,
whereas death and split cobordisms satisfy $b=\pm 1$, and thus
have odd superdegree.

We can now define
the odd Putyra category \mbox{$o\:\!\mathcal{C}\mathit{ob}^3_{/\ell}$}  \cite{Pu2015},
which can be seen as an odd analog of Bar-Natan's category  $\mathcal{C}\mathit{ob}^3_{/\ell}$ from \cite{Bn2005}.
The objects of \mbox{$o\:\!\mathcal{C}\mathit{ob}^3_{/\ell}$} are closed $1$-manifolds embedded in the plane $\mathbb{R}^2$, and
the morphisms are given by formal $\ZZ$-linear combinations of chronological cobordisms. Such morphisms are considered
up to chronology-preserving isotopy, and up to the following relations:
\begin{enumerate}
\item Change-of-chronology relations.
\item Odd Bar-Natan relations.
\end{enumerate}

The odd Bar-Natan relations are shown below, 
where in the third relation, we are assuming that the critical points
are oriented in accordance with the default conventions mentioned after Definition~\ref{def:chroncob}:

{\noindent}\begin{minipage}{0.15\textwidth}
    \begin{equation*}
        \includeFig{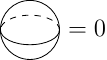}
    \end{equation*}
\end{minipage}%
\begin{minipage}{0.31\textwidth}
    \begin{equation*}
        \includeFig{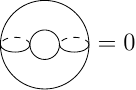}
    \end{equation*}
\end{minipage}%
\begin{minipage}{0.52\textwidth}
    \begin{equation*}
        \includeFig{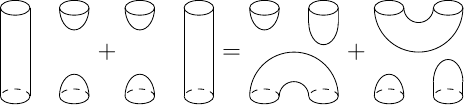}
    \end{equation*}
\end{minipage}

In particular, the second among these relations states that
a torus evaluates to $0$, contrasting with the even version of this relation from \cite{Bn2005}, which states that a torus evaluates to $2$.
Below are also some examples of change-of-chronology relations (for a complete list, see \cite{MW2024}):

{\noindent}\begin{minipage}{0.3\textwidth}
    \begin{equation*}
        \includeFig{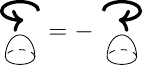}
    \end{equation*}
\end{minipage}%
\begin{minipage}{0.35\textwidth}
    \begin{equation*}
        \includeFig{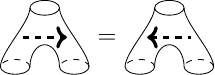}
    \end{equation*}
\end{minipage}%
\begin{minipage}{0.35\textwidth}
    \begin{equation*}
        \includeFig{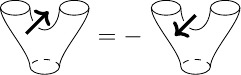}
    \end{equation*}
\end{minipage}

\[
        \includeFig{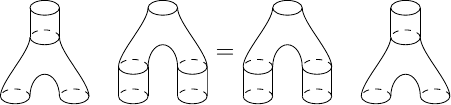}
\]
\[
        \includeFig{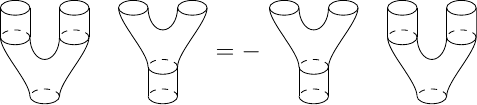}
\]
Note that the second change-of-chronology relation shown above implies that the orientation
at a merge saddle can be omitted.
One can check that all of the relations in \mbox{$o\:\!\mathcal{C}\mathit{ob}^3_{/\ell}$} are compatible with
the chronological grading. As a consequence, the morphism sets in the odd Putyra category carry both
a quantum grading and a supergrading. We can define a graded refinement of \mbox{$o\:\!\mathcal{C}\mathit{ob}^3_{/\ell}$} in which objects are
symbols of the form $C\{j\}\langle s\rangle$, where
\begin{itemize}
\item $C$ is an object of $o\:\!\mathcal{C}\mathit{ob}^3_{/\ell}$,
\item $\{j\}$ is a formal shift of the quantum grading by $j\in\mathbb{Z}$,
\item $\langle s\rangle$ is a formal shift of the supergrading by $s\in\mathbb{Z}_2$.
\end{itemize}
These formal grading shifts have the following significance: 
morphisms $S\colon C\{j\}\langle s\rangle\rightarrow C'\{j'\}\langle s'\rangle$ are no different
from morphisms $S\colon C\rightarrow C'$ in the original category \mbox{$o\:\!\mathcal{C}\mathit{ob}^3_{/\ell}$}, except that their quantum degrees are
raised by $j'-j$, and their superdegrees by $s'-s$.

\subsection{Dotted Chronological Cobordisms}\label{subs:dotchroncobs}

We will now describe the dotted odd Putyra category \mbox{ $o\:\!\mathcal{C}\mathit{ob}^3_{\bullet/\ell}$},
which is obtained from  \mbox{$o\:\!\mathcal{C}\mathit{ob}^3_{/\ell}$} by allowing chronological cobordisms to be decorated by at most finitely many distinct dots (cf. \cite[Sect.~11]{Pu2015}).
If $S$ is a chronological cobordism, then these dots
are viewed as marked points on the surface $S\setminus\partial S$.
The heights of these points are crucial, in that no dot is allowed to live at the same height as another dot, or at the
same height as a critical point.
A single dot is assumed to have chronological degree $(-1,-1)$,
and thus has quantum degree $-2$ and odd superdegree.
In particular, dots are assumed to satisfy the following relations, which imply that
dots commute vertically with cobordisms of even superdegree,
and anticommute with cobordisms of odd superdegree:

{\noindent}\begin{minipage}{0.5\textwidth}
    \begin{equation*}
    \includeFig{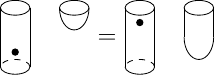}
    \end{equation*}
\end{minipage}%
\begin{minipage}{0.5\textwidth}
    \begin{equation*}
    \includeFig{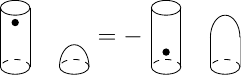}
    \end{equation*}
\end{minipage}

{\noindent}\begin{minipage}{0.5\textwidth}
    \begin{equation*}
        \includeFig{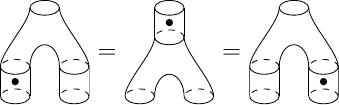}
    \end{equation*}
\end{minipage}%
\begin{minipage}{0.5\textwidth}
    \begin{equation*}
        \includeFig{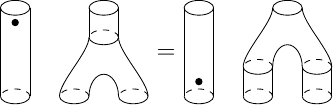}
    \end{equation*}
\end{minipage}

{\noindent}\begin{minipage}{0.5\textwidth}
    \begin{equation*}
        \includeFig{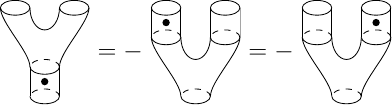}
    \end{equation*}
\end{minipage}%
\begin{minipage}{0.5\textwidth}
    \begin{equation*}
        \includeFig{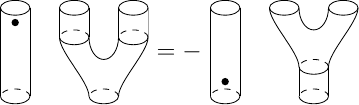}
    \end{equation*}
\end{minipage}

\begin{equation*}
	\includeFig{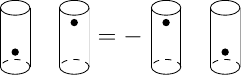}
\end{equation*}

Within a particular slice of time $\mathbb{R}^2\times\{t\}$, dots are allowed to slide freely around the component of $S$ they live on.
Dots are further assumed to satisfy following relations, where
deaths are oriented clockwise:

{\noindent}\begin{minipage}{0.333\textwidth}
    \begin{equation}\label{eqn:DottedSphere}
        \includeFig{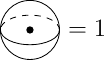}
    \end{equation}
\end{minipage}%
\begin{minipage}{0.333\textwidth}
    \begin{equation}\label{eqn:vneck}
        \includeFig{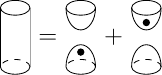}
    \end{equation}
\end{minipage}%
\begin{minipage}{0.333\textwidth}
    \begin{equation}
        \includeFig{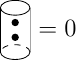}
    \end{equation}
\end{minipage}

Relation~\eqref{eqn:vneck} is known as the vertical neck-cutting relation.
By \cite[Lemma~26]{NW2024}, it also implies the following horizontal neck-cutting relation
    \begin{equation}\label{eqn:hneck}
        \includeFig{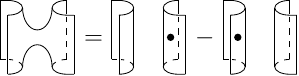}
    \end{equation}
where the upper saddle on the left-hand side is assumed to be oriented to the front
(and the orientation of the lower saddle does not matter, cf. \cite{NW2024}).
Since all relations that hold in \mbox{$o\:\!\mathcal{C}\mathit{ob}^3_{/\ell}$} are assumed
to hold in \mbox{$o\:\!\mathcal{C}\mathit{ob}^3_{\bullet/\ell}$} as well, there is an obvious functor
\[
o\:\!\mathcal{C}\mathit{ob}^3_{/\ell}\longrightarrow o\:\!\mathcal{C}\mathit{ob}^3_{\bullet/\ell}
\]
given by sending each undotted chronological cobordism to itself.
We note that this functor is not faithful because cobordisms of genus $g\geq 1$
are zero in \mbox{$o\:\!\mathcal{C}\mathit{ob}^3_{\bullet/\ell}$},
as a consequence of the above relations. 
In contrast, such cobordisms are nonzero in \mbox{$o\:\!\mathcal{C}\mathit{ob}^3_{/\ell}$}
unless they have closed components, although they become zero when multiplied by $2$.

\subsection{Definition of the Odd Khovanov Bracket}\label{subs:def}

The odd Khovanov bracket of a link diagram was introduced implicitly in \cite{Pu2015}
and has the form of an abstract chain complex.
In this subsection, we will briefly recall the definition of this chain complex.

Let $D$ be a link diagram with $n$ crossings.
Each crossing \includeSymb{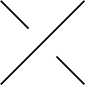} of $D$ has two possible resolutions, a $0$-resolution looking like \includeSymb{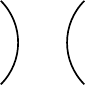},
and a $1$-resolution looking like \includeSymb{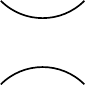} If we replace each crossing of $D$ by one
of the two resolutions, we obtain a crossingless link diagram, called a complete
 resolution of $D$.
After enumerating the crossings of $D$, we can index
the complete resolutions of $D$ by
sequences $\alpha=(\alpha_1,\ldots,\alpha_n)\in\{0,1\}^n$, where value of
$\alpha_i\in\{0,1\}$ specifies the resolution chosen at the $i$th crossing.
Let $D_\alpha$ denote the complete resolution of $D$ associated with the sequence $\alpha$.
We can arrange the $D_\alpha$ around the vertices of an $n$-dimensional hypercube,
called the resolution hypercube of $D$.
In this hypercube, two vertices $\alpha$ and $\alpha'$ are connected by an oriented edge $\xi\colon\alpha\rightarrow\alpha'$
if the sequences $\alpha$ and $\alpha'$ differ in a single entry, which is a $0$ in $\alpha$ and a $1$ in $\alpha'$.
For such $\alpha$ and $\alpha'$, the complete resolutions $D_\alpha=D_{\ldots 0\ldots}$ and $D_{\alpha'}=D_{\ldots 1\ldots}$ differ
at a single crossing of $D$. In particular, there is an obvious saddle cobordism
\[
d'_\xi=d'_{\ldots\star\ldots}\colon D_{\ldots 0\ldots}\longrightarrow D_{\ldots 1\ldots}
\]
between them, which is an identity cobordism at all places where $D_\alpha$ and $D_{\alpha'}$ coincide.
Now decorated each crossing of $D$ with an
arrow that connects the two arcs of its $0$-resolution.
These arrows induce orientations of the saddle points in the cobordisms $d'_\xi\subset\mathbb{R}^2\times I$,
and we can thus view the $d'_\xi$ as chronological cobordisms, or as morphisms in the category \mbox{$o\:\!\mathcal{C}\mathit{ob}^3_{/\ell}$}. In particular, this
allows us to view the resolution hypercube of $D$ as a diagram in \mbox{$o\:\!\mathcal{C}\mathit{ob}^3_{/\ell}$}.

Note that each $2$-dimensional face of the resolution hypercube
corresponds to a configuration of two distant saddles. Two examples
of such configurations are shown in Figure~\ref{fig:special}, where in each picture, the circle
represents the initial resolution, and the oriented arcs represent oriented saddles
(so that performing surgery along one of the arcs corresponds to performing
a saddle move). Because of the relations that hold in \mbox{$o\:\!\mathcal{C}\mathit{ob}^3_{/\ell}$},
all configurations of two distant saddles commute up to sign, but the two configurations
from Figure~\ref{fig:special} are special in that they simultaneously commute and anticommute. 
In fact, this is possible because in these two configurations the relevant compositions of saddles are annihilated
by $2$, and thus do not have well-defined signs.

To obtain a link invariant, one must artificially decide which of the two configurations from Figure~\ref{fig:special}
anticommutes. There are two different conventions for doing so, called type X and type Y. In type X, it is assumed that the left
configuration anticommutes and the right one commutes, and in type Y,
these assumptions are reversed.
\begin{figure}[H]
\includeFig{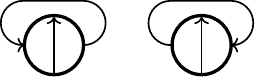}
\caption{The two special configurations, viewed as surgery diagrams in $S^2$.}\label{fig:special}
\end{figure}

Suppose we have fixed one of the two types.
We can then define a commutativity $2$-cochain
\[
\sigma\colon\{\mbox{faces}\}\longrightarrow\{\pm 1\}
\]
on the resolution hypercube of $D$ by sending each $2$-dimensional face of the hypercube
to $1$ or $-1$, depending on whether the face commutes or anticommutes.
It was shown in \cite{ORS2007} that the map
$\sigma$ provides a cellular $2$-cocycle \[\sigma\in Z^2(Q;\{\pm 1\}),\] on the hypercube $Q:=[0,1]^n$, equipped with
its usual cell structure, where $\{\pm 1\}$ is viewed as an abelian group with respect to multiplication.
Since $Q$ is contractible, this implies that there is a $1$-cochain
\[
\epsilon\in C^1(Q;\{\pm 1\})
\]
with coboundary
$\delta\epsilon=(-1)\sigma$.
Any such $\epsilon$ is called a sign assignment for $D$, where the sign assignment $\epsilon$ is said to have type X or type Y, depending
on the type that was used to define $\sigma$. After fixing a sign assignment,
we can make the faces of the resolution hypercube anticommute
by replacing each saddle cobordism $d'_\xi$
by the edge map
\[
d_\xi:=\epsilon_\xi d'_\xi
\]
where $\epsilon_\xi\in\{\pm 1\}$ is the sign that $\epsilon$ assigns
to the edge $\xi$. For a sequence $\alpha\in\{0,1\}^n$, let $\deg(\alpha):=\alpha_1+\ldots+\alpha_n$ and $\llbracket D\rrbracket_\alpha:=D_\alpha$.
Assume that the link diagram $D$ is oriented (in the usual sense), and let
$n_+$ and $n_-$ denote, respectively, the numbers of positive and negative crossings of $D$.
The $i$th chain group in the odd Khovanov bracket of $D$ is now defined as the direct sum
\[
\llbracket D\rrbracket^i:=\bigoplus_{i=\deg(\alpha)-n_-}\llbracket D\rrbracket_\alpha,
\]
where the sum is taken in the additive closure of $o\:\!\mathcal{C}\mathit{ob}^3_{/\ell}$.
The $i$th differential $d^i\colon\llbracket D\rrbracket^i\rightarrow\llbracket D\rrbracket^{i+1}$ is given by the sum
of all edge maps $d_\xi$ that correspond to edges $\xi\colon\alpha\rightarrow\alpha'$ which start at a vertex
$\alpha\in\{0,1\}^n$ with $\deg(\alpha)-n_-=i$. By construction, $\llbracket D\rrbracket:=\{\llbracket D\rrbracket^i,d^i\}_{i\in\mathbb{Z}}$
is a chain complex, and it was shown in \cite{ORS2007,Pu2015} that the homotopy type of this chain complex is a link invariant.

The odd Khovanov bracket $\llbracket D\rrbracket$ has a graded refinement, in which
all differentials are homogeneous of quantum degree and superdegree zero.
To obtain this graded refinement, one replaces the category \mbox{$o\:\!\mathcal{C}\mathit{ob}^3_{/\ell}$} by its graded version,
and one redefines
$\llbracket D\rrbracket_\alpha$ as
\begin{equation}\label{eqn:refinedDalpha}
\llbracket D\rrbracket_\alpha:=D_\alpha\{j(\alpha)\}\langle s(\alpha)\rangle
\end{equation}
where $\{j(\alpha)\}$ and $\langle s(\alpha)\rangle$ denote formal shifts 
of the quantum grading and the supergrading. Explicitly,
$j(\alpha):=\deg(\alpha)+n_+-2n_-=i+n_+-n_-$, and $s(\alpha)$ is defined as the modulo $2$ reduction of
\begin{equation}\label{eqn:Smap}
S(D,\alpha):=\frac{1}{2}\bigl(c(D,\alpha)+\deg(\alpha)+n_+-2n_--|D|\bigr),
\end{equation}
where $c(D,\alpha)$ denotes the number of components of $D_\alpha$, and $|D|$ denotes the number of components
of the link represented by $D$. One can check that $S(D,\alpha)$ is always an integer.
Moreover, $S(D,\alpha)$ increases by $1$ along edges $\xi\colon\alpha\rightarrow\alpha'$ that correspond to split saddles,
and it remains unchanged along edges that correspond to merge saddles \cite{MW2024}.

\subsection{Odd Khovanov homology}\label{subs:oddKh}

It was shown in \cite{ORS2007,Pu2015} that there is a chronological TQFT functor
\[\mathcal{F}_{odd}\colon o\:\!\mathcal{C}\mathit{ob}^3_{/\ell}\longrightarrow\ZZ\mbox{-mod},
\]
which takes the odd Khovanov bracket $\llbracket D\rrbracket$ to the odd Khovanov complex from \cite{ORS2007}. Thus,
the odd Khovanov homology of $D$ can be defined as the homology
\[
OKh(D):=H(\mathcal{F}_{odd}(\llbracket D\rrbracket)).
\]
The functor $\mathcal{F}_{odd}$ is defined as follows.
To a closed $1$-manifold $C\subset\mathbb{R}^2$ with connected components $a_1,\ldots,a_k$,
it assigns the exterior algebra
\[
\mathcal{F}_{odd}(C):=\Lambda^*[a_1,\ldots,a_k]=\Lambda^*V(C),
\]
where 
$V(C)$ denotes the free abelian group formally spanned by the $a_i$.
If $S\colon C\rightarrow C'$ is a merge saddle that merges two components $a_i$ and $a_j$ into a single component,
then $\mathcal{F}_{odd}(S)$ is the obvious quotient map
\[\mathcal{F}_{odd}(C)\longrightarrow\mathcal{F}_{odd}(C)/(a_i-a_j)\cong\mathcal{F}_{odd}(C').\]
On the other hand, if $S\colon C\rightarrow C'$ is a split saddle that divides a component of $C$
into two components $a'_i$ and $a'_j$, then $\mathcal{F}_{odd}(S)$ is the map
\[
\mathcal{F}_{odd}(C)\cong
\mathcal{F}_{odd}(C')/(a'_i-a'_j)\xrightarrow{(a'_i-a'_j)\wedge} \mathcal{F}_{odd}(C'),
\]
where we are assuming that $a'_i$ and $a'_j$ are located, respectively, to the left and to the right of the arrow that specifies
the orientation at the saddle point in $S$.\footnote{This is consistent with the convention used in \cite[Sect.~9]{Pu2015} but
inconsistent with the convention used in the preprint version of \cite{ORS2007}. 
The two conventions are largely interchangeable because one can go from one to the other
by reversing all arrows at the crossings of the link diagram $D$. This does not
change the commutativity behavior of the faces of the resolution hypercube,
and it also does not affect the case-by-case analysis in the proofs of Lemmas~\ref{lem:typeXslide} and \ref{lem:typeYslide} below.}
Finally, if $S$ is a birth cobordism $S\colon C\rightarrow C'$, then $\mathcal{F}_{odd}(S)$ is
the obvious inclusion $\mathcal{F}_{odd}(C)\rightarrow\mathcal{F}_{odd}(C')$, and if $S$ is
a death cobordism, then $\mathcal{F}_{odd}(S)$ is the contraction from the left with the component that
gets annihilated. The functor $\mathcal{F}_{odd}$ factors through a functor
\[
o\:\!\mathcal{C}\mathit{ob}^3_{\bullet/\ell}\longrightarrow\ZZ\mbox{-mod}.
\]
To an identity cobordism $C\times I$ with a dot placed on a component $a_i\times I$, the
latter functor assigns the map given by wedge multiplication from the left by $a_i$.

\subsection{Reduced Odd Khovanov Homology}\label{subs:reducedfirst}

If $C$ is nonempty, then the subalgebra of $\mathcal{F}_{odd}(C)=\Lambda^*[a_1,\ldots,a_k]$ generated by all differences $a_i-a_j$
is a proper submodule $\overline{\mathcal{F}}_{odd}(C)\subset\mathcal{F}_{odd}(C)$.
One can check that the assignment $C\mapsto\overline{\mathcal{F}}_{odd}(C)$
extends to a functor, where the map $\overline{\mathcal{F}}_{odd}(S)$
is defined as the restriction of $\mathcal{F}_{odd}(S)$. Using this functor, we can define
the reduced odd Khovanov homology of a nonempty link diagram $D$ by
\[
\overline{OKh}(D):=H(\overline{\mathcal{F}}_{odd}(\llbracket D\rrbracket)).
\]
Note that $\overline{\mathcal{F}}_{odd}$ does
not extend to dotted chronological cobordisms
because wedge multiplication by a component $a_i$ does
not preserve the reduced submodule $\overline{\mathcal{F}}_{odd}(C)$.
We will revisit this issue in Subsection~\ref{subs:reduced}.

\subsection{Gradings}\label{subs:gradings}

If $C$ has $k$ components $a_1,\ldots,a_k$, then
we can equip the module $\mathcal{F}_{odd}(C)$ with a quantum grading
and a supergrading by declaring that 
 an element $a_{i_1}\wedge\cdots\wedge a_{i_\ell}\in \mathcal{F}_{odd}(C)$
has quantum degree $k-2\ell$ and superdegree $\ell$ modulo $2$.
The functor $\mathcal{F}_{odd}$ can be extended to the graded version of \mbox{$o\:\!\mathcal{C}\mathit{ob}^3_{/\ell}$}
by sending an object $C\{j\}\langle s\rangle$ to the module $\mathcal{F}_{odd}(C)\{j\}\langle s\rangle$,
which is obtained from $\mathcal{F}_{odd}(C)$ by raising the quantum grading by $j\in\ZZ$ and the supergrading by $s\in\ZZ_2$.
The extended functor then gives rise to a graded version of the odd Khovanov complex.
One can check that,
in the graded odd Khovanov complex, the superdegree of a homogeneous generator $g$ coincides
with the modulo $2$ reduction
of $(j-|D|)/2$, where $j$ denotes the quantum degree of $g$, and $|D|$ denotes
the number of link components.

In the reduced subcomplex
$\overline{\mathcal{F}}_{odd}(\llbracket D\rrbracket)\subset\mathcal{F}_{odd}(\llbracket D\rrbracket)$,
we will assume that the quantum grading has been lowered by $1$, so that the reduced odd Khovanov homology
of an unknot is supported in quantum degree zero.

\subsection{Coloring Modules of Link Diagrams}\label{subs:ColD}

Throughout this subsection, $D$ will be a link diagram representing a link $L\subset\mathbb{R}^2$. By an overstrand of $D$, we will mean
a maximal subarc of $L$ that does not undercross any strands in $D$.
A trivial circle in $D$ that has no crossings is also considered to be an overstrand.

The following definition is based on \cite{KH2025}:

\begin{definition}\label{def:ColD}
The coloring module\footnote{The coloring module is a universal coloring group, in the sense that
Fox's $n$-colorings correspond to homomorphisms from $\col(D)$ to $\ZZ/n\ZZ$.}
$\col(D)$ is an abelian group which is generated
by the overstrands of $D$, and which has a relation of the form
$2b-a-c=0$ for each crossing of $D$, where $a,b,c$ denote the overstrands that meet
at that crossing. More precisely,
\[
\col(D):=\left\{\mbox{overstrands of $D$}\,\left|\,\includeFig{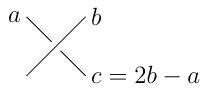}\right.\right\}
\]
\end{definition}
The coloring module admits an obvious group homomorphism
\[
\phi\colon\col(D)\longrightarrow\ZZ
\]
given by sending each overstrand to $1$. The kernel of this homomorphism
is called the reduced coloring module, and denoted $\col(D)^{red}$.
Note that the reduced coloring module is generated by differences of overstrands.

The following is well known (see, e.g., \cite{Pr1998}):

\begin{lemma}\label{lem:ColDH}
$\col(D)^{red}$ is canonically isomorphic to $H_1\left(\Sigma(L);\ZZ\right)$.
\end{lemma}

The isomorphism in this lemma is canonical under the assumption that the branched double-cover, $\Sigma(L)$, is equipped with
a distinguished preimage of the basepoint $x_0:=\infty\in S^3$.
We will briefly prove the lemma because the proof will matter later in the paper.

\begin{proof} Let
\[E=S^3\setminus\operatorname{nbhd}(L)\]
be the link complement and $E_2$ be 
its usual double-cover, so that a loop $\gamma\subset E$ lifts to a loop in $E_2$
iff $\gamma$ has even linking number with $L$.
Starting with the Wirtinger presentation for $\pi_1(E,x_0)$, we can 
can construct a handle decomposition for $E$ with
\begin{itemize}
\item
a $0$-handle corresponding to the basepoint $x_0$,
\item
a $1$-handle for each Wirtinger generator,
\item
a $2$-handle for each Wirtinger relation,
\item
a single $3$-handle.
\end{itemize}
By taking all preimages of handles in this handle decomposition, we also
obtain a handle decomposition for $E_2$.
Note that the $0$-handles in this handle decomposition for $E_2$ correspond to the points
$\widetilde{x}_0$ and $t\widetilde{x}_0$, where $\widetilde{x}_0$ is the distinguished preimage of $x_0$,
and $t$ denotes the covering transformation.
Given a Wirtinger generator $a\in\pi_1(E,x_0)$, we will denote by $\widetilde{a}\subset E_2$ its lift starting at $\widetilde{x}_0$.
For simplicity, we will also write $\widetilde{a}$ 
for the corresponding $1$-handle in $E_2$.
Note that $\widetilde{a}$ connects $\widetilde{x}_0$ to $t\widetilde{x}_0$,
while $t\widetilde{a}$ connects $t\widetilde{x}_0$ to $\widetilde{x}_0$
(see Figure~\ref{fig:lifts}).

\begin{figure}[H]
\includeFig{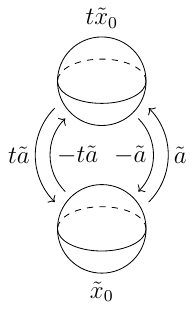}
\caption{The lifts of a Wirtinger generator $a$
and its inverse $a^{-1}$.
}\label{fig:lifts}
\end{figure}

We can extend the handle decomposition of $E_2$ to a handle decomposition for $\Sigma(L)\supset E_2$
by attaching an additional $2$-handle along each loop $\widetilde{a}+t\widetilde{a}$,
and then filling in additional $3$-handles.
The cellular chain group $C_1(\Sigma(L))$ 
associated with the resulting handle decomposition of $\Sigma(L)$ is generated
by all $1$-handles $\widetilde{a}$ and $t\widetilde{a}$, for all Wirtinger generators $a$.
Moreover, each lift of a $2$-handle coming from a crossing of $D$
gives rise to a relation in the quotient group $C_1(\Sigma(L))/\operatorname{im}(\partial_2)$.
Explicitly, this relation is shown in \eqref{eqn:doublecoverrelation},
where the $1$-chain $\widetilde{b}+t\widetilde{a}-t\widetilde{b}=t\widetilde{b}-\widetilde{a}-\widetilde{b}$
arises as a lift
of either $b*a*b^{-1}$ or $b^{-1}*a*b$, depending on whether the crossing
is positive or negative.

\begin{equation}\label{eqn:doublecoverrelation}
\includeFig{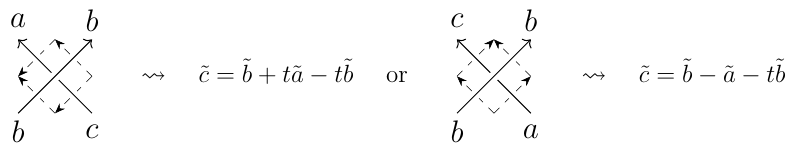}
\end{equation}
The additional $2$-handles in $\Sigma(L)$
provide
the relations
$\widetilde{a}+t\widetilde{a}=0$, and we can therefore
replace each generator $t\widetilde{a}$ by $-\widetilde{a}$. This turns
the relations from \eqref{eqn:doublecoverrelation} into the relation
\[
\widetilde{c}=2\widetilde{b}-\widetilde{a},
\]
which corresponds to the relation that holds in $\col(D)$.
In conclusion, we see that there is an identification
\begin{equation}\label{eqn:ColDH}
\col(D)=C_1(\Sigma(L))/\operatorname{im}(\partial_2),
\end{equation}
given by sending each overstrand in $D$ to the lift $\widetilde{a}$ of the corresponding
Wirtinger generator $a$. To complete the proof, we note that
the cellular boundary of a linear combination
$\sum_am_a\widetilde{a}$ is equal to
\[\partial_1\sum_am_a\widetilde{a}=m(t\widetilde{x}_0-\widetilde{x}_0),\] where
$m:=\sum_am_a$ is the coefficient sum.
Hence the homology group $\operatorname{ker}(\partial_1)/\operatorname{im}(\partial_2)$
is precisely the submodule of $C_1(\Sigma(L))/\operatorname{im}(\partial_2)$
generated by all linear combinations of generators $\widetilde{a}$ for which the coefficient sum is equal to zero.
Under the identification \eqref{eqn:ColDH}, this submodule corresponds to $\col(D)^{red}$,
and so the lemma follows.
\end{proof}

\begin{remark} The above proof also shows
the covering transformation acts by multiplication by $-1$ on $H_1(\Sigma(L);\ZZ)$.
\end{remark}

\begin{remark} If $a$ is a Wirtinger generator, then the lifts of $a$ and $a^{-1}$ that
start at $\widetilde{x}_0$ represent the same element of $C_1(\Sigma(L))/\operatorname{im}(\partial_2)$
(but if $a$ and $b$ are two Wirtinger generators, then the lifts of $a* b$ and $(a* b)^{-1}$
that start at $\widetilde{x}_0$ are negatives of each other in $H_1(\Sigma(L);\ZZ)$).
\end{remark}

\begin{remark} 
Changing the orientation of a link component has the effect of replacing each Wirtinger
generator $a$ associated with
this component by its inverse $a^{-1}$.
The previous remark therefore shows that the isomorphism from Lemma~\ref{lem:ColDH}
does not depend on the orientation of $L$.
\end{remark}

Let $L\cup U$ denote the disjoint union of $L$ with an unknot $U$.

\begin{corollary}\label{cor:ColDH} $\col(D)$ is canonically isomorphic to $H_1(\Sigma(L\cup U);\mathbb{Z})$.
\end{corollary}

\begin{proof} This follows rom Lemma~\ref{lem:ColDH} because
there is an identification
\[
\col(D)=\col(D\cup U)^{red}
\]
given by sending a generator $a$ to $a-u$, where $u$ corresponds to the
component $U$.
\end{proof}

\begin{remark}\label{rem:doublecoverLU} If $D$ is nonempty, then there is a short exact sequence
\[
0\longrightarrow\col(D)^{red}\longrightarrow\col(D)\stackrel{\phi}{\longrightarrow}\ZZ\longrightarrow 0,
\]
which gives rise to a non-canonical isomorphism
$\col(D)\cong\col(D)^{red}\oplus\ZZ$.
Together with Lemma~\ref{lem:ColDH} and Corollary~\ref{cor:ColDH},
this implies
\[
H_1(\Sigma(L\cup U);\ZZ)\cong H_1(\Sigma(L);\ZZ)\oplus\ZZ.
\]
The latter can also be seen geometrically by observing
that $\Sigma(L\cup U)\cong\Sigma(L)\# (S^1\times S^2)$
whenever $L$ is nonempty.
In fact, the homomorphism
$\phi\colon\col(D)\rightarrow\ZZ$ corresponds to the map 
$H_1(\Sigma(L\cup U);\ZZ)\rightarrow H_1(S^1\times S^2;\ZZ)\cong\ZZ$
induced by collapsing the first summand in $\Sigma(L)\# (S^1\times S^2)$.
\end{remark}

\begin{remark}\label{rem:relative} The proof of Lemma~\ref{lem:ColDH} also gives rise to a canonical identification
of $\col(D)$ with the relative homology group $H_1(\Sigma(L),\Sigma(L)^0;\ZZ)$ where $\Sigma(L)^0:=\{\widetilde{x}_0,t\widetilde{x}_0\}$.
Under this identification, the short exact sequence from the previous remark corresponds to the exact sequence
\[
0\longrightarrow H_1(\Sigma(L);\ZZ)\longrightarrow H_1(\Sigma(L),\Sigma(L)^0;\ZZ)\longrightarrow \widetilde{H}_0(\Sigma(L)^0;\ZZ)\longrightarrow 0
\]
which arises as a portion of the 
long exact sequence for the pair $(\Sigma(L),\Sigma(L)^0)$.
\end{remark}

\subsection{Modified Coloring Modules}

We can modify the definition of the coloring module by using understrands instead of overstrands:

\begin{definition}\label{def:mColD}
The modified coloring module $\mcol(D)$
is the abelian group
\[
\mcol(D):=\left\{\mbox{understrands of $D$}\,\left|\,\includeFig{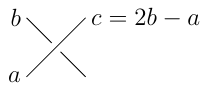}\right.\right\}.
\]
\end{definition}
Let $m(D)$ denote the link diagram obtained from $D$ by switching all crossings, so that understrands become overstrands.
Then there is an obvious identification
\[
\mcol(D)=\col(m(D)).
\]
If $r(D)$ denotes the reflection of $m(D)$ in the plane $\mathbb{R}^2$,
we further have an identification $\col(m(D))=\col(r(D))$.
Note that $m(D)$ represents the mirror image of $L$, and $r(D)$ represents
the link $r(L)$ obtained by rotating $L\subset\mathbb{R}^3$ by $180^\circ$ about an axis parallel to the plane of the picture.
Since $r(L)$ is isotopic to $L$, we have
\[
\mcol(D)=\col(r(D))\cong\col(D)=H_1(\Sigma(L\cup U);\ZZ).
\]
In the remainder of the paper, we would like to have an explicit identification
\begin{equation}\label{eqn:mColDH}
\mcol(D)=H_1(\Sigma(L\cup U);\ZZ)
\end{equation}
which does not involve rotating the link $L$.
Such an identification can be obtained by using a modified Wirtinger presentation
for $\pi_1(E,x_0)$ in which the generators correspond to understrands.
This modified presentation gives rise to a new handle decomposition of $\Sigma(L)$,
which has the same $0$-handles as the original one (and $\widetilde{x}_0$ and $t\widetilde{x}_0$
play the same roles as before), but different higher-dimensional handles.
The desired identification \eqref{eqn:mColDH}
is obtained by repeating the arguments from
Lemma~\ref{lem:ColDH} and Corollary~\ref{cor:ColDH} for this new handle decomposition.
Using these arguments, one also gets an identification
\begin{equation}\label{eqn:mColDHred}
\mcol(D)^{red}=H_1(\Sigma(L);\ZZ),
\end{equation}
where $\mcol(D)^{red}\subset\mcol(D)$ is the subgroupn generated by differences of understrands.

\section{Module Structure on Odd Khovanov Homology}\label{s:module}

In this section, we will introduce dot maps and use them
to prove Theorems~\ref{thm:module} through \ref{thm:arcaction} from
the introduction.
Throughout this section, we will view
the odd Khovanov bracket as a complex $\llbracket D\rrbracket$ in additive closure of the dotted odd Putyra category
\mbox{$o\:\!\mathcal{C}\mathit{ob}^3_{\bullet/\ell}$}.

\subsection{Dot Chain Maps}\label{subs:DotChainMaps}

Let $D$ be a link diagram and $p$ be a marked point on $e\setminus\partial e$ for an edge $e$
of $D$.
To this data, we can assign a cubical chain map
\[
x_e\colon\llbracket D\rrbracket\longrightarrow\llbracket D\rrbracket
\]
with components $(x_e)_\alpha\colon \llbracket D\rrbracket_\alpha\rightarrow\llbracket D\rrbracket_\alpha$
defined as follows. Let $S(D,\alpha)$ denote the quantity from \eqref{eqn:Smap}, and assume the point $p$ lies on a connected component $a_i$ of the 
resolution $D_\alpha$. Then
\begin{equation}\label{eqn:xedef}
\includeFig{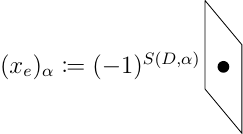}
\end{equation}
where the dotted vertical
sheet represents the identity cobordism $D_\alpha\times I$, decorated by a single dot placed on the component $a_i\times I$.
Since dots commute with merge saddles and anticommute with split saddles, the sign in \eqref{eqn:xedef} ensures
that $x_e$ is a chain map (see the discussion after \eqref{eqn:Smap}).
We will call $x_e$ the dot chain map assigned to $e$, and depict it as a dot placed on $e$, where the dot stands for the point $p$.

By applying the functor $\mathcal{F}_{odd}$ to $x_e$, we obtain a dot chain map
$\mathcal{F}_{odd}(x_e)\colon\mathcal{F}_{odd}(\llbracket D\rrbracket)\rightarrow\mathcal{F}_{odd}(\llbracket D\rrbracket)$
on the odd Khovanov complex. Explicitly, the latter map takes an element $a\in\mathcal{F}_{odd}(D_\alpha)$ to the wedge product
$(-1)^{S(D,\alpha)}a_i\wedge a$ for $a_i$ as above. We will henceforth abuse notation
and write $x_e$ not only for the original dot chain map from \eqref{eqn:xedef}, but also for $\mathcal{F}_{odd}(x_e)$
and for the induced map on odd Khovanov homology.

\begin{remark}\label{rem:ye} Instead of
taking the wedge product from the left with $(-1)^{S(D,\alpha)}a_i$, we could define a different
dot chain map by taking the wedge product from the right with $a_i$. This idea will be discussed
in more detail in Subsection~\ref{subs:modifieddots}.
\end{remark}

\subsection{Reduced Odd Khovanov Homology and Dot Maps}\label{subs:reduced}

As in the previous subsection, let
$p$ be a marked point on an edge $e$ of a link diagram $D$. Then the module
$\mathcal{F}_{odd}(D_\alpha)$ decomposes into two isomorphic summands
\[
\mathcal{F}_{odd}(D_\alpha)=\overline{\mathcal{F}}_{odd}(D_\alpha)\oplus\mathcal{F}^{(p)}_{odd}(D_\alpha),
\]
where the first summand is the reduced submodule from Subsection~\ref{subs:reducedfirst},
and the second summand is the submodule $a_i\wedge\mathcal{F}_{odd}(D_\alpha)=a_i\wedge\overline{\mathcal{F}}_{odd}(D_\alpha)$, for $a_i$ as above. The above decomposition extends to a decomposition of chain complexes,
which in turn induces a decomposition of odd Khovanov homology
\begin{equation}\label{eqn:OKhdecomp}
OKh(D)=\overline{OKh}(D)\oplus OKh(D)^{(p)}.
\end{equation}
The first summand in \eqref{eqn:OKhdecomp} is the reduced odd Khovanov homology defined in Subsection~\ref{subs:reducedfirst},
and the second summand is equal to
\[
OKh(D)^{(p)}=\operatorname{im}(x_e)=\operatorname{ker}(x_e).
\]
The two summands \eqref{eqn:OKhdecomp}  in are isomorphic, where the isomorphism is 
induced by the dot map $x_e\colon OKh(D)\rightarrow OKh(D)$.

There is another connection between reduced and nonreduced odd Khovanov homology,
which involves the disjoint union of $D$ with an unknot $U$.
Let $p$ be a marked point on this unknot, and $x_u$ be the corresponding dot map on $OKh(D\cup U)$.
Then there is an isomorphism
\[
OKh(D)\cong
OKh(D\cup U)^{(p)}
\]
given by $x_u\circ\iota$, where $\iota$ denotes the obvious inclusion
$OKh(D)\rightarrow OKh(D\cup U)$ given by a birth map on each $\mathcal{F}_{odd}(D_\alpha)$. 
Since $OKh(D\cup U)^{(p)}$ is isomorphic to $\overline{OKh}(D\cup U)$,
we thus obtain an isomorphism
\begin{equation}\label{eqn:OKhDU}
OKh(D)\cong
\overline{OKh}(D\cup U).
\end{equation}
One can check that this isomorphism is induced by the algebra isomorphisms
\[
\Lambda^*[a_1,\ldots,a_k]\stackrel{\cong}{\longrightarrow}
\Lambda^*[a_1-u,\ldots,a_k-u]
\]
which are given by sending each $a_j\subset D_\alpha$
to the difference $a_j-u$, where $u$ represents the component $U\subset D_\alpha\cup U$.

\begin{remark}
In \eqref{eqn:OKhdecomp}, we have ignored the gradings. If we define the quantum grading on $\overline{OKh}(D)$
as discussed in Subsection~\ref{subs:gradings}, then \eqref{eqn:OKhdecomp} implies:
\[
OKh(D)\cong\overline{OKh}(D)\{1\}\langle 0\rangle\oplus \overline{OKh}(D)\{-1\}\langle 1\rangle. 
\]
\end{remark}

\begin{remark} 
From the perspective of this paper, it is natural to allow $p$ to be an arbitrary
$\ZZ$-linear combination of marked points on $D$. If the coefficient sum in this $\ZZ$-linear combination is
equal to $1$, then all of the results from this subsection remain true in this generalized
setting, provided one replaces the component $a_i$ by the corresponding linear combination of components,
and the dot map $x_e$ by the corresponding linear combination of dot maps.
\end{remark}

\begin{remark} One can define an antiderivation $\partial_\alpha$ on the exterior algebra $\mathcal{F}_{odd}(D_\alpha)$
by setting $\partial_\alpha 1:=0$ and $\partial_\alpha(a_j\wedge a):=a-(a_j\wedge\partial_\alpha a)$ for all generators $a_j$.
This antiderivation was first used in \cite{ES2025} to define a $\mathfrak{gl}_{1|1}$-action on odd Khovanov homology, and over $\mathbb{F}_2$, related maps were previously studied in \cite{We2010,Shu2014}.
The maps $\partial_\alpha$ give rise to a chain map with components $(-1)^{S(D,\alpha)}\partial_\alpha$.
In particular, there is an induced map $\partial\colon OKh(D)\rightarrow OKh(D)$, and the latter map provides an inverse for the map
$\overline{OKh}(D)\rightarrow OKh(D)^{(p)}$ given by $x_e$. Moreover, $\partial$ satisfies
$\partial\circ x_e+x_e\circ\partial=\operatorname{id}$,
where the two terms on the left-hand side are precisely the projections onto the two summands in \eqref{eqn:OKhdecomp}.
\end{remark}

\subsection{Definition of the Module Structure}\label{subs:moduledef}

In \cite{Ma2014}, Manion studied a variant of our dot chain maps,
although he viewed his maps as components of a twisted odd differential.
In the type X setting, he proved that the isomorphism type of
his twisted odd Khovanov
complex does not change if a marked point passes under a crossing \cite[Theorem 3.1]{Ma2014}. 
An adaption of his proof shows:

\begin{lemma}\label{lem:typeXslide} On the type X odd Khovanov bracket,
\[
\includeFig{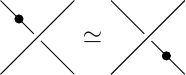}
\]
where the dots represent dot chain maps, defined as in \eqref{eqn:xedef}.
\end{lemma}

In the present form, this lemma first appeared in the first author's thesis \cite{Mig24}, where it was also
used to establish Proposition~\ref{prop:5:colMod} below (see also \cite{WehrliTalk}).
More recently, a related result appeared in \cite{ES2025}. Although the proof given there
is shorter because it is formulated in a different setting, we will supply a proof of Lemma~\ref{lem:typeXslide} 
based on Manion's argument from \cite{Ma2014}.

\begin{proof}
Let $\epsilon$ be a type X sign assignment for the given link diagram $D$,
and let $d$ be the corresponding
differential on the odd Khovanov bracket of $D$.
Assume that the crossings of $D$ are numbered so that the crossing shown in
the lemma comes first.
Then the component of $d$ that corresponds to this crossing
can be written as $d_c=\sum_\alpha d_{\star\alpha}$,
where $\alpha$ stands for a resolution of the other crossings of $D$,
and $d_{\star\alpha}$
is the component of $d$ that connects $\llbracket D\rrbracket_{0\alpha}$ and $\llbracket D\rrbracket_{1\alpha}$.
The desired homotopy between the two dot chain maps from Lemma~\ref{lem:typeXslide}
is now given by
\[
h=\sum_\alpha h_{\star\alpha},
\]
where $h_{\star\alpha}\colon\llbracket D\rrbracket_{1\alpha}\rightarrow\llbracket D\rrbracket_{0\alpha}$
is the diagonal arrow in Figure~\ref{fig:homotopy}.
The horizontal arrows in this figure represent the edge map
$d_{\star\alpha}=\epsilon_{\star\alpha}d'_{\star\alpha}\colon\llbracket D\rrbracket_{0\alpha}\rightarrow\llbracket D\rrbracket_{1\alpha}$,
where $\epsilon_{\star\alpha}$ is the sign in $d_{\star\alpha}$, and $d'_{\star\alpha}$
is the underlying saddle cobordism.
Moreover, the vertical arrow on the left side represents the chain map $f$
which is given by the indicated difference of dot chain maps.
The vertical arrows on the right represent the components $f_{0\alpha}$ and $f_{1\alpha}$
of this chain map, where in $f_{1\star}$, we have used a solid dot to represent
a dot on the front sheet, and an empty dot to represent a dot on the back sheet.
Note also that the chronological cobordism $h'_{\star\alpha}$ that appears in the homotopy $h_{\star\alpha}=(-1)^{S(D,1\alpha)}\epsilon_{\star\alpha}h'_{\star\alpha}$ can
be obtained by flipping $d'_{\star\alpha}$ upside down and then rotating
the arrow at the saddle point counterclockwise by $90^\circ$.

\begin{figure}[H]
\includeFig{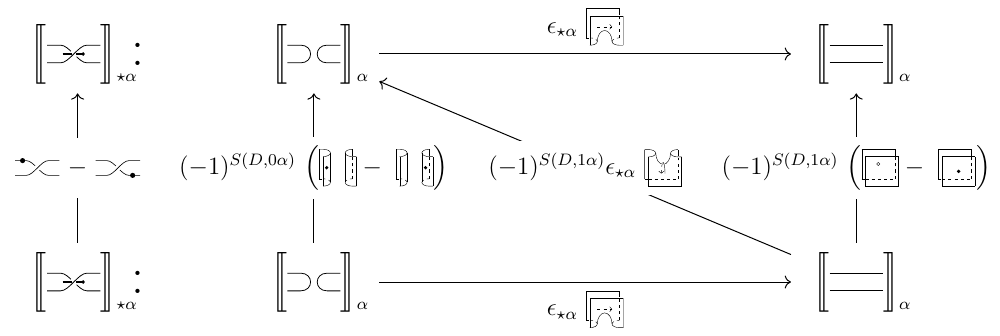}
\caption{The homotopy $h_{\star\alpha}$.}\label{fig:homotopy}
\end{figure}

Applying the horizontal neck cutting relation \eqref{eqn:hneck} to the tube in $d'_{\star\alpha}\circ h'_{\star\alpha}$, we obtain
\[
\includeFig{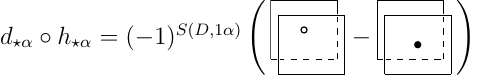}
\]
Similarly, we obtain
\[
\includeFig{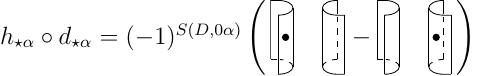}
\]
where we have used that $S(D,0\alpha)=S(D,1\alpha)$ whenever
the difference on the right-hand side of this equation is nonzero.
Thus
$d_{\star\alpha}\circ h_{\star\alpha}=f_{1\alpha}$ and
$h_{\star\alpha}\circ d_{\star\alpha}=f_{0\alpha}$, and hence
\[d_ch+hd_c=f.\]
To complete the proof of Lemma~\ref{lem:typeXslide}, we must show that $dh+hd=d_ch+hd_c$.
For this, it suffices to show that
the two terms $d_{\xi_0}h_{\star\alpha}$ and $h_{\star\alpha'}d_{\xi_1}$
cancel in $dh+hd$
for any pair of components $d_{\xi_0}\colon D_{0\alpha}\rightarrow D_{0\alpha'}$ 
and $d_{\xi_1}\colon D_{1\alpha}\rightarrow D_{1\alpha'}$ of $d-d_c$.
This was essentially shown by Manion \cite[Lemma~2.3]{Ma2014}
by using a case-by-case analysis depending on
possible configurations of two saddles.
More precisely, this analysis shows that the~square
\begin{equation}\label{eqn:square}
\begin{tikzcd}
{}& D_{1\alpha}\ar[dl,"h_{\star\alpha}"']\ar[dr,"d_{\xi_1}"]&{}\\
D_{0\alpha}\ar[dr,"d_{\xi_0}"']&{}&D_{1\alpha'}\ar[dl,"h_{\star\alpha'}"]\\
{}& D_{0\alpha'}&{}
\end{tikzcd}
\end{equation}
always anticommutes, provided the sign assignment $\epsilon$ is of type X.
\end{proof}

If $\epsilon$ instead has type Y, then dots can pass over crossings:

\begin{lemma}\label{lem:typeYslide} On the type Y odd Khovanov bracket,
\[
\includeFig{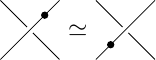}
\]
where the dots represents dot chain maps, defined as in \eqref{eqn:xedef}.
\end{lemma}

\begin{proof} The proof is almost identical to the previous proof. The main difference
is that the desired homotopy is now given by
\[
k=\sum_\alpha k_{\star\alpha}
\]
for $k_{\star\alpha}=(-1)^{S(D,1\alpha)}\epsilon_{\star\alpha}k'_{\star\alpha}$,
where $\epsilon$ is a type Y sign assignment, and the chronological cobordism $k'_{\star\alpha}$ is obtained by
reversing the arrow in $h'_{\star\alpha}$ (see Figure~\ref{fig:homotopy1}).
One can check that $k_{\star\alpha}$ is also equal to $-(-1)^{S(D,0\alpha)}\epsilon_{\star\alpha}h'_{\star\alpha}$.

\begin{figure}[h]
\includeFig{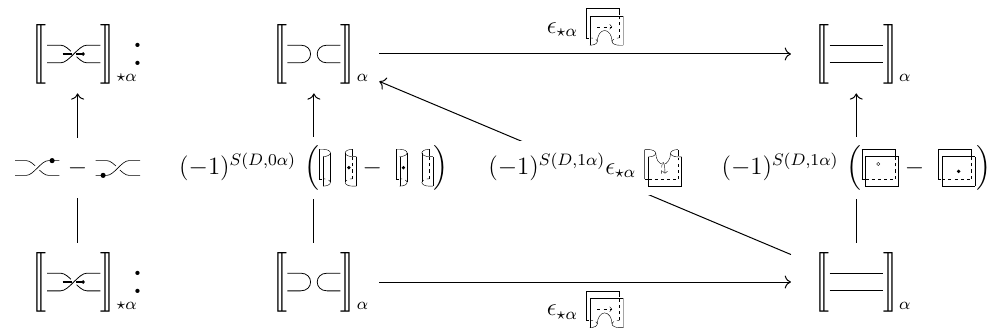}
\caption{The homotopy $k_{\star\alpha}$.}\label{fig:homotopy1}
\end{figure}

The proof that $k$ is a homotopy between the dot chain maps from
Lemma~\ref{lem:typeYslide} is analogous to the corresponding proof for $h$.
In particular, the proof that $dk+kd=d_ck+kd_c$ uses a minor
modification of Manion's case-by-case analysis from \cite{Ma2014},
and relies on the assumption that $\epsilon$ is a type Y sign assignment.
\end{proof}

\begin{figure}[h]
\includeFig{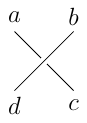}
\caption{Crossing with labeled edges.}\label{fig:crossing}
\end{figure}

Now consider a crossing of $D$, and assume that the adjacent edges are labeled $a,b,c,d$ in clockwise order, where $b$ and $d$ belong to the overcrossing strand (see Figure~\ref{fig:crossing}).
Then the associated dot chain maps satisfy
\begin{equation}\label{eqn:abcd}
x_a+x_c=x_b+x_d
\end{equation}
because their restrictions to the $0$-resolution satisfy
$x_a=x_d$ and $x_c=x_b$,
and their restrictions to the $1$-resolution satisfy $x_a=x_b$ and $x_c=x_d$.

If we use a type Y sign assignment, then
Lemma~\ref{lem:typeYslide} implies 
that the homotopy class of $x_e\colon\llbracket D\rrbracket\rightarrow\llbracket D\rrbracket$
depends only on the overstrand that contains $e$. In particular,
$x_b\simeq x_d$,
and hence \eqref{eqn:abcd} implies
\begin{equation}\label{eqn:dotmaprelation}
x_a+x_c\simeq 2x_b.
\end{equation}
Up to homotopy,
the assignment $e\mapsto x_e$ is thus compatible with 
the relation $2b-a-c=0$ that holds in the coloring module of $D$.
This leads to:

\begin{proposition}\label{prop:5:colMod}
The type Y odd Khovanov homology of $D$ is a module over $\Lambda^*\col(D)$.
\end{proposition}

\begin{proof}
The module action on the type Y odd Khovanov homology is induced by the assignment $e_1\wedge\cdots\wedge e_\ell\mapsto x_{e_1}\circ\cdots\circ x_{e_\ell}$. This assignment is compatible with the exterior algebra
relations because dots anticommute and square to zero.
\end{proof}

If we instead use a type X sign assignment, then Lemma~\ref{lem:typeXslide}
and equation \eqref{eqn:abcd} imply:

\begin{proposition}\label{prop:5:mcolMod}
The type X odd Khovanov homology of $D$ is a module over $\Lambda^*\mcol(D)$.
\end{proposition}

Using the identification $\col(D)=H_1(\Sigma(L\cup U);\ZZ)$ from Corollary~\ref{cor:ColDH} and 
the identification $\mcol(D)=H_1(\Sigma(L\cup U);\ZZ)$ from
\eqref{eqn:mColDH}, we thus obtain the following theorem from the introduction:

\TheoremModuleNonRed*

Since the reduced coloring module is generated by differences
of overstrands, elements of $\col(D)^{red}$ act 
on the type Y odd Khovanov homology by differences of dot maps.
It follows from the definitions of
$\overline{\mathcal{F}}_{odd}(D_\alpha)$
and $\mathcal{F}^{(p)}_{odd}(D_\alpha)$ that such differences preserve 
the two summands $\overline{OKh}(D)$
and $OKh(D)^{(p)}$ in \eqref{eqn:OKhdecomp}. 
In particular, this shows that the reduced odd Khovanov homology
$\overline{OKh}(D)$ is a module over
$\Lambda^*\col(D)^{red}$ if type Y sign assignments are used,
or similarly, a module over $\Lambda^*\mcol(D)^{red}$ if type X sign assignments are used.
The identifications from Lemma~\ref{lem:ColDH} and equation \eqref{eqn:mColDHred} now show:

\TheoremModule*

\begin{remark}\label{rem:OKhDU} 
 Theorem~\ref{thm:modulenonred}
can be deduced from
Theorem~\ref{thm:module} by using the isomorphism $OKh(D)\cong\overline{OKh}(D\cup U)$
from \eqref{eqn:OKhDU}. In fact, this isomorphism intertwines the module actions
if these actions are defined as described above.
In the type Y setting, this follows because
the module action on $OKh(D)$ is defined by using the identifications $\col(D)=\col(D\cup U)^{red}=H_1(\Sigma(L\cup U);\ZZ)$,
where the first identification takes an overstrand $a$ to $a-u$ (see the proof of Corollary~\ref{cor:ColDH}).
For a type Y sign assignment and an overstrand $a$ of $D$, the element \[a-u\in\col(D\cup U)^{red}=H_1(\Sigma(L\cup U);\ZZ)\] thus acts on $OKh(D)$ by
$x_a$ and on $\overline{OKh}(D\cup U)$ by $x_a-x_u$, and these actions
are intertwined by the isomorphism from \eqref{eqn:OKhDU}.
\end{remark}

\begin{remark} By Remark~\ref{rem:relative}, we can also interpret the
odd Khovanov homology of $L$ as a module over the exterior algebra $\Lambda^*H_1(\Sigma(L),\Sigma(L)^0;\ZZ)$.
\end{remark}

\subsection{Representing Homology Classes by Embedded Arcs}\label{subs:arcs}

We will now describe the module action from 
Theorem~\ref{thm:module}
more concretely in terms of embedded arcs in $\mathbb{R}^3$.
As in the previous subsection, let $L\subset\mathbb{R}^3$ be a link with diagram $D$,
and consider an embedded oriented arc $\alpha\subset\mathbb{R}^3$
such that $\alpha\cap L=\partial\alpha$, as shown in Figure~\ref{fig:arcalpha}.
The full preimage of $\alpha$ in $\Sigma(L)$ is an unoriented simple
closed curve $\widehat{\alpha}\subset\Sigma(L)$.

To specify an orientation on $\widehat{\alpha}$,
we will assume that the arc $\alpha$ meets $L$ transversely along $\partial\alpha$, and that the
embedded graph $L\cup\alpha\subset\mathbb{R}^3$
is in general position with respect to the projection to the plane of the picture.
We will further consider a short subarc $\alpha_\epsilon\subset\alpha$ that starts at the initial point of $\alpha$,
and that does not cross any strands of $D$.

\begin{figure}[h]
\includeFig{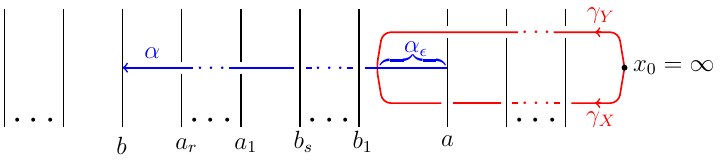}
\caption{The paths $\gamma_X$ and $\gamma_Y$ from the basepoint $x_0=\infty$ to
the endpoint of the initial subarc $\alpha_\epsilon\subset\alpha$.}\label{fig:gammaXandY}
\end{figure}

Let $\gamma_X,\gamma_Y\subset S^3\setminus L$
be two paths
from the basepoint $x_0:=\infty$ to the final endpoint of the subarc $\alpha_\epsilon$, such that $\gamma_X$
does not overcross any strands of $D$, and $\gamma_Y$ does not undercross any strands
of $D$ (see Figure~\ref{fig:gammaXandY}). Moreover, let $\widetilde{\gamma}_X,\widetilde{\gamma}_Y\subset\Sigma(L)$
be, respectively, the lifts of these paths that start at
the distinguished basepoint $\widetilde{x}_0$. Then
\[
\widehat{\alpha}=\widetilde{\alpha}_X\cup t\widetilde{\alpha}_X=\widetilde{\alpha}_Y\cup t\widetilde{\alpha}_Y,
\]
where $t$ denotes the covering transformation,
and $\widetilde{\alpha}_X$ and $\widetilde{\alpha}_Y$ denote
the lifts of the arc $\alpha$ that pass through the endpoints of $\widetilde{\gamma}_X$ and $\widetilde{\gamma}_Y$, respectively.
The arcs $\widetilde{\alpha}_X$ and $\widetilde{\alpha}_Y$ inherit orientations from the orientation of $\alpha$,
and we can thus define two oriented versions of $\widehat{\alpha}$ by
\[
\widehat{\alpha}_X:=\widetilde{\alpha}_X\cup-t\widetilde{\alpha}_X\quad\mbox{and}\quad
\widehat{\alpha}_Y:=\widetilde{\alpha}_Y\cup-t\widetilde{\alpha}_Y.
\]
It is easy to see that $\widehat{\alpha}_X=\widehat{\alpha}_Y$
precisely when the loop $\gamma_Y^{-1}*\gamma_X$
has even linking number with $L$. Equivalently, this is the case when the link diagram
$D$ has even winding number around the final point of $\alpha_\epsilon$.
Furthermore, our definitions imply that the homology classes
\[[\widehat{\alpha}_X],[\widehat{\alpha}_Y]\in H_1(\Sigma(L);\ZZ)\]
remain unchanged under ambient isotopies of $\alpha\subset\mathbb{R}^3$
that fix the link $L\subset\mathbb{R}^3$ setwise and the subarc $\alpha_\epsilon\subset\alpha$ pointwise.
Using the same notations as in Figure~\ref{fig:arcalpha}
from the introduction, we now have:

\TheoremArcAction*

\begin{proof} 
We will only prove the statement for $[\widehat{\alpha}_Y]$ and leave the proof of the other statement
to the reader. To begin with, write $\alpha$ as a composition
\[
\alpha=\beta_\epsilon*\mu*\alpha_\epsilon,
\]
where $\alpha_\epsilon$ denotes the subarc of $\alpha$ described above, $\beta_\epsilon$ denotes a final
subarc of $\alpha$ that does not cross any strands of $D$, and $\mu$ denotes the remaining middle part of $\alpha$.
Let $\gamma_Y$ be as above, and let $\eta_Y$ denote a path in $S^3\setminus L$ that starts at $x_0$
and ends at the initial point of $\beta_\epsilon$, and that does not undercross any strands of $D$.
By applying a path homotopy to $\alpha$ that is supported away from $L$, 
we can replace the subarc $\beta_\epsilon$ by $\beta_\epsilon*\eta_Y$,
the subarc $\mu$ by $\eta_Y^{-1}*\mu*\gamma_Y$, and the subarc
$\alpha_\epsilon$ by
$\gamma_Y^{-1}*\alpha_\epsilon$.
By making this replacement, we can thus assume that $\beta_\epsilon$
is a path from $x_0$ to $L$, $\mu$ is a loop based at $x_0$, and $\alpha_\epsilon$
is a path from $L$ to $x_0$.

In the remainder of this proof, we will use the same notation for an edge $e$ of the link diagram $D$,
the overstrand of $D$ that contains $e$, and
the Wirtinger generator of $\pi_1(S^3-L,x_0)$ that corresponds to this overstrand.
With this in mind, we can express
the loop $\mu$ as the composition
\[\mu=b_s*\cdots*b_1,\]
where $b_1,\ldots,b_s$ denote the Wirtinger generators
that correspond to the edges of the same name from Figure~\ref{fig:arcalpha}.
Let $\widetilde{b}_1,\ldots,\widetilde{b}_s$ denote the lifts of these generators that start at $\widetilde{x}_0$.
Likewise, let $\widetilde{\alpha}_\epsilon$ denote the lift of $\alpha_\epsilon$ that ends at $\widetilde{x}_0$,
and $\widetilde{\beta}_\epsilon$ denote the lift of $\beta_\epsilon$ that starts at $\widetilde{x}_0$.
Interpreting all of these lifts as singular $1$-simplices, we then have
\[
\widetilde{\alpha}_Y=\widetilde{\alpha}_\epsilon + \widetilde{b}_1+t\widetilde{b}_2+\ldots+t^{s-1}\widetilde{b}_s
+t^s\widetilde{\beta}_\epsilon,
\]
where the powers of $t$ arise because the path
$t^{j-1}\widetilde{b}_j$ starts at $t^{j-1}\widetilde{x}_0$ and ends at $t^j\widetilde{x}_0$.
In view of our orientation convention for $\widehat{\alpha}_Y$, this further implies
\begin{align*}
\widehat{\alpha}_Y &=\widetilde{\alpha}_Y-t\widetilde{\alpha}_Y\\
&=(1-t)\widetilde{\alpha}_Y\\
&=(1-t)\widetilde{\alpha}_\epsilon + (1-t)\widetilde{b}_1+t(1-t)\widetilde{b}_2+\ldots+t^{s-1}(1-t)\widetilde{b}_s
+t^s(1-t)\widetilde{\beta}_\epsilon.
\end{align*}
To complete the proof, we note that the linear combination
$(1-t)\widetilde{\alpha}_\epsilon=-t\widetilde{\alpha}_\epsilon+\widetilde{\alpha}_\epsilon$
can be viewed as a path in $\Sigma(L)$ from $t\widetilde{x}_0$ to $\widetilde{x}_0$.
Using the handle decomposition of $\Sigma(L)$ from the proof of Lemma~\ref{lem:ColDH},
we can further assume that this path is contained in the $2$-handle attached to $\widetilde{a}+t\widetilde{a}$,
where $a$ is the Wirtinger generator associated to the edge $a$ from Figure~\ref{fig:arcalpha},
and $\widetilde{a}$ denotes the lift of the generator $a$ that starts at $\widetilde{x}_0$.
Since $-\widetilde{a}$ is contained in the same $2$-handle and also connects $t\widetilde{x}_0$
to $\widetilde{x}_0$, we obtain that $(1-t)\widetilde{\alpha}_\epsilon$ is homologous to $-\widetilde{a}$.
Similarly, we obtain that $(1-t)\widetilde{\beta}_s$ is homologous to $\widetilde{b}$.
Lastly, we know that for any Wirtinger generator $g$,
the lift $t\widetilde{g}$ is homologous to $-\widetilde{g}$ because $\widetilde{g}+t\widetilde{g}$
bounds a $2$-handle in $\Sigma(L)$.
The previous formula for $\widehat{\alpha}_Y$ thus implies that
$\widehat{\alpha}_Y$ is homologous to the linear combination
\[
-\widetilde{a} + 2\widetilde{b}_1-2\widetilde{b}_2+\ldots+(-1)^{s-1}2\widetilde{b}_s
+(-1)^s\widetilde{b}.
\]
The second part of Theorem~\ref{thm:arcaction} now follows
because, under
 the identification $H_1(\Sigma(L);\ZZ)=\col(D)^{red}$ from Lemma~\ref{lem:ColDH}, 
this linear combination corresponds to the element
\[
-a + 2b_1-2b_2+\ldots+(-1)^{s-1}2b_s
+(-1)^sb
\]
of the reduced coloring module of $D$.
\end{proof}

\begin{remark}\label{rem:dotdifference}
If $\alpha$
does not undercross any strands of $D$, then Theorem~\ref{thm:arcaction}
implies that
$[\widehat{\alpha}_Y]$ acts by $x_b-x_a$ on the reduced type Y odd Khovanov homology of $D$.
Similarly, if $\alpha$ does not overcross any strands of $D$, then $[\widehat{\alpha}_X]$ acts by the same difference $x_b-x_a$ on the
reduced type X odd Khovanov homology of $D$.
\end{remark}

Although we have formulated Theorem~\ref{thm:arcaction} for the reduced odd Khovanov homology of $D$,
the theorem remains valid for the nonreduced odd Khovanov homology of $D$ if one interprets the homology classes
$[\widehat{\alpha}_X]$ and $[\widehat{\alpha}_Y]$ as elements of $H_1(\Sigma(L\cup U);\ZZ)$
by using the inclusion $H_1(\Sigma(L);\ZZ)\subset H_1(\Sigma(L\cup U);\ZZ)$ that follows implicitly
from the inclusion $\col(D)^{red}\subset\col(D)$.

For general elements of $H_1(\Sigma(L\cup U);\ZZ)$ (that do not come from elements of $H_1(\Sigma(L);\ZZ)$),
one can understand the module action on the nonreduced odd Khovanov homology of $D$ as follows.
Assume that the unknot $U\subset\mathbb{R}^3$ is
given by a trivial circle in the unbounded
component of $\mathbb{R}^2\setminus D$, and
let $\alpha\subset\mathbb{R}^3$ be an embedded oriented arc
which starts at a point of $U$ and ends on an edge $e$ of $D$,
and which meets $L\cup U$ in no other points.
Assume further that $\alpha$ does not cross $U$.
Then the diagram $D\cup U$ does not wind around the final point of $\alpha_\epsilon$, where
$\alpha_\epsilon\subset\alpha$ denotes a short initial subarc of $\alpha$ that does not cross $D$.
If we define the oriented simple closed curves $\widehat{\alpha}_X,\widehat{\alpha}_Y\subset\Sigma(L\cup U)$
by repeating our earlier construction, it thus follows that $\widehat{\alpha}_X=\widehat{\alpha}_Y$.
We will therefore drop the subscripts $X$ and $Y$ and 
simply write $\widehat{\alpha}$
for the oriented simple closed curve $\widehat{\alpha}_X=\widehat{\alpha}_Y$.

Now let $a_1,\ldots,a_r$ denote the edges of $D$ that $\alpha$ overcrosses
as it goes from $U$ to $e$, and let $b_1,\ldots,b_s$ denote the edges that it undercrosses.
Moreover, let $[\widehat{\alpha}]$ denote the homology class of $\widehat{\alpha}$ in $H_1(\Sigma(L\cup U);\ZZ)$.
Theorem~\ref{thm:arcaction} then implies:

\begin{corollary}\label{cor:arcaction} 

On the type X odd Khovanov homology, 
$[\widehat{\alpha}]$ acts by
\[
2x_{a_1}-2x_{a_2}+\ldots+(-1)^{r-1}2x_{a_r}+(-1)^rx_e,
\]
and on the type Y odd Khovanov homology, $[\widehat{\alpha}]$ acts by
\[
2x_{b_1}-2x_{b_2}+\ldots+(-1)^{s-1}2x_{b_s}+(-1)^sx_e.
\]
\end{corollary}

\begin{proof}
By Theorem~\ref{thm:arcaction}, the homology class $[\widehat{\alpha}]$
acts by
\[
-x_u+2x_{a_1}-2x_{a_2}+\ldots+(-1)^{r-1}2x_{a_r}+(-1)^rx_e
\]
on the reduced type X Khovanov homology of $D\cup U$, and by
\[
-x_u+2x_{b_1}-2x_{b_2}+\ldots+(-1)^{s-1}2x_{b_s}+(-1)^sx_e,
\]
on the reduced type Y Khovanov homology $D\cup U$.
In each of these linear combinations of dot maps, the coefficient
sum is equal to zero, and thus these linear combinations do not change if
one replaces each dot map by its difference with $x_u$.
Moreover, the initial term $-x_u$ in these linear
combinations becomes $-(x_u-x_u)=0$ under this replacement.
The corollary now follows by applying the isomorphism $\overline{OKh}(D\cup U)\cong OKh(D)$ from~\eqref{eqn:OKhDU}, 
and using that this isomorphism intertwines the module actions by Remark~\ref{rem:OKhDU}.
\end{proof}

Note that, in Corollary~\ref{cor:arcaction}, the arc $\alpha$ always starts on the unknot $U$.
When representing such an arc graphically,
we can therefore omit $U$ from the picture, and imagine
that $\alpha$ starts in the unbounded component of $\mathbb{R}^2\setminus D$.
The homology class of
$\widehat{\alpha}$ in $\Sigma(L\cup U)$
then remains unchanged under ambient isotopies of $\alpha\subset\mathbb{R}^3$ that fix $L$ setwise
and keep the initial point of $\alpha$ in this unbounded component.
Note that this perspective is in line with Remark~\ref{rem:relative}, which suggests that $\alpha$ can be viewed
as an arc that starts $x_0$, and $[\widehat{\alpha}]$ can be viewed as a relative homology class
in $H_1(\Sigma(L),\Sigma(L)^0;\ZZ)$ for $\Sigma(L)^0=\{\widetilde{x}_0,t\widetilde{x}_0\}$.

In the special case where $\alpha$ does not overcross any strands of $D$ and type X sign assignments are used,
or $\alpha$ does not undercross any strands of $D$ and type Y sign assignments are used, the action of $[\widehat{\alpha}]$
on odd Khovanov homology is simply given by the dot map $x_e$. Thus, we
obtain a graphical representation of $x_e$, in terms 
of an embedded arc in $\mathbb{R}^3$ (see Figure~\ref{fig:dotarc}).

\begin{figure}[H]
\includeFig{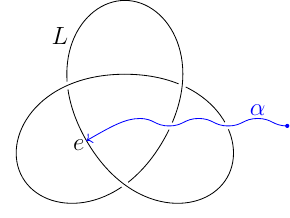}
\caption{An arc $\alpha$ that corresponds to the dot map $x_e$ in the type Y setting.}\label{fig:dotarc}
\end{figure}

Using the graphical description of the dot maps $x_e$, we can gain a renewed understanding of
the (modified) coloring module relations that hold among the dot maps.
For the type Y theory, this is explained in Figure~\ref{fig:arccrossing},
where $\alpha$ and $\alpha'$ give rise to the same homology class.
The relation $x_c=2x_b-x_a$ follows because Corollary~\ref{cor:arcaction}
tells us that $\alpha$ acts by $x_c$, and $\alpha'$ acts by $2x_b-x_a$.

\begin{figure}[h]
\includeFig{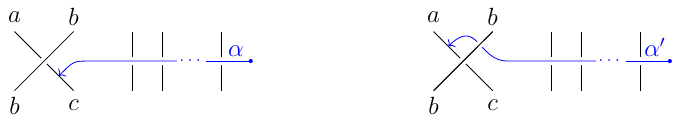}
\caption{The crossing relation in the type Y setting.
The arcs $\alpha$ and $\alpha'$ are isotopic, and their actions are given, respectively,
by $x_c$ and $2x_b-x_a$, showing that $x_c=2x_b-x_a$.}\label{fig:arccrossing}
\end{figure}

We can also use Corollary~\ref{cor:arcaction} and our graphical description of dot maps
to prove the
equivalence of the module structures in type X and type Y:

\TheoremTypes*

\begin{proof} Let $L\subset\mathbb{R}^3$
be a link, and $L'$ be the link $r(L)$, where $r$ denotes the $180^\circ$ rotation of $\mathbb{R}^3$
about an axis parallel to the plane of the picture (see Figure~\ref{fig:types}). The definition of the odd Khovanov complex of
$L$ involves a choice of arrows at the crossings of $L$, and we will assume that, in $L'$, this
arrows are rotated as well.

\begin{figure}[h]
\includeFig{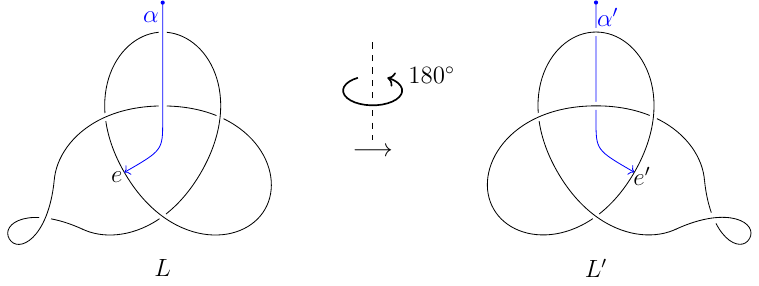}
\caption{The links $L$ and $L'$ and the arcs $\alpha$ and $\alpha'$.}\label{fig:types}
\end{figure}

The canonical isomorphism between the odd Khovanov homologies of type Y and type X
is then given by a sequence of isomorphisms
\begin{equation}\label{eqn:types}
OKh(L,\epsilon_Y)\cong OKh(L',\epsilon'_{Y})\cong OKh(L,\epsilon_X),
\end{equation}
where $\epsilon_Y$ and $\epsilon'_Y$ are sign assignments of type Y, and
$\epsilon_X$ is the sign assignment $\epsilon'_Y$, but viewed as a type X
assignment for the original link $L$. The first isomorphism in \eqref{eqn:types}
is induced by the Reidemeister invariance of odd Khovanov homology,
and the second isomorphism is a canonical identification, coming
from the fact that each resolution of the diagram of $L'$ is the
reflection in the plane $\mathbb{R}^2$ of the corresponding resolution of $L$.

Now let $e$ be an edge of the link diagram of $L$, and $\alpha$ be an arc
that corresponds to the dot map $x_e$ in the type Y setting, as shown in Figure~\ref{fig:types}.
Note that $\alpha$ starts in the unbounded component of $\mathbb{R}^2\setminus D$, and ends on $e$,
while overcrossing any intermediate strands of $L$. Let $\alpha':=r(\alpha)\subset\mathbb{R}^3$
be the rotated arc, and note that it undercrosses any intermediate arcs of $L'$.

To prove Theorem~\ref{thm:types}, we observe that the first isomorphism
in \eqref{eqn:types} intertwines the module actions of $[\widehat{\alpha}]$
and $[\widehat{\alpha}']$ because of the Reidemeister invariance of these actions, shown in
the proof of Theorem~\ref{thm:functor} below. The second isomorphism in \eqref{eqn:types}
intertwines the actions of $[\widehat{\alpha}']$ and $[\widehat{\alpha}]$ because
the formula from Corollary~\ref{cor:arcaction}
for the action of $[\widehat{\alpha}']$ in the type Y setting
coincides with the formula for the action of $[\widehat{\alpha}]$
in the type X setting. Indeed, this follows because the edges of $L'$ that $\alpha'$
undercrosses correspond precisely to the edges of $L$ that $\alpha$ overcrosses.

In summary, we see that the isomorphism
$OKh(L,\epsilon_Y)\cong OKh(L,\epsilon_X)$ from \eqref{eqn:types} intertwines
the actions of a homology class $c\in H_1(\Sigma(L\cup U);\ZZ)$ in the special case where $c$ acts
on $OKh(L,\epsilon_Y)$ by an individual dot map $x_e$.
This completes the proof of Theorem~\ref{thm:types} because the action of
$\Lambda^*H_1(\Sigma(L\cup U))$ is generated
by such dot maps.
\end{proof}

\begin{remark} The above proof shows that the canonical isomorphism
between $OKh(L,\epsilon_Y)$ and $OKh(L,\epsilon_X)$ intertwines the action of $x_e$ on $OKh(L,\epsilon_Y)$
with the action of
$2x_{a_1}-2x_{a_2}+\ldots+(-1)^{r-1}2x_{a_r}+(-1)^rx_e$
on $OKh(L,\epsilon_X)$, where $a_1,\ldots,a_r$ denote the edges of the link diagram
that the arc $\alpha$ from Figure~\ref{fig:types} overcrosses on its way from the
unbounded region to the edge $e$.
\end{remark}

\subsection{Module Structure over Rational Coefficients}

In the previous subsections, we have seen that any difference of two dot maps can
be interpreted as the action of a first homology class in $\Sigma(L)$. If $L$ is a knot,
then $\Sigma(L)$ is a rational homology sphere, and hence the rational first homology
of $\Sigma(L)$ is trivial. This implies:

\begin{proposition}
If $L$ is a knot, then the action of the dot map $x_e$ on the rational odd Khovanov homology
of $L$ is independent of the edge $e$.
\end{proposition}

In other words, this proposition states that dots can pass
freely over and under crossings of $L$ as long as $L$ is a knot
and coefficients are in $\mathbb{Q}$.
We shall see in the next subsection that this is not the case when coefficients are in $\mathbb{Z}$.
It is easy to see that this also not the case over rational coefficients when $L$ is a link with more than
one component.

An example is shown in  Figure~\ref{fig:unlink}, where the two pictures represent two
dot configurations on a nontrivial diagram of a $2$-component unlink. In the left picture, the two dots
lie on the same edge of the link diagram, and thus the corresponding dot maps compose
to zero since $x_e^2=0$. In the right picture, the two dots still lie on the same link component,
but no longer on the same edge of the link diagram.
In the type X setting, the induced dot maps still compose to zero because the two dots lie
on the same understrand of $L$. However, in the type Y setting, an application of the
relation $x_c\simeq 2x_b-x_a$ from \eqref{eqn:dotmaprelation} shows that the induced
dot maps compose to a nonzero map on odd Khovanov homology.
Explicitly, this map sends a generator of top quantum degree to twice a generator of bottom quantum
degree. In particular, it is nonzero over integer and rational coefficients (but zero over $\mathbb{F}_2$).

\begin{figure}[h]
\includeFig{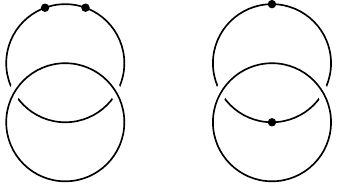}
\caption{Two dot configurations on a diagram of a $2$-component unlink.
The left configuration always induces the zero map, while the right configuration induces
a nonzero map on the rational type Y odd Khovanov homology.}\label{fig:unlink}
\end{figure}

\subsection{Module Structure for Pretzel Knots}

In \cite{Sh2011} after considering the even and odd Khovanov homologies of $9_{46}$ or the $(3,3,-3)$-pretzel knot, Shumakovich observed torsion in the odd setting that did not appear in the even setting.  He conjectured 
``This suggest a certain [...] symmetry on the odd Khovanov chain complexes for these pretzel links that cannot be explained by the construction.''

\begin{figure}[H]
    \centering
    \begin{NiceTabular}{rIc|c|c|c|c|c|c} 
	& -6 & -5 & -4 & -3 & -2 & -1 & 0 \\\boldhline
	0   & & & & & & & $\ZZ^2$ \\\hline
	-2  & & & & & & $\ZZ$ & \\\hline
	-4  & & & & & $\ZZ$ & & \\\hline
	-6  & & & & $\ZZ^2$ & & & \\\hline
	-8  & & & $\ZZ$ & & & & \\\hline
	-10 & & $\ZZ$ & & & & & \\\hline
	-12 & $\ZZ$ & & & & & & 
	\end{NiceTabular}
	\quad\quad
	\begin{NiceTabular}{rIc|c|c|c|c|c|c} 
	& -6 & -5 & -4 & -3 & -2 & -1 & 0 \\\boldhline
	0   & & & & & & & $\ZZ^2$ \\\hline
	-2  & & & & & & $\ZZ$ & \textbf{$\ZZ_3$} \\\hline
	-4  & & & & & $\ZZ$ & & \\\hline
	-6  & & & & $\ZZ^2$ & & & \\\hline
	-8  & & & $\ZZ$ & & & & \\\hline
	-10 & & $\ZZ$ & & & & & \\\hline
	-12 & $\ZZ$ & & & & & & 
	\end{NiceTabular}
    \caption{The even and odd reduced Khovanov homologies of $9_{46}$ on the left and right, respectively.
    The horizontal and vertical directions correspond to the homological grading and the quantum grading.}\label{fig:tables}
\end{figure}

The module structure we have developed up to this point provides a tool to explain some of the additional symmetries that appear in the odd Khovanov complexes when compared to the even setting, at the very least in this particular context.
In the following proposition, $\overline{OKh}(K)^0$ denotes the $0$th reduced odd Khovanov
homology of the $(3,3,-3)$-pretzel knot $K$. The complete even and odd reduced Khovanov homologies
of this knot are shown in Figure~\ref{fig:tables}.

\begin{proposition}\label{prop:pretzel}
There is a homology class in
$H_1(\Sigma(K);\ZZ)$ whose action 
on
 $\overline{OKh}(K)^{0}\cong\mathbb{Z}^2\oplus\mathbb{Z}/3\mathbb{Z}$
is given by a nonzero map
$\mathbb{Z}^2\rightarrow\mathbb{Z}/3\mathbb{Z}$.
\end{proposition}

Through a conference and conversations with the second author \cite{WehrliTalk}, Mark Ebert and L\'eo Schelstraete became aware of this result prior to this publication.  They have since published a generalization in \cite{ES2025}, finding $\mathbb{Z}/n\mathbb{Z}$ torsion in the odd Khovanov homology of
the $(n,n,-n)$-pretzel knot.
Unlike our original proof of the above proposition, their argument
uses
the foam version of odd Khovanov homology and does not require any computer calculuations.
Our original proof of  Proposition~\ref{prop:pretzel} with slightly abridge details appears below.

\begin{proof}[Proof of Proposition~\ref{prop:pretzel}]

In $K$, the $(3,3,-3)$-pretzel knot, we consider the arc $\alpha$ shown
in Figure~\ref{fig:pretzel}. As in this figure, let $a$ and $b$ be the edges of the link diagram of $K$ that contain the endpoints of $\alpha$,
so that the module action of $[\widehat{\alpha}]\in H_1(\Sigma(K);\ZZ)$ is given by $x_a-x_b$.

\begin{figure}[H]
\includeFig{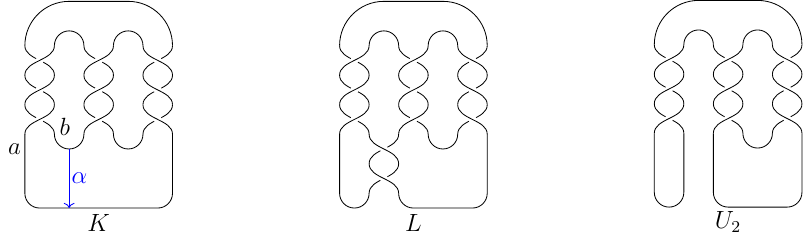}
\caption{The $(3,3,-3)$-pretzel knot $K$ with the distinguished arc $\alpha$, the link $L$, and a resolution of $L$ ambient isotopic to a
$2$-component unlink $U_2$.}\label{fig:pretzel}
\end{figure}

In Figure~\ref{fig:homotopyeq}, we illustrate a chain homotopy equivalence between the reduced odd Khovanov complex $C(L)$ and
a complex of the form $\text{Cone}(x_a-x_b)\rightarrow C(U_2)$, where $L$ and $U_2$ are as in Figure~\ref{fig:pretzel}.
The top row in Figure~\ref{fig:homotopyeq} represents the complex $C(L)$, where
each planar diagram
represents a resolution (turned sideways) of the twisted band that appears in $L$ but not in $K$.
The complex $C(L)$ is related by delooping chain isomorphisms to an intermediary complex, which admits a chain homotopy
to the final complex $\text{Cone}(x_a-x_b)\rightarrow C(U_2)$ via Gaussian elimination.
The $\epsilon_1$ and $\epsilon_2$ are the requisite linear maps produced by the sign assignment on $L$, 
and $\epsilon$ is a function of the vertex $\alpha$, given by $\epsilon=(-1)^{S(D,\alpha)}$ which produces a valid chain map from the dot and death in the odd Putyra category.  
Normally we would also need signs on the saddle cobordisms on the left hand side of the diagram, but as they are both merges we can produce a valid differential entirely with the signs $\epsilon_1$ and $\epsilon_2$.

\begin{figure}[h]
\includeFig{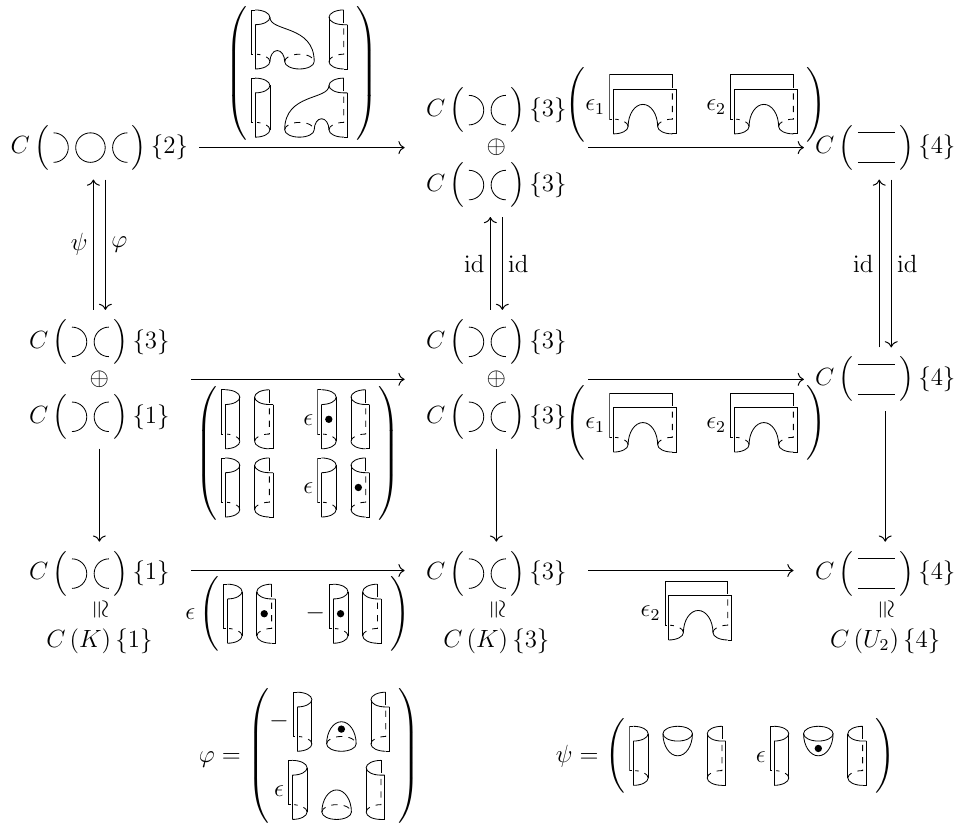}
\caption{A homotopy equivalence between the complex $C(L)$ (in the top row) and
a complex of the form $\text{Cone}(x_a-x_b)\rightarrow C(U_2)$ (in the bottom row),
where the numbers in curly brackets denote shifts of the quantum grading.}\label{fig:homotopyeq}
\end{figure}

The chain homotopy equivalence from Figure~\ref{fig:homotopyeq} gives rise to a long exact sequence
involving the homology of $\text{Cone}(x_a-x_b)$ and the reduced odd Khovanov homologies of $L$ and $U_2$.
Using this long exact sequence, one can show
\[
\operatorname{coker}(x_a-x_b)\cong\overline{OKh}(L)^{1,1}
\]
where $x_a-x_b$ is viewed as a map $\ZZ\rightarrow\ZZ/3\ZZ$ as in the proposition,
and where $\overline{OKh}(L)^{1,1}$
denotes the component of $\overline{OKh}(L)$ that lives in homological degree and quantum degree $1$.
By a computation in \texttt{KnotJob} \cite{KnotJob},
\[
\overline{OKh}(L)^{1,1}=0,
\]
whence $x_a-x_b\neq 0$.
\end{proof}

\section{Module Structure and Link Cobordisms}\label{s:cobordism}

In this section, we will related the module structure on odd Khovanov homology to maps induced by link
cobordisms and prove Theorems~\ref{thm:tubeaction} and \ref{thm:functor}.

\subsection{Maps Induced by Link Cobordisms}

Recall that  
each smooth link cobordism $F\colon L_0\rightarrow L_1$ in $\mathbb{R}^3\times I$
induces a linear map \[OKh(F)\colon OKh(L_0)\longrightarrow OKh(L_1),\]
which is
well-defined up to an overall sign, and invariant under smooth
ambient isotopies in $\mathbb{R}^3\times I$  \cite{MW2024,Spy2025}.
This map has the following properties.

\begin{lemma}\label{lem:cobproperties} Let $F$ be a smooth link cobordism.
\begin{enumerate}[label=\alph*.,ref=\alph*]
\item\label{propertyA} If $F$ is a nonempty closed surface, then $OKh(F)=0$.
\item\label{propertyB}  $OKh(F)$ can be nonzero even if $F$ contains a closed component.
\item\label{propertyC} $OKh(F)$ is not invariant under ribbon moves, defined as in \cite{Og2000}.
\item\label{propertyD} $OKh(F)$ is not invariant under smooth ambient isotopies in $S^3\times I$.
\end{enumerate}
\end{lemma}

\begin{proof}\phantom{.}\\
\textit{Part~\ref{propertyA}.} Consider the chain map $\llbracket F\rrbracket$ that $F$ induces
on the odd Khovanov bracket. If $F$ is a nonempty closed surface, then each matrix entry in this chain map
is given a linear combination of nonempty closed chronological cobordisms without dots. Part~\ref{propertyA} of the lemma
follows because all such cobordisms evaluate to zero in the dotted odd Putyra category.

\textit{Part~\ref{propertyB}.} This follows from Proposition~\ref{prop:torus} below.

\textit{Part~\ref{propertyC}.} This was already observed in \cite{MW2024}. It also follows directly from 
Theorem~\ref{thm:arcaction} together with Theorem~\ref{thm:tubeaction} below.

\textit{Part~\ref{propertyD}.} This was already shown in \cite{MW2024}. It will also be discussed
in Subsection~\ref{s:noninvariance} and explored in more detail in \cite{MW2026a}.
\end{proof}

\begin{remark} 
If the closed surface $F$ has nonzero Euler characteristic, then part~\ref{propertyA} of the lemma can
be seen more easily by using that the map $OKh(F)\colon OKh(\emptyset)\rightarrow OKh(\emptyset)$ is homogeneous of quantum degree $\chi(F)$, and $OKh(\emptyset)=\ZZ$ is supported in quantum degree zero.
\end{remark}

\begin{remark}
In general,
part~\ref{propertyA} can also be seen by writing $F$ as a composition
of link cobordisms $\emptyset\rightarrow U\rightarrow U\rightarrow\emptyset$
given by the initial birth in $F$, followed by the middle part of $F$, followed by the final death in $F$.
By its definition, the map induced by the initial birth $\emptyset\rightarrow U$ takes $OKh(\emptyset)=\ZZ$
isomorphically to the reduced
subgroup $\overline{OKh}(U)\subset OKh(U)$. The map induced by the middle part $U\rightarrow U$ takes this
subgroup to itself, and this now shows
part~\ref{propertyA} of the lemma
because the map induced by the final death $U\rightarrow\emptyset$ takes this subgroup to
zero.
\end{remark}

\begin{remark}
Contrary to part~\ref{propertyD} of the lemma,
the map $OKh(F)$ is invariant under smooth ambient isotopies in $S^3\times I$ if
$L_0$ or $L_1$ is empty.
This follows because, in this case, any smooth ambient isotopy in $S^3\times I$
can be replaced by a smooth ambient isotopy in $\mathbb{R}^3\times I$.
\end{remark}

\subsection{Coloring Modules of Link Cobordisms}

We will now extend the notion of coloring modules to link cobordisms.
Let $F\subset\mathbb{R}^3\times I$ be a smooth link cobordism represented by
a fixed movie presentation $L_0,\dots,L_n$.

\begin{definition}\label{def:ColF}
	 The coloring module of $F$ is the abelian group given by
	 \[\text{Col}(F)\defeq\left(\bigoplus_{i=0}^n\text{Col}(L_i)\right)\bigg/\mathcal{R}\]
	where $\mathcal{R}$ represents the obvious relations identifying like overstrands in $L_{i}$ and $L_{i+1}$.	
\end{definition}

Instead of representing $F$ by the movie $L_0,\ldots,L_n$, we can represent it
by its broken surface diagram $S\subset\mathbb{R}^2\times\{0\}\times I$.
By an oversheet of $S$, we will mean a maximal connected subsurface of $F$
that does not undercross any other parts of $F$ in the broken surface diagram $S$.
To $S$, we can assign a coloring module as follows:

\begin{definition}\label{def:ColS}
The coloring module of $S$ is an abelian group which is generated by all oversheets of $S$,
and which has a relation of the form $c=2b-a$ for each arc of double points in $S$,
where $a,b,c$ are the oversheets that meet along this arc. More precisely,

\begin{equation}
\text{Col}(S)\defeq\left\{\text{oversheets of }S\,\middle|\,\includeFig{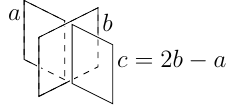}\right\}
\end{equation}

\end{definition}

It turns out that Definitions~\ref{def:ColF} and \ref{def:ColS} are equivalent, in the following sense.

\begin{lemma} There is a canonical identification {\normalfont $\text{Col}(F)=\text{Col}(S)$}.
\end{lemma}

\begin{proof} We can define a map $\text{Col}(F)\rightarrow\text{Col}(S)$ by sending
each overstrand, $a$, in the diagram of a link $L_i$ in the movie of $F$ to the oversheet that contains $a$.
This map is well-defined because the relations that hold in $\text{Col}(F)$ are also satisfied in $\text{Col}(S)$.
It is also surjective because each oversheet of $F$ must intersect at least one of the links $L_i$, and thus
each oversheet must contain an overstrand of one of the links $L_i$.

To see that the map $\text{Col}(F)\rightarrow\text{Col}(S)$ is injective, let $a'\subset L_{i'}$
and $a''\subset L_{i''}$ be two overstrands that map to the same oversheet $f$.
We can view $f$ as a connected component of $F\setminus u$, where $u\subset F$
is the set of all points on $F$ at which $F$ undercrosses an other part of $F$.
Let $\gamma\subset f$ be a path on this connected component which starts at a point of $a'\subset f$
and ends at a point of $a''\subset f$. We can then assume that $\gamma$ is a composition of shorter
paths $\gamma_1,\ldots,\gamma_r$, where each shorter path $\gamma_j\subset f$ satisfies one of the following:
$\gamma_j$ is fully contained in an overstrand of a link $L_{i_j}$ for some $i_j$; or $\gamma_j$
connects two corresponding overstrands in the diagrams of two consecutive links $L_{i_j}$ and $L_{i_j+1}$.
In each of these cases, $\gamma_j$ starts and ends on overstrands that represent
the same element of $\text{Col}(F)$. Thus $a'=a''$ in $\text{Col}(F)$.

To complete the proof of injectivity, we must show that every relation in $\col(S)$ comes from a relation in $\col(F)$.
To see this, note that the set $u\subset F$ considered before contains finitely exceptional points
that correspond to times $t\in I$ when a Reidemeister move occurs.
If we remove these exceptional points from $u$, then $u$ becomes
a disjoint union of embedded arcs in $F$. Suppose $f'$ and $f''$ are two oversheets that meet
along such an arc $u'\subset u$, and let $f'''$ denote the oversheet of $F$ that overcrosses $u'$.
Then there is an index $i$ such that $L_i\cap u'$ corresponds to a crossing in the link diagram of $L_i$.
Hence the defining relation in $\text{Col}(S)$ that holds among the three oversheets $f'$, $f''$, and $f'''$
corresponds to a defining relation that already holds in $\text{Col}(L_i)$, and hence in $\text{Col}(F)$.
This concludes the proof of injectivity.
\end{proof}

The coloring module of $F$ has the following topological interpretation.

\begin{lemma}\label{lem:ColFH}
There is a canonical identification
{\normalfont $\text{Col}(F)=H_1(\Sigma(F\cup\operatorname{id}_U);\ZZ)$}.
\end{lemma}

To prove this lemma, we will denote by $F_i\subset F$ be the elementary link cobordism that occurs between
the links $L_{i-1}$ and $L_i$ in the movie presentation of $F$. We will first show:

\begin{lemma}\label{lem:ColFiH}
There is a canonical identification
 {\normalfont $\text{Col}(F_i)=H_1(\Sigma(F_i\cup\operatorname{id}_U);\ZZ)$}.
\end{lemma}

In the subsequent proofs, we will assume that all homology groups are defined over integer coefficients.

\begin{proof}[Proof of Lemma~\ref{lem:ColFiH}]
Suppose first that the link cobordism $F_i$ corresponds to a Reidemeister move or a planar isotopy. Then
the inclusions of $\Sigma(L_{i-1}\cup U)$ and $\Sigma(L_i\cup U)$ into
$\partial\Sigma(F_i\cup \operatorname{id}_U)\subset\Sigma(F_i\cup \operatorname{id}_U)$ induce canonical identifications
\[
H_1(\Sigma(L_{i-1}\cup U))= H_1(\Sigma(F_i\cup\operatorname{id}_U))= H_1(\Sigma(L_i\cup U)),
\]
in which the first and the last homology group correspond to $\text{Col}(L_{i-1})$
and $\text{Col}(L_i)$, respectively. We can therefore write the middle homology group as
\[
H_1(\Sigma(F_i\cup\operatorname{id}_U))=\bigl(\text{Col}(L_{i-1})\oplus\text{Col}(L_i)\bigr)\big/\mathcal{R}_i,
\]
where $\mathcal{R}_i$ denotes the relation that identifies $\text{Col}(L_{i-1})$ with $\text{Col}(L_i)$.
The lemma now follows for this $F_i$ because $\text{Col}(F_i)$ is equal to
$\bigl(\text{Col}(L_{i-1})\oplus\text{Col}(L_i)\bigr)\big/\mathcal{R}_i$, by definition.

Now suppose that the link cobordism $F_i\colon L_{i-1}\rightarrow L_i$ is an elementary Morse cobordism.
More specifically, suppose that
the diagrams of the links $L_{i-1}$ and $L_i$ are identical except in a small closed disk neighborhood
$A\subset\mathbb{R}^2$, where they differ by birth, death, or saddle move.
Suppose $t_{i-1}<t_i$ are such that $F_i=F\cap\{t\in[t_{i-1},t_i]\}$.

We can then assume that the part of $F_i\subset\mathbb{R}^3\times[t_{i-1},t_i]$ where $F_i$ is not parallel to the $[t_{i-1},t_i]$-direction
is contained in $B\times J$
for $B:=A\times [-1,1]$, and for a subinterval $J:=[r,s]\subset(t_{i-1},t_i)$. Let 
$P$ denote the preimage of $B\times J$ in $\Sigma(F_i\cup\operatorname{id}_U)$,
and $W$ denote the complement
$W:=\Sigma(F_i\cup\operatorname{id}_U)\setminus\operatorname{int}(P)$.
Then
$W$ can be written as
\begin{equation}\label{eqn:W}
W=\bigl(\Sigma(L_{i-1}\cup U)\times[t_{i-1},r]\bigr)\cup\bigl(\Sigma(T\cup U)\times[r,s]\bigr)\cup\bigl(\Sigma(L_i\cup U)\times[s,t_i]\bigr),
\end{equation}
where $T$ denotes the tangle $T:=L_i\setminus B=L_{i-1}\setminus B$,
and $\Sigma(T\cup U)$ denotes the branched double-cover of $S^3\setminus B$, branched along $T\cup U$.

In \eqref{eqn:W}, we can thicken the first and the last set that appears the union by
adding the interior of the middle set to each of them.
Then \eqref{eqn:W} implies $W=V_{i-1}\cup V_i$, where $V_{i-1}$ and $V_i$ denote the resulting thickened sets.
Since the intersection $V_{i-1}\cap V_i$ is equal to $\Sigma(T\cup U)\times (r,s)$, which is path connected, a Mayer-Vietoris argument
shows that $H_1(W)$ is isomorphic to a quotient of the sum $H_1(V_{i-1})\oplus H_1(V_i)$
by relations coming from elements of $H_1(V_{i-1}\cap V_i)$.
Moreover, we have $H_1(V_{i-1})=H_1(\Sigma(L_{i-1}\cup U))=\text{Col}(L_{i-1})$ and
$H_1(V_i)=H_1(\Sigma(L_i\cup U))=\text{Col}(L_i)$, and therefore
\[
H_1(W)=\bigl(\text{Col}(L_{i-1})\oplus\text{Col}(L_i)\bigr)\big/\mathcal{R}_i=\text{Col}(F_i),
\]
where $\mathcal{R}_i$ is the relation that identifies corresponding generators in $\text{Col}(L_{i-1})$ and
$\text{Col}(L_i)$. To complete the proof, we will show that $H_1(\Sigma(F_i\cup\operatorname{id}_U))=H_1(W)$.
Note that
\[
\Sigma(F_i\cup\operatorname{id}_U)=W\cup P
\]
for $P$ as before. In particular,
$P$ is the branched double-cover
of $B\times J$, branched along $F_i':=F_i\cap (B\times J)$.
One can check that the pair $(B\times  J,F_i')$ is
abstractly homeomorphic to the pair $(D^2\times D^2,D^2\times\{(0,0)\})$.
The branched double-cover $\Sigma(F_i')=P$ is therefore itself a $4$-ball,
and gluing this $4$-ball to $W$ along the full boundary of $P$ does not add
any generators or relations to $H_1(W)$.
\end{proof}

Thus, we have shown that there is an identification $\col(F)=H_1(\Sigma(F);\ZZ)$ in the case where $F$ is
an elementary link cobordism. We can now use this to establish such an identification for a general link cobordism $F$.

\begin{proof}[Proof of Lemma~\ref{lem:ColFH}]
Let $F_{\leq i}:=F_1\circ\ldots\circ F_i$ denote the part of $F$ that lies
between $L_0$ and $L_i$. To prove that $\col(F)=H_1(\Sigma(F))$, we show
inductively that
\[
\text{Col}(F_{\leq i})=H_1(\Sigma(F_{\leq i}\cup\operatorname{id}_U))
\]
for all $i$.
For $i=1$, this follows immediately from Lemma~\ref{lem:ColFiH}. For $i>1$, we have
$F_{\leq i}=F_i\circ F_{\leq i-1}$
and hence
\[
\Sigma(F_{\leq i}\cup\operatorname{id}_U)=
\Sigma(F_i\cup \operatorname{id}_U)\cup
\Sigma(F_{\leq i-1}\cup \operatorname{id}_U).
\]
Since $\Sigma(L_{i-1}\cup U)$ is path connected, a Mayer-Vietoris argument now implies
\[
H_1(\Sigma(F_{\leq i}\cup\operatorname{id}_U))=\bigl(H_1(\Sigma(F_i\cup\operatorname{id}_U))\oplus H_1(\Sigma(F_{\leq i-1}\cup\operatorname{id}_U)\bigr)\big/\!\sim,
\]
where $\sim$ identifies the images of elements of $H_1(\Sigma(L_{i-1}\cup U))$ in the two summands.
Using Lemma~\ref{lem:ColFiH} and the induction hypothesis, we can further identify each of these summands with a coloring module.
This leads to
\[
H_1(\Sigma(F_{\leq i}\cup\operatorname{id}_U))=\bigl(\text{Col}(F_i)\oplus\text{Col}(F_{\leq i-1})\bigr)\big/\mathcal{R}',
\]
where $\mathcal{R}'$ denotes the relation that identifies the images of elements of $\text{Col}(L_{i-1})$ in
$\text{Col}(F_i)$ and in $\text{Col}(F_{\leq i-1})$.
This completes the proof of the inductive step because the resulting quotient group is precisely $\text{Col}(F_{\leq i})$.\end{proof}

\begin{remark}\label{rem:HSigmaFone}
If $F$ is nonempty, then there is a non-canonical isomorphism
$H_1(\Sigma(F\cup\operatorname{id}_U);\ZZ)\cong H_1(\Sigma(F);\ZZ)\oplus\mathbb{Z}$.
This follows from a geometric argument similar to the one described in
Remark~\ref{rem:doublecoverLU}.
\end{remark}

\begin{remark}\label{rem:HSigmaFtwo}
If $F$ is a 2-knot, then $H_1(\Sigma(F);\ZZ)$ is odd torsion.
Note that, because of the Universal Coefficient Theorem, this is equivalent
to saying that $H_1(\Sigma(F);\mathbb{F}_2)=0$.
To show the latter, one can first prove $H_1(\Sigma(F),\widetilde{F};\mathbb{F}_2)=0$
by using the long exact sequence
$\ldots\rightarrow H_1(S^3\times I,F;\mathbb{F}_2)\rightarrow H_1(\Sigma(F),\widetilde{F};\mathbb{F}_2)\rightarrow H_1(S^3\times I,F;\mathbb{F}_2)\rightarrow\ldots$, where
$\widetilde{F}\subset\Sigma(F)$ denotes the preimage of the $2$-knot $F\subset S^3\times I$ in the branched double-cover of $S^3\times I$.
\end{remark}

\subsection{Superdegree of a Link Cobordism}\label{subs:superdecorated}

In the following, let $F\subset\mathbb{R}^3\times I$ be a smooth oriented link cobordism.
Since $F$ is oriented, each saddle that appears between two consecutive frames in a movie presentation for $F$
either merges two link components into one, or splits a link component into two.
We can therefore assign a superdegree to $F$ by repeating the definition that
we used for chronological cobordisms in Subsection~\ref{subs:chronological}.
Explicitly, the superdegree of $F$ is defined as the element $|F|\in\ZZ_2$ given by the modulo $2$ reduction of
\[\#\mbox{deaths}-\#\mbox{splits}=(\chi-D)/2,\]
where $\chi=\chi(F)$ denotes the Euler characteristic of $F$, and $D$ denotes
the difference between the number of top boundary components of $F$ and the
number of bottom boundary components of $F$.

Now suppose that the oriented link cobordism $F\colon L\rightarrow L'$
contains a vertical segment $\{p\}\times I\subset\mathbb{R}^3\times I$, and that this segment does not
meet any double or triple points in the broken surface diagram for $F$.
Let $e$ and $e'$ denote the edges of the link diagrams of $L$ and $L'$ that contain the point $p$, respectively.
Then we have the following lemma, where $|F|$ denotes the superdegree defined above.

\begin{lemma}\label{lem:xcommute}
$x_{e'}\circ OKh(F)=(-1)^{|F|}OKh(F)\circ x_e$.
\end{lemma}

\begin{proof} We will prove this directly for the map that $F$ induces on the formal Khovanov bracket.
Recall that the formal Khovanov bracket admits a graded refinement, in which each
resolution $D_\alpha$ in the resolution cube comes with a formal shift of the supergrading.
In this refinement, the shift of the superdegree at the vertex $\alpha$ is given by the modulo $2$
reduction $s(\alpha)\in\ZZ_2$ of the quantity $S(D,\alpha)\in\ZZ$ from \eqref{eqn:Smap}.
As in \eqref{eqn:xedef}, let
 $(x_e)_\alpha$ denote the restriction of the dot chain map $x_e\colon\llbracket D\rrbracket\rightarrow\llbracket D\rrbracket$ to $D_\alpha$.
For simplicity, we will henceforth drop the subscript $\alpha$ and simply write $x_e$ for this restriction.
If $D_\alpha\langle s\rangle$ denotes the resolution $D_\alpha$ together its assigned
shift of the supergrading by $s=s(\alpha)\in\ZZ_2$, we can then write $x_e$ as
\begin{equation}\label{eqn:xu}
x_e=(-1)^s u_e,
\end{equation}
where $u_e$ denotes the underlying dotted identity cobordism.
Now consider a component
\[
f\colon D_\alpha\langle s\rangle\longrightarrow D'_{\alpha'}\langle s'\rangle
\]
of the chain map $\llbracket F\rrbracket\colon\llbracket L\rrbracket\rightarrow\llbracket L'\rrbracket$, where $D_{\alpha}$ and $D'_{\alpha'}$
denote resolutions of the diagrams of $L$ and $L'$, respectively, and $s=s(\alpha)$ and $s'=s(\alpha')$
are their associated shifts of the supergrading.
Let $|f|$ denote the superdegree of $f$, viewed as a morphism
between the unshifted objects $D_\alpha$ and $D'_{\alpha'}$.
We then have
\begin{align*}
x_{e'}\circ f &= (-1)^{s'}u_{e'}\circ f\\
&=(-1)^{|f|+s'}f\circ u_e\\
&=(-1)^{|f|+s'+s}f\circ x_e,
\end{align*}
where the second equality follows because $u_e$ has odd superdegree in the
odd dotted Putyra category, and the other equalities follow from \eqref{eqn:xu}.
Note that, in the last line, $|f|+s'+s$ is precisely the superdegree
of $f$ after taking into account the formal grading shifts. 
It was observed in \cite{MW2024} that this superdegree coincides with the superdegree of the
link cobordism $F$,
and the lemma thus follows.
\end{proof}

\subsection{Decorated Link Cobordisms}\label{subs:decorated}

We will now complete the definition of the decorated link cobordism category 
$\mathcal{C}\mathit{ob}^4_\Lambda$ from the introduction.
Recall that the objects in this category are generic oriented links $L\subset\mathbb{R}^3$,
and the morphisms are equivalence classes of pairs $(F,c)$, where $F$ is a generic
smooth oriented link cobordism in $\mathbb{R}^3\times I$, and $c$ is an element
\[
c\in\Lambda^*H_1(\Sigma(F\cup\operatorname{id}_U);\ZZ).
\]
As discussed earlier, we are assuming that the branched double-cover $\Sigma(F\cup\operatorname{id}_U)$ is equipped with a
distinguished preimage of the basepoint $(x_0,1/2)\in S^3\times I$, so that it is unique up to unique based isomorphism of branched covering spaces. The composition in $\mathcal{C}\mathit{ob}^4_\Lambda$ is given by
\begin{equation}\label{eqn:composition}
(F,c)\circ (F',c'):=(F\circ F',(-1)^{|c||F'|}c\wedge c'),
\end{equation}
where $c$ is assumed to be wedge product of $|c|$ elements of
$H_1(\Sigma(F\cup\operatorname{id}_U);\ZZ)$.
To complete the definition of $\mathcal{C}\mathit{ob}^4_\Lambda$,
we must specify the equivalence relation that we impose on pairs $(F,c)$.
This will occupy the rest of this subsection.

We start by fixing a $3$-ball $B\subset\mathbb{R}^3$, and we will assume throughout the following
discussion that the link cobordism $F\subset F\cup\operatorname{id}_U$ is contained in $B\times I$,
while the identity cobordism of $U$ is contained in $(S^3\setminus B)\times I$.
Let
\[
D(S^3\times I)\subset\operatorname{Diff}^+(S^3\times I)
\]
denote the set of all diffeomorphisms $\phi\colon S^3\times I\rightarrow S^3\times I$ that are
isotopic to the identity of $S^3\times I$ via an isotopy that fixes all points of $((S^3\setminus B)\times I)\cup V_\phi$
for a neighborhood $V_\phi$ of $S^3\times\partial I$.
Note that if $F$ is as above, then any $\phi\in D(S^3\times I)$ lifts to a unique map
$\widetilde{\phi}\colon\Sigma(F\cup\operatorname{id}_U)\rightarrow\Sigma(\phi(F)\cup\operatorname{id}_U)$
of based manifolds. Moreover, this map $\widetilde{\phi}$ induces an algebra isomorphism
$\widetilde{\phi}_*\colon\Lambda^*H_1(\Sigma(F\cup\operatorname{id}_U);\ZZ)\rightarrow\Lambda^*H_1(\Sigma(\phi(F)\cup\operatorname{id}_U);\ZZ)$.
We now impose the equivalence relations
\begin{equation}\label{eqn:equiv}
(F,c)\sim (\phi(F),\widetilde{\phi}_*(c))
\end{equation}
on morphisms of  $\mathcal{C}\mathit{ob}^4_\Lambda$,
where $\phi$ runs through all diffeomorphisms $\phi\in D(S^3\times I)$.
After imposing this relation, we can interpret
the second entry in the pair $(F,c)$ as an element of the quotient set
\begin{equation}\label{eqn:quotientset}
\bigl(\Lambda^*H_1(\Sigma(F\cup\operatorname{id}_U);\ZZ)\bigr)\big/ \bigl\{c\sim\widetilde{\phi}_*(c)\bigr\},
\end{equation}
where $\phi$ runs through all $\phi\in D(S^3\times I)$ with $\phi(F)=F$.
Although this quotient set is in general no longer an additive group (let alone an algebra),
the composition of morphisms in $\mathcal{C}\mathit{ob}^4_\Lambda$ remains well-defined.
In fact, if $\phi$ and $\phi'$ are two diffeomorphisms in $D(S^3\times I)$,
and $F$ and $F'$ are as in \eqref{eqn:composition}, then
\[
\phi(F)\circ\phi'(F')=\psi(F\circ F'),
\]
where $\psi:=\phi\cup\phi'\in D(S^3\times I)$ denotes the diffeomorphism obtained by
stacking $\phi$ and $\phi'$ vertically.
If $c$ and $c'$ are as in \eqref{eqn:composition}, then this further implies
\[
\widetilde{\phi}_*(c)\wedge\widetilde{\phi'}_*(c')=\widetilde{\psi}_*(c\wedge c').
\]
It follows that $(\phi(F),\widetilde{\phi}_*(c))\circ(\phi'(F'),\widetilde{\phi'}_*(c'))$ is equal to
\[(\psi(F\circ F'),(-1)^{|c||F'|}\widetilde{\psi}_*(c\wedge c'))\sim(F\circ F',(-1)^{|c||F'|}c\wedge c')=(F,c)\circ(F',c'),
\]
which proves that the composition in $\mathcal{C}\mathit{ob}^4_\Lambda$ is compatible with
\eqref{eqn:equiv}.

\begin{remark}\label{rem:groupaction}
The quotient set from \eqref{eqn:quotientset} can be viewed as an orbit set with respect to the 
action of the group of all $\phi\in D(S^3\times I)$ with $\phi(F)=F$.
In certain special cases, this group action is trivial.
In particular, this happens whenever. the map
$H_1(\partial\Sigma(F\cup\operatorname{id}_U);\ZZ)\rightarrow H_1(\Sigma(F\cup\operatorname{id}_U);\ZZ)$ is surjective.
In this case, each $\phi\in D(S^3\times I)$ with $\phi(F)=F$ acts as the identity on
$\Lambda^*H_1(\Sigma(F\cup\operatorname{id}_U);\ZZ)$
because it restricts to the identity on $S^3\times\partial I$, by the definition of $D(S^3\times I)$.
\end{remark}

\subsection{Odd Khovanov Maps for Decorated Link Cobordisms}\label{subs:oddmaps}

We will now prove the following theorem from the introduction.

\TheoremFunctor*

In this subsection, we will define the map $OKh(F,c)\colon OKh(L)\rightarrow OKh(L')$ assigned to a decorated link cobordism $(F,c)\colon L\rightarrow L'$.
We will also show that the assignment $(F,c)\mapsto OKh(F,c)$ is compatible
with the composition of morphisms in \mbox{$\mathcal{C}\mathit{ob}^4_\Lambda$} (Lemma~\ref{lem:composition}), and that the
module structure on odd Khovanov homology is natural (). In the next subsection,
we will complete the proof of Theorem~\ref{thm:functor} by showing that the assignment $(F,c)\mapsto OKh(F,c)$ is also compatible with the
equivalence relation~\eqref{eqn:equiv} (Lemma~\ref{lem:OKhFcinvariance}).
For simplicity, we will assume that type Y sign assignments are used, but all of
our arguments can be adapted to the type X setting.

Consider a smooth oriented link cobordism $F\colon L\rightarrow L'$ presented by a movie
$L_0,\ldots,L_n$, and et $F_i$ dentoe the part of $F$ that lies between
$L_{i-1}$ and $L_i$ so that $F=F_n\circ\cdots\circ F_1$. 
Let $F_{\leq i}:=F_i\circ\cdots\circ F_1$.
For
an edge $e$ of the link diagram of $L_i$, let
$z_e\colon OKh(L_i)\rightarrow OKh(L_i)$ denote the map given by
\begin{equation}\label{eqn:zedef}
z_e:=(-1)^{|F_{\leq i}|}x_e,
\end{equation}
where $|F_{\leq i}|$ denotes the superdegree of $F_{\leq i}$.
Now suppose that $e$ and $e'$ are corresponding edges
of the link diagrams of $L_{i-1}$ and $L_i$, respectively, in the sense that
these two link diagrams agree near a point $p\in\mathbb{R}^2$
that lies on both $e$ and $e'$.
Then:

\begin{lemma}\label{lem:zcommute}
$z_{e'}\circ OKh(F_i)=OKh(F_i)\circ z_e$.
\end{lemma}

\begin{proof} We have
\begin{align*}
z_{e'}\circ OKh(F_i) &= (-1)^{|F_{\leq i}|}\,x_{e'}\circ OKh(F_i)\\
&=(-1)^{|F_{\leq i}|+|F_i|}\,OKh(F_i)\circ x_e\\
&=(-1)^{|F_{\leq i-1}|}\,OKh(F_i)\circ x_e\\
&=OKh(F_i)\circ z_e,
\end{align*}
where the second equation follows from Lemma~\ref{lem:xcommute}, and
the third equation follows because $F_{\leq i}=F_i\circ F_{\leq i-1}$, and because the superdegree is additive under
composition.
\end{proof}

To show that odd Khovanov homology extends to a functor on
$\mathcal{C}\mathit{ob}^4_\Lambda$, we must assign
a map
\[
OKh(F,c)\colon OKh(L_0)\longrightarrow OKh(L_n)
\]
to each  $c\in\Lambda^*H_1(\Sigma(F\cup\operatorname{id}_U);\ZZ)$
where $F=F_n\circ\cdots\circ F_1$ is as above.
Note that each of the maps
\[
OKh(F_i)\colon OKh(L_{i-1})\longrightarrow OKh(L_i).
\]
is well-defined up to an overall sign.
To start the construction of $OKh(F,c)$, we fix a sign for each of these maps,
where this sign choice is independent of $c$.
Using Lemma~\ref{lem:ColFH}, we can identify $\Lambda^*H_1(\Sigma(F\cup\operatorname{id}_U);\ZZ)$ with $\Lambda^*\col(F)$,
which is isomorphic to 
\begin{equation}\label{eqn:tensorCol}
\Lambda^*\Bigl[\bigl(\text{Col}(L_n)\oplus\cdots\oplus\text{Col}(L_0)\bigr)\big/\!\!\sim\Bigr]
\,\,\,=\,\,\,
\Bigl[\bigl(\Lambda^*\text{Col}(L_n)\bigr)\otimes\cdots\otimes\bigl(\Lambda^*\text{Col}(L_0)\bigr)\Bigr]\bigg/\!\!\sim
\end{equation}
where $\sim$ identifies corresponding generators in $\text{Col}(L_{i-1})$ and $\text{Col}(L_i)$.
In particular, this shows that $\Lambda^*\col(F)$ is linearly generated by elements $c\in\Lambda^*\col(F)$ of the form
\begin{equation}\label{eqn:elementc}
c=e_{I_n}\otimes\cdots\otimes e_{I_1}
\end{equation}
where $I_i$ denotes a collection of overstrands $e_{i,1},\ldots,e_{i,\ell_i}$ of the link diagram representing $L_i$,
and $e_{I_i}\in\Lambda^*\col(L_i)$ denotes the wedge product
$e_{I_i}:=e_{i,1}\wedge\cdots\wedge e_{i,\ell_i}$. 
By abuse of notation, we will also write $e_{i,j}$ for any edge that belongs to the overstrand $e_{i,j}$.
For $c$ as in \eqref{eqn:elementc}, we now define the map $OKh(F,c)$ as the composition
\begin{equation}\label{eqn:OKhFc}
z_{I_n}\circ OKh(F_n)\circ z_{I_{n-1}}\circ OKh(F_{n-1})\circ z_{I_{n-2}}\circ\cdots\circ z_{I_1}\circ OKh(F_0)\circ z_{I_0},
\end{equation}
where $z_{I_i}:=z_{e_{i,1}}\circ\cdots\circ z_{e_{i,\ell_i}}$ and where we are assuming type Y sign assignments.
Note that this definition is compatible with the equivalence relation $\sim$ in \eqref{eqn:tensorCol} because
of Lemma~\ref{lem:zcommute}.
We complete the construction of $OKh(F,c)$ by extending the above definition linearly to arbitrary elements
of $\Lambda^*\col(F)$.

Before proving that the assignment $(F,c)\mapsto OKh(F,c)$ is compatible with the composition in $\mathcal{C}\mathit{ob}^4_\Lambda$,
we will deduce the following naturality result for the module structure on odd Khovanov homology.

\begin{lemma}\label{lem:intertwine} Suppose $c_i$ for $i\in\{0,n\}$ are elements of $\Lambda^*H_1(\Sigma(L_i\cup U);\ZZ)$
that correspond to the same element $c\in\Lambda^*H_1(\Sigma(F\cup\operatorname{id}_U);\ZZ)$.
Then the module actions of $c_0$ and $c_n$ are related by
\begin{equation}\label{eqn:intertwine}
OKh(F)\circ c_0=(-1)^{|c||F|}c_n\circ OKh(F)
\end{equation}
where $|c|$ denotes the wedge power of $c$, and $|F|$ denotes the superdegree of $F$.
\end{lemma}

\begin{proof}
For simplicity, assume that $c_0$ corresponds to an element $e_{I_0}\in\Lambda^{|c|}\col(L_0)$
and $c_n$ corresponds to an element $e_{I_n}\in\Lambda^{|c|}\col(L_n)$, where notations
are as above. Then $c_0$ acts on $OKh(L_0)$ by $x_{I_0}:=x_{e_{0,1}}\circ\cdots\circ x_{e_{0,|c|}}$,
and $c_n$ acts on $OKh(L_n)$ by $x_{I_n}:=x_{e_{n,1}}\circ\cdots\circ x_{e_{n,|c|}}$.
By the definition of $z_e$, we have $z_{I_0}=x_{I_0}$ and $z_{I_n}=(-1)^{|c||F|}x_{I_n}$, and hence
\begin{align*}
(-1)^{|c||F|}x_{I_n}\circ OKh(F) &=z_{I_n}\circ OKh(F)\\
&=OKh(F,c)\\
&= OKh(F)\circ z_{I_0}\\
&=OKh(F)\circ x_{I_0},
\end{align*}
where the second and the third equality follow from the well-definedness of $OKh(F,c)$,
and from the assumption that $c_0$ and $c_n$ both correspond to $c$.
\end{proof}

\begin{lemma}\label{lem:composition}
The assignment $(F,c)\mapsto OKh(F,c)$ is compatible with the composition in $\mathcal{C}\mathit{ob}^4_\Lambda$.
\end{lemma}

\begin{proof} Let $(F,c)$ and $(F',c')$ be as in \eqref{eqn:composition}, and assume that
$F$ and $F'$ are presented by movies 
$L_0,\ldots,L_n$ and $L'_0,\ldots,L'_{n'}$, respectively.
Suppose also that we have fixed signs for all of the elementary maps $OKh(F_i)$ and $OKh(F'_j)$
that come from these movie presentations.
If $c$ and $c'$ correspond to tensor products as in \eqref{eqn:elementc},
then $OKh(F,c)$ and $OKh(F',c')$ have the form
\[
OKh(F,c)=z_{I_n}\circ OKh(F_n)\circ z_{I_{n-1}}\circ\cdots\circ z_{I_1}\circ OKh(F_0)\circ z_{I_0}
\]
and
\[
OKh(F',c')=z_{I'_{n'}}\circ OKh(F'_{n'})\circ z_{I'_{n'-1}}\circ\cdots\circ z'_{I_1}\circ OKh(F'_0)\circ z_{I'_0}.
\]
We can obtain a movie presentation $L'_0,\ldots,L'_{n'-1},L_0,\ldots,L_n$
for the composed link cobordism $F\circ F'$ by concatenating the movie
presentations for $F$ and $F'$ (where have used that $L'_{n'}=L_0$).
If $c$ is en element of $\Lambda^{|c|}H_1(\Sigma(F\cup\operatorname{id}_U);\ZZ)$, it thus follows that
\begin{multline*}
OKh(F\circ F',c\wedge c')=(-1)^{|c||F'|}z_{I_n}\circ OKh(F_n)\circ z_{I_{n-1}}\circ\cdots\circ OKh(F_0)\circ z_{I_0}\\
\circ z_{I'_{n'}}\circ OKh(F'_{n'})\circ z_{I'_{n'-1}}\circ\cdots\circ OKh(F'_0)\circ z_{I'_0},
\end{multline*}
where the sign in this formula arises because
of the sign in the definition of $z_e$.
In more detail, the map $z_{I_i}$ that appears in the definition of $OKh(F,c)$
is given by $z_{I_i}=z_{e_{i,1}}\circ\cdots\circ z_{e_{i,\ell_i}}$ where
$z_{e_{i,j}}=(-1)^{|F_{\leq i}|}x_{e_{i,j}}$.
To obtain the corresponding map that appears in the definition of $OKh(F\circ F',c\wedge c')$,
we must replace each $z_{e_{i,j}}$ by
\[
(-1)^{|F_{\leq i}\circ F'|}x_{e_{i,j}}=(-1)^{|F_{\leq i}|+|F'|}x_{e_{i,j}}=(-1)^{|F'|}z_{e_{i,j}}.
\]
The sign in the previous formula now arises because there are $|c|$ such maps $z_{e_{i,j}}$ when
all of the $z_{I_i}$ are considered.
In conclusion, the previous formula shows that
\[
OKh(F\circ F',c\wedge c')=(-1)^{|c||F'|}OKh(F,c)\circ OKh(F',c'),
\]
which is consistent with formula \eqref{eqn:composition} for the composition of
morphisms in $\mathcal{C}\mathit{ob}^4_\Lambda$.
\end{proof}

\subsection{Isotopy Invariance of the Odd Khovanov Maps}\label{subs:isotopyinvariance}

To complete the proof of Theorem~\ref{thm:functor}, we must show:

\begin{lemma}\label{lem:OKhFcinvariance} The definition of $OKh(F,c)$ is compatible with relation \eqref{eqn:equiv}. That is,
\[OKh(F,c)=\pm OKh(\phi(F),\widetilde{\phi}_*(c))\]
for all $\phi\in D(S^3\times I)$.
\end{lemma}

We will first show the following preliminary result:

\begin{lemma}\label{lem:CS} Let $F\colon L_0\rightarrow L_1$ be a link cobordism that corresponds to one of the two sides
of a Carter-Saito movie move. Then the map $H_1(\partial\Sigma(F);\ZZ)\rightarrow H_1(\Sigma(F);\ZZ)$
is surjective.
\end{lemma}

\begin{proof} Every movie that shows up in a Carter-Saito movie move is either given by a sequence of Reidemeister
moves, or by a sequence of Reidemeister moves together with a single birth, death, or saddle move.
While Reidemeister moves preserve the isotopy class of a link in $\mathbb{R}^3$, Morse critical points on a link cobordism
correspond to handle attachments in the branched double-cover (see, e.g., \cite[Theorem~7.1]{OS2022}). 
For $F\colon L_0\rightarrow L_1$ is as in a lemma, the branched double-cover, $\Sigma(F)$, is
therefore abstractly homeomorphic to one of the following:
\begin{itemize}
\item
$\Sigma(L_0)\times I$,
\item
$\Sigma(L_0)\times I$ with a single handle attached to $\Sigma(L_0)\times\{1\}$,
\end{itemize}
If $\Sigma(F)$ is homeomorphic to $\Sigma(L_0)\times I$, then the statement of the lemma is obvious,
and in the second case, it follows because of how handle attachments affect the first homology.
\end{proof}

\begin{proof}[Proof of Lemma~\ref{lem:OKhFcinvariance}]
Let $\phi\in D(S^3\times I)$ be a diffeomorphism as in \eqref{eqn:equiv},
and let $\phi_t$ be an isotopy between $\phi_0=\operatorname{id}$
and $\phi_1=\phi$, such that $\phi_t$ fixes the points of a neighborhood $((S^3\setminus B)\times I)\cup V_\phi$,
as in the definition of $D(S^3\times I)$.
Let $\phi'_t\in D(S^3\times I)$ be another isotopy
such that $\phi'_t=\operatorname{id}$, and
Assume further that $\phi'_t$ satisfies $\phi'_t(F)=\phi_t(F)$ for all $t$.
Moreover, let $\phi':=\phi'_1$.
Then $\phi^{-1}\circ\phi'$ lifts
to a diffeomorphism of $\Sigma(F\cup\operatorname{id}_U)$ which is
isotopic to the identity map, and thus acts by the identity on
$H_1(\Sigma(F\cup\operatorname{id}_U);\ZZ)$. In particular, this implies $\widetilde{\phi}'_*=\widetilde{\phi}_*$.

When computing the map $\widetilde{\phi}_*$ that appears in Lemma~\ref{lem:OKhFcinvariance},
we are therefore allowed to replace $\phi$ by any $\phi'$ as above.
Concretely, this means that the relevant information about $\phi$ is entirely contained
in the family $\phi_t(F)$.
We can therefore assume that $\phi$ is the final map in an isotopy $\phi_t\in D(S^3\times I)$
that starts with the identity map, and that is induced by a sequence of the following moves on a movie presentation
for $F$:
\begin{itemize}
\item Carter-Saito movie moves,
\item moves that time-reorder two spatially distant consecutive events in a movie.
\end{itemize}
To prove Lemma~\ref{lem:OKhFcinvariance}, it is
now sufficient to prove it in the special case where $\phi_t$ is induced by just one of these moves.

Assume first that $\phi_t$ is induced by a Carter-Saito movie move.
Specifically, assume that $\phi$ takes 
$F=F_3\circ F_2\circ F_1$ to $F'=F_3\circ F_2'\circ F_1$,
where $F_1$ and $F_3$ are arbitrary, and
$F_2$ and $F_2'$
correspond
to the two sides of this movie move.

Then Lemma~\ref{lem:CS} implies that
each element $c\in\Lambda^*H_1(\Sigma(F\cup\operatorname{id}_U))$ can be expressed as a linear combination
of elements of the form $c_3\wedge c_1$ where $c_i\in\Lambda^*H_1(\Sigma(F_i\cup\operatorname{id}_U))$
for $i=1,3$. Suppose now that we have fixed movie presentations for the link cobordisms $F_1,F_2,F'_2,F_3$,
together with signs for the maps induced by the elementary cobordisms that appear in these movies. Then
\begin{equation}\label{eqn:OKhFc3c1}
OKh(F,c_3\wedge c_1)=(-1)^{|c_3||F_2\circ F_1|}OKh(F_3,c_3)\circ OKh(F_2)\circ OKh(F_1,c_1),
\end{equation}
where $|c_3|$ denotes the wedge power of $c_3$.
Since $F_2$ and $F_2'$ are isotopic, the maps $OKh(F_2)$ and $OKh(F_2')$ agree up to sign,
and thus replacing $OKh(F_2)$ by $OKh(F_2')$ in \eqref{eqn:OKhFc3c1} will at most change the sign of $OKh(F,c_3\wedge c_1)$.
Moreover, this sign change is independent of the specific element $c_3\wedge c_1$, and thus the map $OKh(F,c)$ also changes by at most a sign
when $c$ is a linear combination of such elements.
Since
\[
\widetilde{\phi}_*\colon\Lambda^* H_1(\Sigma(F\cup\operatorname{id}_U);\ZZ)\longrightarrow
\Lambda^* H_1(\Sigma(F'\cup\operatorname{id}_U);\ZZ)
\]
is induced
by the identity map on $H_1(\Sigma(F_3\cup\operatorname{id}_U);\ZZ)\oplus H_1(\Sigma(F_1\cup\operatorname{id}_U);\ZZ)$,
we further have $\widetilde{\phi}_*(c_3\wedge c_1)=c_3\wedge c_1$, and hence $\widetilde{\phi}_*(c)=c$ for any $c$.
Since we also have $\phi(F)=F'$, the above discussion now implies
\[
OKh(\phi(F),\widetilde{\phi}_*(c))=OKh(F',c)=\pm OKh(F,c),
\]
as desired.

Now suppose $\phi_t$ is induced by time-reordering two spatially distant elementary events in a movie presentation.
More specifically, suppose $F$ is given by a movie with three frames $L_0,L_1,L_2$,
and suppose that 
the diagrams of $L_0$ and $L_1$ differ in a disk-region $R_1\subset\mathbb{R}^2$,
while the diagrams of $L_1$ and $L_2$ differ in a disjoint disk-region $R_2\subset\mathbb{R}^2$.
Let $\phi_t$ be an isotopy that time-reorders the nontrivial parts of $F$ that lie above $R_1$ and $R_2$.
Then the link cobordism $F':=\phi(F)$ for $\phi:=\phi_1$ is represented by the movie $L_0,L_1',L_2$, where $L_1'$ differs from $L_0$
in $R_2$ and from $L_2$ in $R_1$. To see that Lemma~\ref{lem:OKhFcinvariance} holds for $\phi=\phi_1$, 
we now note that the coloring module
\[
\col(F)=\bigr(\col(L_2)\oplus\col(L_1)\oplus\col(L_0)\bigl)\big/\mathcal{R}
\]
is generated by $\col(L_2)\oplus\col(L_0)$ because, at every point, $L_1$ agrees with at least one of $L_0$ and $L_2$.
This means that the map
$H_1(\partial\Sigma(F\cup\operatorname{id}_U);\ZZ)\rightarrow H_1(\Sigma(F\cup\operatorname{id}_U);\ZZ)$ is surjective,
and the same holds true if we replace $F$ by the isotopic link cobordism $F'$.
To complete the proof,
we can therefore argue as in the case where $\phi_t$ was an isotopy induced by a Carter-Saito movie move.
\end{proof}

Note that the sign in Lemma~\ref{lem:OKhFcinvariance} depends on the
chosen signs for the elementary cobordism maps that arise from
the movie presentations of $OKh(F)$ and $OKh(\phi(F))$.
However, we have:

\begin{lemma}\label{lem:signindependent} The sign in Lemma~\ref{lem:OKhFcinvariance} is independent of $c$.
\end{lemma}

\begin{proof} This follows directly from the proof of Lemma~\ref{lem:OKhFcinvariance}
because the sign in equation \eqref{eqn:OKhFc3c1} is independent of $c_3\wedge c_1$.
\end{proof}

Using Lemma~\ref{lem:signindependent},
Lemma~\ref{lem:OKhFcinvariance} can thus be strengthened and reformulated as follows:

\begin{corollary}\label{cor:OKhFcinvariance}
For any given link cobordism $F\colon L\rightarrow L'$, the linear map
\[
OKh(F,-)\colon\Lambda^*H_1((F\cup\operatorname{id}_U);\ZZ)
\longrightarrow\operatorname{Hom}_{\ZZ}\bigl(OKh(L),OKh(L')\bigr)
\]
satisfies
$OKh(\phi(F),-)\circ\widetilde{\phi}_*=\pm OKh(F,-)$
for all $\phi\in D(S^3\times I)$.
\end{corollary}

\begin{remark}
Assuming that we restrict ourselves to diffeomorphisms $\phi\in D(S^3\times I)$ with $\phi(F)=F'$ for a fixed $F'$,
it is reasonable to ask whether the sign in this corollary depends on $\phi$.
It turns out that it does not depend on $\phi$ if $2OKh(F)\neq 0$.
Indeed, in this case, the map $OKh(F)$ has a well-defined sign, and the sign that appears
in the corollary is just the sign that occurs in the equation $OKh(\phi(F))=\pm OKh(F)$.
Note that if $H_1(\partial\Sigma(F\cup\operatorname{id}_U);\ZZ)\rightarrow H_1(\Sigma(F\cup\operatorname{id}_U);\ZZ)$
is surjective, then the map $\widetilde{\phi}_*$ is itself independent of $\phi$,
by an argument similar to the one used in Remark~\ref{rem:groupaction}.
\end{remark}

\subsection{Dotted Link Cobordisms}\label{subs:dottedlinkcobordisms}

Let $S$ be a broken surface diagram
representing a link cobordism $F\subset\mathbb{R}^3\times I$.
By a configuration of dots on $S$, we will mean a collection of at most finitely many distinct
dots placed on the surface $S$, such that each dot lies in the interior of an oversheet of $S$,
and such that no two dots occur at the same time-coordinate.

\begin{definition}\label{def:DotS}
Let $\operatorname{Dot}(S)$ be the free abelian group generated by all configurations of dots on $S$,
modulo the following relations:
\begin{enumerate}
\item\label{rel:freely}
Dots can move freely along the oversheets of $S$.
\item\label{rel:sign}
If two dots in a configuration get moved past each other in time-direction,
then the configuration changes its sign.
\item\label{rel:zero}
If two dots lie on the same oversheet of $S$, then the configuration is equal to zero.
\item\label{rel:adjacent}
For adjacent oversheets of $S$, the relation from Figure~\ref{fig:migration} holds.
\end{enumerate}
\begin{figure}[h]
\begin{center}
\includeFig{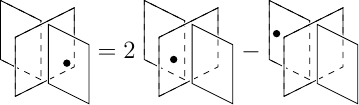}
\end{center}
\caption{Dot migration relation.
The right-hand side in this picture represents a linear
combination of two dot configurations on $S$.
}\label{fig:migration}
\end{figure}
\end{definition}

Given a dot configuration $d=\{d_1,\ldots,d_\ell\}$,
we can consider the element 
$f_1\wedge\cdots\wedge f_\ell\in\Lambda^*\col(S)$, where $f_i$
denotes the oversheet of $S$ that contains the dot $d_i$,
and where we are assuming that
the dots $d_1,\ldots,d_\ell$ are numbered in order of
decreasing time-coordinate.
It is easy to see that the assignment $d\mapsto f_1\wedge\cdots\wedge f_\ell$
takes the relations from Definition~\ref{def:DotS}
to the the defining relations of $\Lambda^*\col(S)$, and we thus obtain an identification
\[
\operatorname{Dot}(S)=\Lambda^*\col(S).
\]
Because of Lemma~\ref{lem:ColFH}, this identification also induces an identification
\begin{equation}\label{eqn:DotSH}
\operatorname{Dot}(S)=\Lambda^*H_1(\Sigma(F\cup\operatorname{id}_U);\ZZ).
\end{equation}

We now define a category \mbox{ $\mathcal{C}\mathit{ob}^4_{\bullet}$}
whose objects are link diagrams, and whose morphisms are given pairs $(S,d)$,
where $S$ is a broken surface diagram, and $d$ is an element of $\operatorname{Dot}(S)$.
Such pairs $(S,d)$ will be considered up to an equivalence relation, to be described later.
The composition in \mbox{$\mathcal{C}\mathit{ob}^4_{\bullet}$}
is given by
\begin{equation}\label{eqn:compositiondots}
(S,d)\circ(S',d'):=(S\circ S',(-1)^{|d||S'|}d\cup d'),
\end{equation}
where $d$ and $d'$ are configurations of dots, $|d|$ denotes the number of dots in the configuration $d$,
and $|S'|$ denotes the superdegree
of the link cobordism represented by $S'$. We extend definition \eqref{eqn:compositiondots}
bilinearly to the case where $d$ and $d'$ are linear combinations of dot configurations.

To describe the equivalence relation on pairs $(S,d)$, we consider
two broken surface diagrams $S$ and $S'$ that represent
link cobordisms $F$ and $F':=\phi(F)$, respectively, where 
$\phi$ is as in the proof of Lemma~\ref{lem:OKhFcinvariance}.
Specifically, we assume that $\phi\in D(S^3\times I)$ is
induced by one of the elementary movie moves described there,
and that this move is supported in a narrow strip $\mathbb{R}^3\times[r,s]\subset\mathbb{R}^3\times I$.
Then the proof of Lemma~\ref{lem:OKhFcinvariance}
and the identification from \eqref{eqn:DotSH} show that $\operatorname{Dot}(S)$
is linearly generated by dot configurations on $S\setminus(\mathbb{R}^2\times[r,s])$.
In particular, this implies that there is an isomorphism
\begin{equation}\label{eqn:mappsi}
\psi\colon \operatorname{Dot}(S)\longrightarrow\operatorname{Dot}(S')
\end{equation}
which is given by the identity map on dot configurations on 
$S\setminus(\mathbb{R}^2\times[r,s])=S'\setminus(\mathbb{R}^2\times[r,s])$, and
which corresponds to the isomorphism $\widetilde{\phi}_*$ from Subsection~\ref{subs:decorated}.

On pairs $(S,d)$ representing morphisms of \mbox{$\mathcal{C}\mathit{ob}^4_{\bullet}$},
we now impose the minimal equivalence relation such that
\[
(S,d)\sim(S',\psi(d))
\]
for all $S,S',\psi$ as in \eqref{eqn:mappsi} and for all $d\in\operatorname{Dot}(S)$.
Note that after imposing this relation, the second entry in $(S,d)$ can
be viewed as an element of a quotient set of $\operatorname{Dot}(S)$,
which corresponds to the quotient set from Remark~\ref{rem:groupaction}
under the identification \eqref{eqn:DotSH}.
The following is clear from the definitions:

\begin{lemma}\label{lem:DotLambda} The identification \eqref{eqn:DotSH} induces an equivalence of categories
$\mathcal{C}\mathit{ob}^4_{\bullet}\rightarrow\mathcal{C}\mathit{ob}^4_{\Lambda}$.
\end{lemma}

By precomposing the
odd Khovanov functor \mbox{$OKh\colon\mathcal{C}\mathit{ob}^4_{\Lambda}\rightarrow\operatorname{Ab}/\{\pm 1\}$} from Theorem~\ref{thm:functor}
with the equivalence from this lemma, we can turn it
into a functor on the category $\mathcal{C}\mathit{ob}^4_{\bullet}$.
In particular, the resulting functor on $\mathcal{C}\mathit{ob}^4_{\bullet}$ assigns a linear map
\[
OKh(S,d)\colon OKh(D_0)\longrightarrow OKh(D_1)
\]
to each pair $(S,d)$ as above. 
In the type Y setting, the map $OKh(S,d)$ assigned to 
a broken surface diagram $S$ and
a dot configuration $d$ is given by formula
\eqref{eqn:OKhFc} from Subsection~\ref{subs:oddmaps}, 
where each instance of the map $z_e$ that appears in that
formula corresponds to a dot in the configuration $d$,
and each $F_i$ is an elementary link
cobordism that occurs in a movie presentation corresponding to $S$.

If we define the
odd Khovanov cobordism maps as in \cite{MW2024}, then
the map $OKh(S,d)$ also has the following property:

\begin{lemma}\label{lem:verticaltube} In the type Y setting, the map $OKh(S,d)$ satisfies the relation
\begin{equation}\label{eqn:verticaltube}
\includeFig{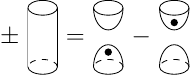}
\end{equation}
where the two terms on the right-hand side are two be viewed as a linear combination of two dot configurations
on the shown link cobordism.
\end{lemma}

\begin{proof}
We will prove this directly for the formal Khovanov bracket $\llbracket D\rrbracket$.
At any given vertex $D_\alpha$ of the resolution cube, the relation from \eqref{eqn:verticaltube}
holds up to possible signs by the vertical neck-cutting relation \eqref{eqn:vneck}.
To see that the signs are as in \eqref{eqn:verticaltube}, 
note that the component $(x_e)_\alpha$ of the dot chain map $x_e$ contains the sign $(-1)^{S(D,\alpha)}$,
and the same sign also occurs in the definition of the death chain map from \cite{MW2024}.
On the right-hand side of \eqref{eqn:verticaltube},
the signs that occur in $x_e$ and in the death chain map thus cancel.

Since the birth chain map from \cite{MW2024} contains no signs,
the only remaining sign is the sign $(-1)^{|F_{\leq i}|}$ that occurs
in the definition of $z_e$. This sign is independent
of the vertex $\alpha$, and it is different for the two terms on the right-hand side
of \eqref{eqn:verticaltube}, because a link cobordism formed by a death and a birth has odd superdegree.
\end{proof}

\subsection{Link Cobordisms with Horizontal Tubes}\label{subs:tubes}

The previous lemma has a variant for horizontal tubes:

\begin{lemma}\label{lem:horizontaltube}
In the type Y setting, the map $OKh(S,d)$ satisfies the relation
\[
\includeFig{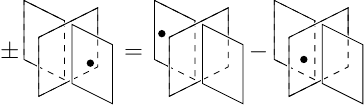}
\]
provided the tube on the left-hand side does not cross any sheets
in the broken surface diagram.
\end{lemma}
\begin{proof} This follows from the pictures below, in which the symbol $\doteq$
stands for equality up to sign:
\[
\includeFig{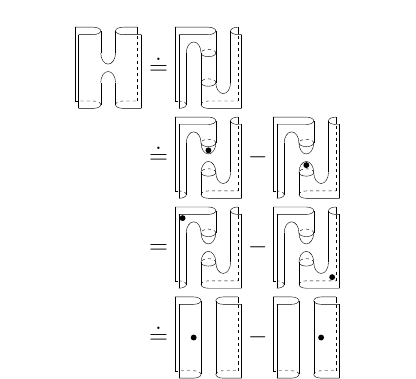}
\]
Note that the second equality follows from Lemma~\ref{lem:verticaltube},
while the other equalities follow from the well-definedness of $OKh(S,d)$.
\end{proof}

Using this lemma, we can now prove the following main result,
in which $\alpha$ and $F_\alpha$ are as in the introduction.

\TheoremTubeAction*

\begin{proof} This follows because the map $OKh(F,c)$
from Subsection~\ref{subs:oddmaps} satisfies the relations below:

\begin{center}
\includeFig{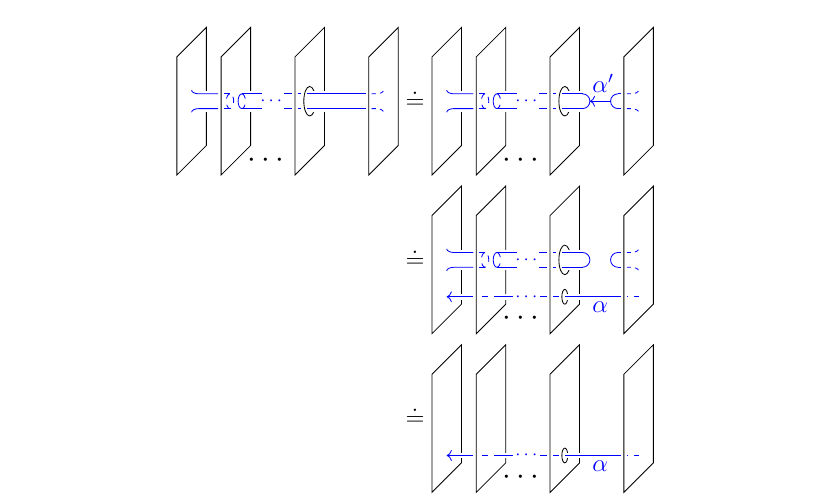}
\end{center}

Here, link cobordisms are depicted by their broken surface diagrams,
and a link cobordism $F$ with a horizontal arc $\alpha$
represents a morphism $(F,c)=(F,[\widehat{\alpha}])$
in the category \mbox{$\mathcal{C}\mathit{ob}^4_\Lambda$}.
The~first equation
follows from Lemma~\ref{lem:horizontaltube}
together with Remark~\ref{rem:dotdifference},
and the other two equations follow from the well-definedness of the map $OKh(F,c)$.
\end{proof}

\subsection{Link Cobordisms with Horizontal Tori}
As a useful application of
Theorem~\ref{thm:tubeaction},
we will now deduce a result about
the map induced by a toroidal component in a link cobordism.
Specifically, consider a link $L\subset\mathbb{R}^3$ and a simple closed curve
$\lambda\subset\mathbb{R}^3\setminus L$
in the complement of $L$. Let $V\subset(\mathbb{R}^3\times I)\setminus\operatorname{id}_L$
denote a $3$-dimensional thickening of the curve $\lambda\times\{1/2\}$, and let
\[T\subset(\mathbb{R}^3\times I)\setminus\operatorname{id}_L\]
denote the torus
$T:=\partial V$. Then the map induced by $\operatorname{id}_L\cup T$
is given as follows:

\begin{proposition}\label{prop:torus}
We have
\[
OKh(\operatorname{id}_L\cup T)=\begin{cases}
0 &\mbox{if $lk(L,\gamma)$ is even,}\\
\pm 2\operatorname{id} &\mbox{if $lk(L,\gamma)$ is odd.}
\end{cases}
\]
\end{proposition}

\begin{proof}
We will regard the torus $T$ as an unknotted $2$-sphere $S$ with a tube $t$ attached to it,
where the $2$-sphere $S$ is centered at a point on the loop $\lambda\times\{1/2\}$, and the tube $t$
runs along a subarc $\alpha$ of this loop with $\partial\alpha\subset S$.
Let $e\subset\mathbb{R}^3\times\{1/2\}$ denote the equator of $S$, and regard it as an unknot in the complement of $L$.
Moreover, assume $\partial\alpha\subset e$, and let $F_\alpha$ denote the
identity cobordism of $L\cup e$ with the horizontal tube $t$ attached to it.
By Theorems~\ref{thm:arcaction} and \ref{thm:tubeaction}, the map $\pm OKh(F_\alpha)$
is given by
\[
-x_e+2x_{a_1}-2x_{a_2}+\ldots+(-1)^{r-1}2x_{a_r}+(-1)^rx_e,
\]
or by
\[
-x_e+2x_{b_1}-2x_{b_2}+\ldots+(-1)^{s-1}2x_{b_s}+(-1)^sx_e,
\]
where $a_1,\ldots,a_r$ are the edges of $L$ that $\lambda$ overcrosses,
and $b_1,\ldots,b_s$ are the edges that it undercrosses.
Now note that  $\operatorname{id}_L\cup T$ can be obtained from $F_\alpha$
by capping off the copies of $e$ that appear in the top and the bottom boundary of $F_\alpha$.
The above formulas for the map $OKh(F_\alpha)$ therefore imply
\[
OKh(\operatorname{id}_L\cup T)=\pm(-1+(-1)^r)\operatorname{id}
\]
or
\[
OKh(\operatorname{id}_L\cup T)=\pm(-1+(-1)^s)\operatorname{id},
\]
where we are using that capping of the identity cobordism of $e$ yields the
$2$-sphere $S$, and a $2$-sphere evaluates to $0$ in the odd dotted Putyra
category if it contains no dots, and to $\pm 1$ if it contains a single dot
(as in the proof of Lemma~\ref{lem:verticaltube}, we can
ignore the sign that occurs in $x_e$ because it
cancels with the sign in the map induced by the death in $S$).
The corollary now follows because $r$ and $s$ have the
same parity as the linking number $lk(L,\gamma)$.
\end{proof}

\section{Odd Khovanov-Jacobsson Numbers}\label{s:okjnumbers}

In \cite{Ta2005}, Tanaka---working with even Khovanov homology---built off the recent functoriality results of Jacobsson, Khovanov, and Bar-Natan \cite{Ja2004,Kh2005,Bn2005} to define an invariant for smooth surface-knots realized as link cobordisms.
Tanaka's invariant extends to dotted $2$-knots, but turns out to be uninteresting, and evaluates
to $1$ on every $2$-knot decorated by a single dot \cite{Ra2005,Ta2005}.
Using the functoriality of odd Khovanov homolog up to sign, we can define
a similar $2$-knot invariant.
Furthermore the module structure on odd Khovanov homology provides a tool to compute this invariant.
Let $F$ be a 2-knot and $S$ be a the broken surface diagram of $K$ with $d\in \operatorname{Dot}(S)$ consisting of a single dot.  The odd Khovanov map $OKh(S,d)$ provides an endomorphism of $OKh(\emptyset)=\ZZ$, and we define $n(S,d)\in \ZZ$ as the image of $1\in\ZZ$ under this endomorphism.

\begin{proposition}
	$n(F)\defeq\left|n(S,d)\right|$ is an $\NN$-valued invariant of 2-knots.
\end{proposition}

\begin{proof}
	To show that $n(F)$ is an invarient we must show that it is unaffected by where the dot is placed, and the choice of surface diagram $S$.
	We will start by fixing a broken surface diagram.  
	We must show that for $d,d'\in\text{Dot}(S)$ both consisting of single dots, $n(S,d)=n(S,d')$.  
	We will provide two different arguments for proving this fact, as together they highlight how the machinery we have developed allows one to seamlessly work from an algebraic or topological perspective. Note that there is a third argument, which we omit,
	and which is based on identifying ambient isotopy classes of $2$-knots with ambient isotopy classes of capped-off slice disks.
	
	\vspace{1em}\textit{Argument 1}\hspace{1em}
	Our goal is to attach a tube to $F$ so that the map $OKh(F_\alpha)$ induced by the resulting surface $F_\alpha$
	agrees with the map $\pm OKh(S,d-d')$. Up to this point, we have only shown how to do this when the dots $d$ and $d'$
	occur at the same time-coordinate (see Theorem~\ref{thm:tubeaction} and Remark~\ref{rem:dotdifference}).
	If the dots $d$ and $d'$ occur at different time-coordinates $t$ and $t'$, let $\alpha$ be an arc in $\mathbb{R}^3\times I$
	from $d$ to $d'$, such that $\alpha$ consists of a vertical piece $\{p\}\times[t,t']$ contained in the unbounded component of $S$,
	and of horizontal pieces in $\mathbb{R}^3\times\{t\}$ and $\mathbb{R}^3\times\{t'\}$ that connect $d$ and $d'$ to the point $p$,
	respectively. Assume that the horizontal pieces overcross any intermediate sheets of $F$, and let $F_\alpha$
	be the undotted surface obtained from $F$ by attaching a tube $t_\alpha$ running along the arc $\alpha$.
	As the tube $t_\alpha$ has a part that rises vertically, we can apply the vertical neck-cutting relation \eqref{eqn:verticaltube}
	from Lemma~\ref{lem:verticaltube}, where we are assuming that we are working in the type Y setting.
	We can then slide the dots along the severed tubes and onto the surface of $F$ near where $t_\alpha$ was attached, finally retracting the severed tubes back into $F$, leaving us with $F$ decorated by $d-d'\in\text{Dot}(S)$.
	This shows $OKh(F_\alpha)=\pm OKh(S,d-d')$, as desired, and since $F_\alpha$ is an undotted nonempty closed
	surface, Lemma~\ref{lem:cobproperties}\ref{propertyA} implies $OKh(F_\alpha)=0$, and thus $0=n(F_\alpha,0)=n(F,d)-n(F,d')$ and $n(F,d)=n(F,d')$.

	\vspace{1em}\textit{Argument 2}\hspace{1em}

	The assignment $a\mapsto n(S,d_a)$ where $d_a$ denotes a dot placed on the oversheet $a$
	extends to a homomorphism $\nu\colon\col(S)\rightarrow\ZZ$.
	As $S$ is a broken surface diagram of a 2-knot, the abelian group $\col(S)\cong H_1(\Sigma(F);\ZZ)\oplus\ZZ$ has rank 1,
	by Remarks~\ref{rem:HSigmaFone} and \ref{rem:HSigmaFtwo}.  
	Therefore, the only possible homomorphisms from $\col(S)$ to $\ZZ$ are scalar multiples of the obvious homomorphisms sending each oversheet $a\in\col(S)$ to $1\in\ZZ$. Hence it is necessary that $\nu(a)=\nu(a')$ for all oversheets $a,a'\in\col(S)$.
	
	The final requirement that $n(S,d)$ is invariant, up to sign, under changes in the broken surface diagram representing $F$ is a direct application of Theorem~\ref{thm:functor}.
\end{proof}

While for the even theory the Khovanov-Jacobsson number for a dotted $2$-knot $F$ is always equal to $1$,
in the odd setting the invariant $n(F)$ is nontrivial, and in \cite{MW2024} we conjectured the following:

\begin{conjecture}
	For every smooth 2-knot,  $n(F)$ agrees with $|H_1(\Sigma(F);\ZZ)|$.
\end{conjecture}

In \cite{MW2024} we announced that we would include a proof of this conjecture for ribbon 2-knots in this paper, and a proof of the conjecture for even-twist spun knots in our upcoming paper \cite{MW2026b}.  
More recently the conjecture was proven true in general by Spyropoulos, Vidyarthi, and Zhang in \cite{SVZ2026}.  
However, our proofs in the stated special cases
are of independent value because they are purely combinatorial,
while the proof in \cite{SVZ2026} is analytical.

\subsection{Ribbon 2-Knots}

Our conjecture does hold particularly for ribbon 2-knots:

\TheoremRibbonKnot*

\begin{proof}
	Let $F$ be a smoothly embedded ribbon 2-knot in Morse position, which we can consider as a link cobordism from the empty-link to itself:  $$F:\emptyset\rightarrow\emptyset$$   
	Ribbon 2-knots can be constructed by taking an $m+1$ trivial 2-link---we will denote these components $S_0,S_1,\dots,S_m$---and connecting each $S_i$ for $i>0$ to $S_0$ by a tube denoted $t_i$. 
	For a ribbon 2-knot we can---and will---impose the following restrictions on its movie presentation: first---chronologically---$S_0$ is born, followed by $S_1$ all through $S_m$, then the tubes occur
	sequentially, before each $S_i$ is capped off now in reverse order ending with $S_0$ (see Figure~\ref{fig:twoknot}).
	
	\begin{figure}[H]
		\centering
		\includeFig{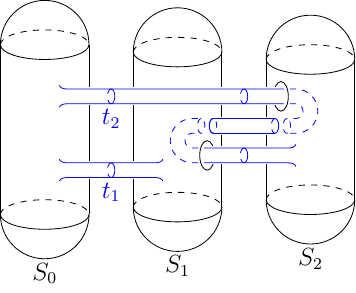}
		\caption{A ribbon 2-knot \label{fig:twoknot}}
	\end{figure}
	
	We now consider the cobordism 
	$$\dot F:U\rightarrow U$$ 
	from the unknot to itself produced by puncturing $F$ at the north and south poles of $S_0$ or alternatively begining the movie of $F$ directly after the birth of $S_0$ and ending directly before its death.
	On reduced odd Khovanov homology $\dot F$ induces the integer endomorphism
	$$f\defeq\overline{OKh}\bigl(\dot{F}\bigr):
	\underset{\substack{\rotatebox[origin=cc] {-90}{$ \cong $}\\\ZZ}}{\overline{OKh}(U)}
	\rightarrow
	\underset{\substack{\rotatebox[origin=cc] {-90}{$ \cong $}\\\ZZ}}{\overline{OKh}(U)}$$
		
	Let $L=L_0\cup L_1\cup\dots\cup L_m$ denote the $m+1$ component unlink living in a frame of the movie of $\dot F$
	that occurs after the birth of $S_m$, and before the movie of the first tube $t_1$ begins, so that each $L_i$ is a slice of the sphere $S_i$.
	
	\begin{figure}[H]
		\centering
		\includeFig{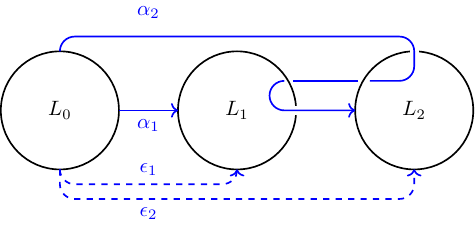}
		\caption{The link $L$, $\alpha$-arcs, and $\epsilon$-arcs for the 2-knot in Figure~\ref{fig:twoknot}.}\label{fig:linkL}
	\end{figure}
	
	On this link we can include arcs $\alpha_1,\dots,\alpha_m$ corresponding with the tubes $t_1,\dots,t_m$, so that for each $0<i\leq m$, the arc $\alpha_i$ connects $L_0$ to $L_i$ (see Figure~\ref{fig:linkL}).
	By Theorem~\ref{thm:tubeaction}, the portion of $f$ induced by the tubes in $\dot{F}$ can be expressed by the actions of the $\alpha$-arcs, giving $f$ the following overall structure:
	$$f=
	(\text{deaths of $S_1$ through $S_m$})\circ
	\left(\left[\widehat{\alpha}_1\right]\wedge\cdots\wedge\left[\widehat{\alpha}_m\right]\right)\circ
	(\text{births of $S_1$ through $S_m$}).$$
	Also consider the trivial arcs $\epsilon_1,\ldots,\epsilon_m$ where each $\epsilon_i$ connects $L_0$ to $L_i$ but does not link with $L_j$ for $0<j\neq i$.
	The homology classes $\left[\widehat{\epsilon}_1\right],\ldots,\left[\widehat{\epsilon}_m\right]$ form a basis for $H_1(\Sigma(L);\ZZ)\cong\ZZ^m$.
	Let $a_{ij}$ denote the scalars such that $\left[\widehat{\alpha}_i\right]=\sum_{j=1}^ma_{ij}\left[\widehat{\epsilon}_j\right]$, and let $A\defeq [a_{ij}]$.  
	The action induced by the tubes can be decomposed in the following manner
	
	\begin{equation}\nonumber
	\begin{split}
	\left(\left[\widehat{\alpha}_1\right]\wedge\cdots\wedge\left[\widehat{\alpha}_m\right]\right) 
	& =\det(A)\left(\left[\widehat{\epsilon}_1\right]\wedge\cdots\wedge\left[\widehat{\epsilon}_m\right]\right)\\
	& =\det(A)\left((x_1-x_0)\circ\cdots\circ(x_m-x_0)\right)\\
	& =\det(A)(x_1\circ x_2\circ\cdots\circ x_m)+\det(A)\sum(\text{remaining terms})\\
	\end{split}
	\end{equation}
	
	In each of the remaining terms, there will be at least one $i>0$ such that $-x_0$ replaces $x_i$ in the composition $x_1\circ x_2\circ\cdots\circ x_m$.  
	When we consider the remainder of $f$, which is induced by capping off the spheres $S_1,\ldots,S_m$ below and above,
this means that the resulting cobordism will contain the sphere $S_i$ undotted and unknotted from all other spheres,
	thus killing the entire term (since an isolated undotted sphere evaluates to zero).	
	The surviving term $x_1\circ x_2\circ\cdots\circ x_m$, capped off above and below, now consists of the identity cobordism 
	of $L_0$ together with each sphere $S_i$ for $i>0$ singly dotted.  
	We can remove the singly dotted spheres by using relation \eqref{eqn:DottedSphere},
	leaving us with the identity cobordism. Thus $f=\pm\det(A)\text{id}$, and it is easy to see that this implies $n(F)=|\det(A)|$.

	To complete the proof, we will now show $|\det(A)|=|H_1(\Sigma(F);\ZZ)|$.  
	Suppose that each tube $t_i$ is presented by a movie where a portion of $L_0$ knots through the circles $L_0,\ldots,L_m$
 via Reidemeister II moves, then merges with $L_i$, next splits from $L_i$, and finally retracts back along the same path undoing all the Reidemeister II moves. Note that no births or deaths occur in this movie presentation of $t_i$.
	Now consider a movie presentation for $F$ in which the merges in the $t_1,\ldots,t_m$ precede all of the splits.
	Thus, this movie presentation begins with the births in $S_0,\dots,S_m$ and then proceeds
	to the Reidemeister II moves and the merges in the tubes $t_1,\dots,t_m$ before running the entire process in reverse,
	starting with the splits in $t_1,\ldots,t_m$.
	The intermediate frame of the movie---between the final merge and the first split---consists of a knot $K$, and the cobordism $F$ naturally decomposes into a composition of cobordisms $C,\overline{C}\subset\RR^3\times I$ such that $F=\overline{C}\circ C$. 
	The cobordism $C:\emptyset\rightarrow K$ is a
	 ribbon disk for $K$, and 
	the cobordism $\overline{C}:K\rightarrow\emptyset$ is the mirror image of $C$ across the hyperplane $\RR^3\times\left\{\frac{1}{2}\right\}$.
	
	The branched double-cover of $S^4$ along $F$, denoted $\Sigma(S^4,F)$, decomposes into $\Sigma(B^4,C)$ and $\Sigma(B^4,\overline{C})$ connected along $\Sigma(S^3,K)$.
	In the Mayer-Vietoris sequence
	\[
	H_1(\Sigma(S^3,K)) \stackrel{\partial_1}{\longrightarrow}
	H_1(\Sigma(B^4,C))\oplus H_1(\Sigma(B^4,\overline{C})) \stackrel{\partial_2}{\longrightarrow}
	H_1(\Sigma(S^4,F))\longrightarrow 0
	\]
	for reduced singular homology,
	the map $\partial_2$ is surjective, and the map $\partial_1$
	is given by the $2\times 1$ matrix with components $\iota$ and $-\overline{\iota}$, where $\iota$ and $\overline{\iota}$
	denote the obvious inclusion-induced maps.
	As $C$ is a ribbon disk,
	all the generators of $H_1(\Sigma(B^4,C))$ appear in $H_1(\Sigma(S^3,K))$  (see \cite{OS2022}), and thus the maps $\iota,\overline{\iota}$ are surjective.
	Modding out by the image of $\partial_1$ therefore identifies like elements of $H_1(\Sigma(B^4,C))$ and $H_1(\Sigma(B^4,\overline{C}))$ resulting in a quotient group isomorphic to $H_1(\Sigma(B^4,C))$. Hence we obtain
	\[
	H_1(\Sigma(S^4,F))\cong\frac{H_1(\Sigma(B^4,C))\oplus H_1(\Sigma(B^4,\overline{C}))}{\operatorname{im}\partial_1}\cong H_1(\Sigma(B^4,C)).
	\]
	
	The branched double-cover of the four-ball along $C$ admits the following handle decomposition consisting of one 0-handle, $m$ 1-handles, and $m$ 2-handles.  The $0$-handle is a branched double-cover of a four-ball, branched along a disk corresponding to the birth of $L_0$,
	and
	the 1-handles correspond to the births of $L_1,L_2,\dots,L_m$. Note that the co-attaching spheres of these 1-handles are dual to the curves $\widehat{\epsilon}_1,\dots,\widehat{\epsilon}_m$.
	The 2-handles correspond to the $m$ merge saddles that appear in the tubes $t_1,t_2,\dots,t_m$, and are attached along the curves $\widehat{\alpha}_1,\dots,\widehat{\alpha}_m$ (cf. \cite[Theorem~7.1]{OS2022}).
	From this handle decomposition, it follows that $A$ is a presentation matrix for $H_1(\Sigma(B^4,C))$, and thus 
$|\det(A)|=|H_1(\Sigma(B^4,C))=|H_1(\Sigma(S^4,F))|=|H_1(\Sigma(S^3\times I,F))|$.
\end{proof}

\subsection{Odd Levine-Zemke Theorem}

Using the dotted even Khovanov homology along with canonical movie presentation of ribbon concordances, Levine and Zemke \cite{LZ2019} proved that ribbon concordances induce injective maps between the Khovanov homology of the boundary links.  
That is, for links $L_0$ and $L_1$ and a ribbon concordance $C:L_0\rightarrow L_1$, the induced map $$Kh(C):Kh(L_0)\longrightarrow Kh(L_1)$$ is injective.
Using a very similar argument to that of Theorem~\ref{thm:ribbonknot}, we can prove a semi-analogous result in the odd setting:

\TheoremRibbonConcordance*

\begin{proof}
	Let $C\subset\RR^3\times I$ be a ribbon concordance $C:L_0\rightarrow L_1$ such that projection onto $I$ restricts to a Morse function containing no death critical points.  
	Consider $F\defeq\overline{C}\circ C$ where $\overline{C}$ is the reflection of $C$ about $\mathbb{R}^3\times\{1/2\}$.  
	Following a similar strategy to the proof of Theorem~\ref{thm:ribbonknot}, we find that $$OKh\left(F\right):OKh(L_0)\longrightarrow OKh(L_0)$$ acts by some integer multiple $a\in\ZZ$ of the identity map on $OKh(L_0)$.
	Moreover, $a$ is odd because it is nonzero modulo $2$ by the result from \cite{LZ2019}.
	By the functoriality of odd Khovanov homology, $$OKh(F)=OKh\left(\overline{C}\circ C\right)=OKh\left(\overline{C}\right)\circ OKh(C)$$
	and thus $OKh\left(\overline{C}\right)\circ OKh(C)=a\cdot\text{id}$.
	In turn over $\Bbbk=\mathbb{Q}$ or $\Bbbk=\ZZ_{2^k}$, the integer $a$ is invertible,
	 meaning $OKh(C;\Bbbk)$ has left inverse $\frac{1}{a}OKh\left(\overline{C};\Bbbk\right)$ and must be injective.
\end{proof}

\begin{remark}
The cobordism $F=\overline{C}\circ C$ from this proof can be obtained
by attaching $m$ tubes to a cobordism of the form $F':=\operatorname{id}_{L_0}\cup S_1\cup\ldots\cup S_m$,
where $S_1,\ldots,S_m$ are unknotted $2$-spheres as in the proof of Theorem~\ref{thm:ribbonknot}.
In particular, the map $OKh(F)$ is equal to \[OKh(F)=OKh(F',[\widehat{\alpha}_1]\wedge\cdots\wedge[\widehat{\alpha}_m])\]
where $\alpha_1,\ldots,\alpha_m$ are the arcs of corresponding to these tubes.
To compute the map on the right-hand side of this equation, one can consider the map
\[
\Lambda^mH_1(\Sigma(F');\ZZ)\longrightarrow\operatorname{End}_{\ZZ}\bigl(OKh(L_0)\bigr)
\]
given by $c\mapsto OKh(F',c)$. Using that isolated undotted $2$-spheres evaluate to zero, one can
see that the latter map factors through $\Lambda^mH_1(\Sigma(F'),\partial_-\Sigma(F');\ZZ)$,
where $\partial_-\Sigma(F')$ denotes the lower boundary of $\Sigma(F')$.
An argument similar to the one used in the proof of Theorem~\ref{thm:ribbonknot} then
shows that the integer $a$ from the proof of Theorem~\ref{thm:ribbonconcordance} is equal to
\[
a=\pm|H_1(\Sigma(F),\partial_-\Sigma(F);\ZZ)|=\pm|H_1(\Sigma(C),\partial_-\Sigma(C);\ZZ)|.
\]
In particular, the conclusion of Theorem~\ref{thm:ribbonconcordance} remains true over integer coefficients
in the special case where $H_1(\Sigma(C),\partial_-\Sigma(C);\ZZ)=0$ (and hence $a=\pm 1$).
\end{remark}

\section{Odds and Ends}\label{s:oddend}

\subsection{Non-invariance in $S^3\times I$ Computations}\label{s:noninvariance}
In \cite{MW2024} we introduced a pair of cobordisms $F_1$ and $F_2$ which are ambient isotopic in $S^3\times I$ but not in $\RR^3\times I$ and induce different maps on odd Khovanov homology, proving that odd Khovanov homology is not functorial in $S^3\times I$.

\begin{figure}[H]
    \centering
    \includeFig{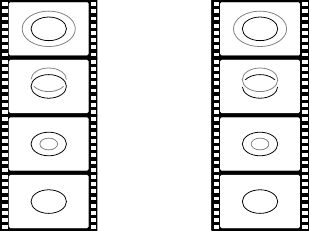}
    \caption{The cobordisms $F_1$ and $F_2$ on the left and right respectively}
\end{figure}

The computations were deferred until this paper as they are simpler and clearer using the machinery of dotted cobordisms.  
In all the following let $\doteq$ denote an equality up to sign.
A natural starting point would be to apply the neck cutting relation, but that would introduce complexity with the signs introduced by the dot in relation to the death.
We can instead precompose $F_1$ and $F_2$ with a birth---denoted $B$---as the resulting maps will compute the image of $1$, then do the same with a dotted birth---denoted $\dot{B}$---as the image of $1$ in these cobordisms is the image of $A$ under the original maps.

\begin{equation}
	\centering
    \includeFig{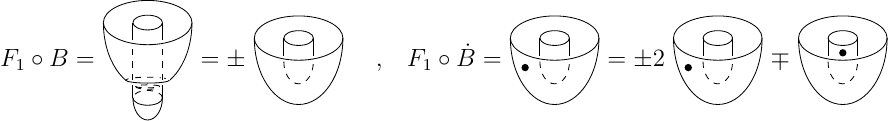}
\end{equation}

\begin{equation}
	\centering
    \includeFig{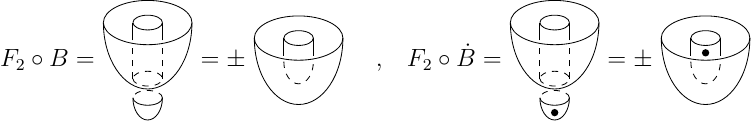}
\end{equation}

It is important to note that the signs incurred while retracting are not necessarily related between the two cobordisms.  
Let $U_2=x_1\cup x_2$ denote a 2-component unlink consisting of the inner component $x_1$ and an outer component $x_2$.  
Using notation from \cite{ORS2007}, the odd Khovanov homology groups of $U$ and $U_2$ are given by $OKh(U)=\operatorname{Span}\{1,x_1\}$ and $OKh(U_2)=\operatorname{Span}\{1,x_1,x_2,x_1\wedge x_2\}$.
Observe what each of the final maps does on the level of generators:

\begin{equation}
	\centering
	\includeFig{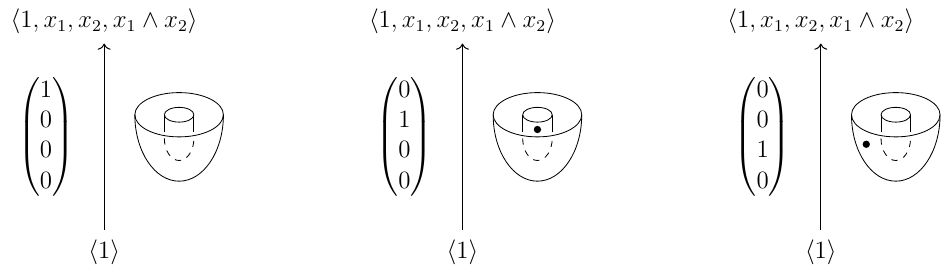}
\end{equation}

It follows that $OKh(F_1)\doteq\begin{pmatrix} \pm1 & 0 \\ 0 & -1 \\ 0 & 2 \\ 0 & 0 \end{pmatrix}$ and $OKh(F_2)\doteq\begin{pmatrix} -1 & 0 \\ 0 & 1 \\ 0 & 0 \\ 0 & 0 \end{pmatrix}$
which regardless of sign uncertainties differ by a non-zero map factor in where they map the generator $A$.

We extended $F_1$ and $F_2$ to an infinite family $\{F_n\}$ of cobordisms ambient isotopic in $S^3\times I$, non-isotopic in $\RR^3\times I$, which all induce different maps on odd Khovanov homology.
In order to facilitate computation of the map induced by $F_n$ we will represent the cobordisms by a tangle in the half plane which reproduces the cobordisms when swept around the boundary axis.  
Any dots or arcs in the diagrams attach to the tangle at precisely the point in the cross section and are not rotated about the axis.

\begin{equation}
	\centering
	\includeFig{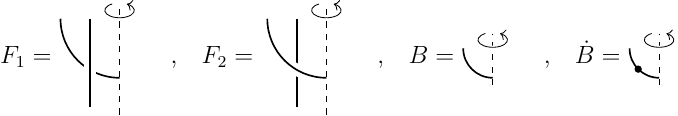}
\end{equation}

In order to construct $F_n$ we need the additional cobordism $R:U_2\rightarrow U_2$ where the outer circle crosses over the inner circle---becoming the inner circle---then crosses under the inner circle to return to its original position.\footnote{We could also add cobordisms to the family by composing with the reverse $\overline{R}$ in place of $R$.}

\begin{equation}
	\centering
	\includeFig{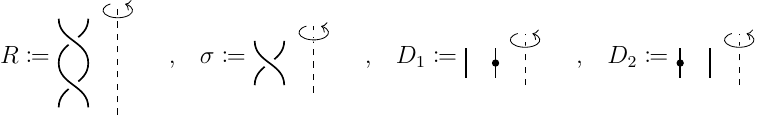}
\end{equation}

The cobordism $F_n:U\rightarrow U_2$ is defined by $F_n\defeq R^{n-1}\circ F_1$ where the power is the repeated composition of the cobordism $R$.\footnote{Note that the original $F_2$ is isotopic to the $F_2$ we just defined}
Our relations for passing dots through a crossing can be summed up by the matrix $\sqrt{M}:=\begin{pmatrix}0&-1\\1&2\end{pmatrix}$ measuring how dots on $x_1$ and $x_2$---in that order---precomposed with $\sigma$, relate to those dots postcomposed with $\sigma$.
For the full cobordism $R$, we have $M=\begin{pmatrix}-1&-2\\2&3\end{pmatrix}$, whose powers have the form $M^m=\begin{pmatrix}-2m+1&-2m\\2m&2m+1\end{pmatrix}$.  We will use the same strategy from before precomposing with both $B$ and $\dot{B}$.

\begin{equation}
	\centering
	\includeFig{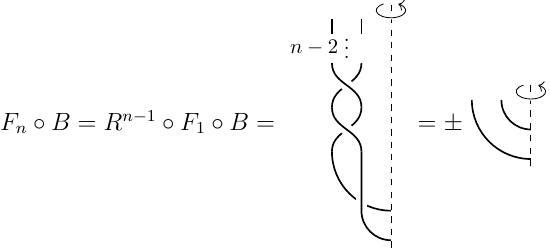}
\end{equation}

\begin{equation}
	\centering
	\includeFig{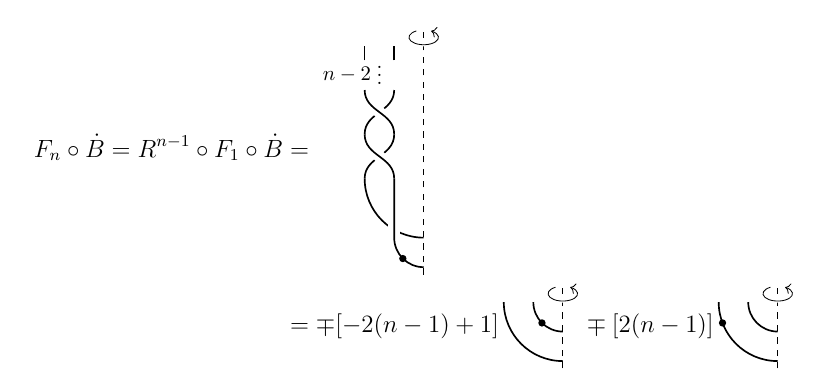}
\end{equation}

Therefore for $F_n=R^{n-1}\circ F_1$ the following map is induced
$OKh(F_n)\doteq\begin{pmatrix} 1 & 0 \\ 0 & 2n-3 \\ 0 & -2n+2 \\ 0 & 0 \end{pmatrix}$

\subsection{Two Approaches for Eliminating Signs}\label{subs:modifieddots}

Throughout this paper, we have been working under the assumption
that the dot map $x_e$ is defined as in Subsection~\ref{subs:DotChainMaps}, and that the map $OKh(F)$
assigned to a link cobordism $F$ is defined as in \cite{Pu2015,MW2024}.
The signs $(-1)^{|c||F'|}$ and $(-1)^{|c||F|}$ that appear in \eqref{eqn:composition} and \eqref{eqn:intertwine} are directly tied to these conventions.
In fact, these signs arise because the maps $x_e$ supercommute (in the sense of Lemma~\ref{lem:xcommute}) with
the maps induced by link cobordisms. We will now describe two independent
approaches for making the dot maps commute properly with the cobordism maps, thereby eliminating the need for signs in \eqref{eqn:composition} and \eqref{eqn:intertwine}.

The first approach is to redefine the dot maps by replacing $x_e$ by $x_e\circ f$,
where $f$ denotes the automorphism of odd Khovanov homology that acts as the identity
on generators of even superdegree, and as minus the identity on generators of odd
superdegree. By construction, the automorphism $f$ commutes with all maps of even superdegree, and anticommutes
with all maps of odd superdegree. 
Since the superdegree of the map $OKh(F)$ coincides with the superdegree of $F$,
Lemma~\ref{lem:xcommute} thus implies
that the maps $x_e\circ f$ properly commute with maps induced by link cobordisms.
The second approach is to leave the dot maps $x_e$ unchanged, but to modify
the link cobordism maps by replacing $OKh(F)$ by $OKh(F)\circ f$
whenever $F$ has odd superdegree. 
In view of the properties of $f$, it is easy to see that
the resulting modified cobordism maps
commute properly with the dot maps $x_e$, and that they are compatible with composition of link cobordisms
up to a possible overall sign.

\begin{remark}
It turns out that the map $x_e\circ f$ from the first approach
has a simple description on the level of resolutions.
Namely, let $D_\alpha$ be a resolution of the given link diagram,
and let $a_i$ be the component of $D_\alpha$ that meets the edge $e$.
Then the map $x_e\circ f$ acts on $\mathcal{F}_{odd}(D_\alpha)\langle s(\alpha)\rangle$
by wedge multiplication from the right by $a_i$.
To see this, let $a_{i_1},\ldots,a_{i_k}$ be connected components of $D_\alpha$.
Then the generator $a_{i_1}\wedge\cdots\wedge a_{i_k}\in\mathcal{F}_{odd}(D_\alpha)\langle s(\alpha)\rangle$
has superdegree $k+s(\alpha)$, and
the map $x_e$ acts on this generator 
wedge multiplication from the left by $(-1)^{S(D,\alpha)}a_i$.
Hence $x_e\circ f$ sends this generator to
\[
(-1)^{S(D,\alpha)}(-1)^{k+s(\alpha)}a_i\wedge\left( a_{i_1}\wedge\cdots\wedge a_{i_k}\right)=\left(a_{i_1}\wedge\cdots\wedge a_{i_k}\right)\wedge a_i,
\]
where this equation holds because
$s(\alpha)$ is the modulo $2$ reduction of $S(D,\alpha)$.
\end{remark}

\subsection{Bimodules and a Bifunctor on Undecorated Link Cobordisms}
Suppose we have used one of the two approaches from the previous subsection to eliminate the sign in \eqref{eqn:intertwine}. Then the map
\[
OKh(F,-)\colon\Lambda^*H_1(\Sigma(F\cup\operatorname{id}_U);\ZZ)
\longrightarrow\operatorname{Hom}_{\ZZ}\bigl(OKh(L),OKh(L')\bigr)
\]
from Corollary~\ref{cor:OKhFcinvariance} becomes a map of $(\Lambda_{L'},\Lambda_L)$-bimodules,
where
\[
\Lambda_L:=\Lambda^*H_1(\Sigma(L\cup U);\mathbb{Z})\quad\mbox{and}\quad
\Lambda_{L'}:=\Lambda^*H_1(\Sigma(L'\cup U);\mathbb{Z}),\] and where the bimodule structure on $\Lambda_F:=\Lambda^*H_1(\Sigma(F\cup\operatorname{id}_U);\ZZ)$ is induced by the obvious algebra maps
$\Lambda_L\rightarrow\Lambda_F$ and $\Lambda_{L'}\rightarrow\Lambda_F$.
By adjunction, we can reinterpret the bimodule map $OKh(F,-)$ as a $\Lambda_{L'}$-linear map
\begin{equation}\label{eqn:OKhbifunctor}
\mathbf{OKh}(F)\colon\Lambda_F\otimes_{\Lambda_L}OKh(L)\longrightarrow OKh(L')
\end{equation}
given by $\mathbf{OKh}(F)(c\otimes h):=OKh(F,c)(h)$.
If $F\colon L\rightarrow L'$ and $F'\colon L'\rightarrow L''$ are
composable link cobordisms, we further have
\[
\Lambda_{F'\circ F}=\Lambda_{F'}\otimes_{\Lambda_{L'}}\Lambda_F.
\]
In fact, this follows from the proof of Lemma~\ref{lem:ColFH} and from the definition of $\col(F)$.
Using these observations, we can now repackage the odd Khovanov functor
\[OKh\colon\mathcal{C}\mathit{ob}^4_\Lambda\longrightarrow\operatorname{Ab}/\{\pm 1\}\]
from Theorem~\ref{thm:functor} as a bifunctor
\[
\mathbf{OKh}\colon\mathbf{Cob}^4\longrightarrow\mathbf{C}
\]
defined on a bicategory $\mathbf{Cob}^4$ with the following objects and morphisms:
\begin{itemize}
\item Objects: Links in $\mathbb{R}^3$ in general position.
\item 1-Morphisms: Link cobordisms in $\mathbb{R}^3\times I$ in general position.
\item 2-Morphisms from $F$ to $F'$: Diffeomorphims $\phi\in D(S^3\times I)$ such that $\phi(F)=F'$.
\end{itemize}
The target bicategory of $\mathbf{OKh}$ is defined as follows.
The objects are pairs $(A,M)$, where $A$ is an algebra, and $M$ is a left module over $A$.
The $1$-morphisms between two objects $(A,M)$ and $(B,N)$
are  pairs $(\Lambda,g)$,
where $\Lambda$ is a $(B,A)$-bimodule, and $g$ is a $B$-linear map
\[g\colon\Lambda\otimes_AM\longrightarrow N,\]
considered up to an overall sign.
Given two $1$-morphisms $(\Lambda,g)\colon(A,M)\rightarrow(B,N)$ and $(\Lambda',g')\colon(B,N)\rightarrow(C,O)$,
we define
\[
(\Lambda',g')\circ(\Lambda,g):=(\Lambda'\otimes_B\Lambda,g'\circ(\operatorname{id}_{\Lambda'}\otimes g)).
\]
Finally, the $2$-morphisms between two $1$-morphisms $(\Lambda,g),(\Lambda',g')\colon(A,M)\rightarrow(B,N)$
are given by bimodule
isomorphisms $\varphi\colon\Lambda\rightarrow\Lambda'$
such that \[g'\circ(\varphi\otimes\operatorname{id}_M)=\pm g.\]
The bifunctor $\mathbf{OKh}$ is now defined as follows:
\begin{itemize}
\item To a link $L$, it assigns the group $OKh(L)$, viewed as a module over $\Lambda_L$.
\item To a link cobordism $F$, it assigns the map $\mathbf{OKh}(F)$ from \eqref{eqn:OKhbifunctor}.
\item To a $2$-morphism $\phi$, it assigns the induced isomorphism $\widetilde{\phi}_*\colon\Lambda_F\rightarrow\Lambda_{F'}$.
\end{itemize}

\printbibliography

\end{document}